%% file: main.tex
\documentclass[a4paper,twoside,10pt]{article}

\usepackage[a4paper,left=3cm,right=3cm, top=3cm, bottom=3cm]{geometry}
\usepackage[latin1]{inputenc}
\usepackage{mathrsfs}
\usepackage{enumerate}
\usepackage{comment}
\usepackage{cancel}
\usepackage{graphicx}
\usepackage{epstopdf}
\usepackage{amsmath}
\usepackage{amsthm}
\usepackage{algorithm,algpseudocode}
\usepackage{amssymb}
\usepackage{hyperref}
\usepackage{stmaryrd}
\usepackage{subcaption}
\usepackage{float}
\usepackage{bigints}
\usepackage{cite}
\usepackage{color}
\usepackage[abs]{overpic}
\usepackage[font=footnotesize,labelfont=bf]{caption}
\usepackage{cases}
\usepackage[author={Lorenzo}]{pdfcomment}
\usepackage{soul}
\usepackage[table]{xcolor}
\usepackage{mathtools}

\restylefloat{table}
\theoremstyle{plain}
\newtheorem{thm}{Theorem}[section]
\newtheorem{cor}[thm]{Corollary}
\newtheorem{lem}[thm]{Lemma}
\newtheorem{prop}[thm]{Proposition}
\theoremstyle{definition}

\theoremstyle{remark}
\newtheorem{remark}[thm]{Remark}
\newtheorem{assumption}[thm]{Assumption}

\newcommand{\eremk}{\hbox{}\hfill\rule{0.8ex}{0.8ex}}
\newcommand{\h}{h} 

\newcommand{\htilde}{\widetilde \h} 
\newcommand{\hE}{\h_\E} 
\newcommand{\PItD}{\mathcal P_{ItD}} 
\newcommand{\T}{\mathcal T} 
\newcommand{\Th}{\T_\h} 
\newcommand{\Aprime}{\mathcal A'} 
\newcommand{\B}{\mathcal B} 
\newcommand{\Aprimek}{\Aprime_k} 
\newcommand{\Bk}{\B_k} 
\newcommand{\vtilde}{v^{ext}} 
\newcommand{\vext}{\vtilde} 
\newcommand{\lambdah}{\lambda_\h} 
\newcommand{\xih}{\xi_\h} 
\newcommand{\wext}{w^{ext}} 
\newcommand{\uext}{u^{ext}} 
\newcommand{\uhext}{\uext_\h} 
\newcommand{\f}{f} 
\newcommand{\ftilde}{\widetilde\f} 
\renewcommand{\k}{k} 
\newcommand{\kn}{k n} 
 
\renewcommand{\Re}{\mathbb {RE}} 
\renewcommand{\Im}{\mathbb {IM}} 
\newcommand{\m}{m} 
\newcommand{\mh}{\m_\h} 
\newcommand{\G}{G}
\newcommand{\Gk}{\G_{\k}}

\newcommand{\xbold}{\mathbf x} 
\newcommand{\ybold}{\mathbf y} 
\newcommand{\Omegap}{\Omega^{\text{ext}}} 
\newcommand{\nGamma}{\mathbf n_{\Gamma}} 
\newcommand{\V}{\mathcal V} 
\newcommand{\Vz}{\V_0} 
\newcommand{\Vtilde}{\widetilde \V} 
\newcommand{\Vk}{\V_k} 

\newcommand{\Vtildek}{\Vtilde_k} 
\newcommand{\K}{\mathcal K} 
\newcommand{\Kz}{\K_0} 
\newcommand{\Kprime}{\K'} 
\newcommand{\Kprimez}{\Kprime_0} 
\newcommand{\Ktilde}{\widetilde \K} 
\newcommand{\W}{\mathcal W} 
\newcommand{\Wz}{\W_0} 
\newcommand{\Kk}{\K_k} 
\newcommand{\Kprimek}{\Kprime_k} 
\newcommand{\Ktildek}{\Ktilde_k} 
\newcommand{\Wk}{\W_k} 
\newcommand{\p}{p} 
\newcommand{\Hpw}{H_{\operatorname{pw}}}

\newcommand{\gammaz}{\gamma_0} 
\newcommand{\gammazint}{\gammaz^{int}} 
\newcommand{\gammazext}{\gammaz^{ext}} 
\newcommand{\gammao}{\gamma_1} 
\newcommand{\gammaoint}{\gammao^{int}} 
\newcommand{\gammaoext}{\gammao^{ext}} 
\newcommand{\psim}{\psi_m} 
\newcommand{\psitilde}{\psi^{ext}} 
\newcommand{\psih}{\psi_\h} 
\newcommand{\psimh}{\psi_{m\h}} 
\newcommand{\psitildeh}{\psi_\h^{ext}} 
\newcommand{\n}{\mathbf n} 

\newcommand{\Rm}{R_\m}
\newcommand{\Rext}{R^{ext}}
\newcommand{\F}{F} 
\newcommand{\Fhat}{\widehat \F}
\newcommand{\uh}{u_h} 
\newcommand{\vh}{v_h} 

\newcommand{\vexth}{v_\h^{\text{ext}}} 
\newcommand{\Vh}{V_h} 
\newcommand{\Vtildeh}{\widetilde V_\h}
\newcommand{\Wh}{W_h} 
\newcommand{\Zh}{Z_h} 

\newcommand{\taun}{\Omega_\h} 
\newcommand{\Omegah}{\taun}
\newcommand{\tautilden}{\widetilde \Omega_h}
\newcommand{\tautildenE}{\widetilde \Omega_h^\E}

\newcommand{\cG}{c_G}
\newcommand{\rr}{r}
\renewcommand{\rm}{\rr_\m}
\newcommand{\rext}{\rr^{ext}}
\renewcommand{\div}{\operatorname{div}}
\newcommand{\Ad}{\nu}

\newcommand{\Pcal}{\mathcal P}
\newcommand{\Pcalo}{\Pcal_1}
\newcommand{\Pcalt}{\Pcal_2}
\newcommand{\Pcalohat}{\widehat{\Pcal}_1}
\newcommand{\Pcalthat}{\widehat{\Pcal}_2}

\renewcommand{\i}{\text{i}}
\newcommand{\sigmabold}{\boldsymbol \sigma}

\newcommand{\taubold}{\boldsymbol \tau}

\newcommand{\E}{K}
\newcommand{\Ehat}{\widehat \E}
\newcommand{\Etilde}{\widetilde \E}
\newcommand{\nE}{\mathbf n_{\E}}
\newcommand{\uhat}{\widehat u}
\newcommand{\uhhat}{\widehat u(\uh)}
\newcommand{\sigmaboldhat}{\widehat{\sigmabold}}
\newcommand{\sigmaboldhhat}{\widehat {\sigmabold}(\uh)}
\newcommand{\Fh}{\mathcal F_h}
\newcommand{\FhI}{\Fh^I}
\newcommand{\FhB}{\Fh^B}

\newcommand{\ahE}{a^\E_\h}

\newcommand{\Xh}{X_\h}
\newcommand{\bhGamma}{b_\h^\Gamma}
\newcommand{\DG}{\text{DG}(\Omega)}
\newcommand{\DGp}{\text{DG}^+(\Omega)}
\newcommand{\ccoer}{c_{\text{coer}}}
\newcommand{\cc}{c_c}

\newcommand{\xbf}{\mathbf x}
\newcommand{\afrak}{\mathfrak{a}}
\newcommand{\bfrak}{\mathfrak{b}}
\newcommand{\dfrak}{\mathfrak{d}}
\newcommand{\afrakz}{\mathfrak{a}_0}
\newcommand{\bfrakz}{\mathfrak{b}_0}
\newcommand{\dfrakz}{\mathfrak{d}_0}
\newcommand{\Ctrace}{C_{\text{trace}}}

\newcommand{\Pkar}{\mathcal P}
\newcommand{\nablah}{\nabla_\h}

\newcommand{\aaleph}{\aleph}

\newcommand{\epsilontilde}{\widetilde{\varepsilon}}
\newcommand{\ldc}{\{\!\!\{}
\newcommand{\rdc}{\}\!\!\}}
\newcommand{\cinvG}{c_{\text{inv}}^G}
\newcommand{\Gammah}{\Gamma_\h}
\newcommand{\PhiE}{\Phi_\E}

\newcommand{\boldalpha}{\boldsymbol \alpha}
\newcommand{\Nbb}{\mathbb N}
\newcommand{\Rbb}{\mathbb R}

\newcommand{\DGFEMBEM}{DGFEM-BEM{}}
\newcommand{\dG}{DGFEM}
\newcommand{\msf}{{\mathfrak h}}
\newcommand{\hEone}{\h_{\E_1}} 
\newcommand{\hEtwo}{\h_{\E_2}} 
\newcommand{\numax}{\Ad_{\text{max}}}
\newcommand{\numin}{\Ad_{\text{min}}}

\newcommand{\lefttriplenorm}{\ensuremath{\left| \! \left| \! \left|}}
\newcommand{\righttriplenorm}{\ensuremath{\right| \! \right| \! \right|}}
\newcommand{\energynorm}[1]{\lefttriplenorm #1 \righttriplenorm_{\DG}}
\newcommand{\energynormp}[1]{\lefttriplenorm #1 \righttriplenorm_{\DGp}}
\newcommand{\x}{x}
\newcommand{\xh}{x_h}
\newcommand{\yh}{y_h}

\newcommand{\tepsilon}{t}

\definecolor{ilariagreen}{rgb}{0 0.7 0.2}

\definecolor{ceyellow}{rgb}{0.7 0.4 0.5}

\newcommand{\vertiii}[1]{{\left\vert\kern-0.25ex\left\vert\kern-0.25ex\left\vert #1  \right\vert\kern-0.25ex\right\vert\kern-0.25ex\right\vert}}

\numberwithin{equation}{section}

\usepackage{tikz}
\usetikzlibrary{shapes}
\usepackage{pgfplots}
\usetikzlibrary{calc}
\usetikzlibrary{positioning}
\usetikzlibrary{arrows}
\usepgfplotslibrary{colorbrewer}

\pgfplotsset{compat=1.13}
\include{create_h_plot}

\begin{tiny}
\author{Christoph Erath \thanks{University College of Teacher Education Vorarlberg, 6800 Feldkirch, Austria\newline E-mail: {\tt christoph.erath@ph-vorarlberg.ac.at}} \and
Lorenzo Mascotto \thanks{Fakult\"at f\"ur Mathematik, Universit\"at Wien, 1090 Vienna, Austria\newline E-mail: {\tt lorenzo.mascotto@univie.ac.at, ilaria.perugia@univie.ac.at, alexander.rieder@univie.ac.at}} \and 
Jens M.~Melenk \thanks{Institut f\"ur Analysis und Scientific Computing, TU Wien, 1040 Vienna, Austria\newline E-mail: {\tt melenk@tuwien.ac.at}} \and 
Ilaria Perugia \footnotemark[2] 
\and 
Alexander Rieder \footnotemark[2]}
\end{tiny}

\title{\textbf{\Large{Mortar coupling of $hp$-discontinuous Galerkin and boundary element methods for the Helmholtz equation}}}

\begin{document}
\maketitle

\begin{abstract}
\noindent
We design and analyze a coupling of a discontinuous Galerkin finite element method with a boundary element method to solve the Helmholtz equation with variable coefficients in three dimensions. 
The coupling is realized with a mortar variable that is related to an impedance trace on a smooth interface.
The method obtained has a block structure with nonsingular subblocks.
We prove quasi-optimality of the $\h$- and $\p$-versions of the scheme, under a threshold condition on the approximability properties of the discrete spaces.
Amongst others, an essential tool in the analysis is a novel
discontinuous-to-continuous reconstruction operator on tetrahedral meshes with curved faces.

\medskip\noindent
\textbf{Keywords}: discontinuous Galerkin method; boundary element method; mortar coupling; Helmholtz equation; variable sound speed
\end{abstract}

\section{Introduction}
A natural habitat of wave propagation problems are unbounded domains. An important class of numerical methods for this setting is the coupling of a
volume-based method such as the finite element method (FEM) or one of its  variants for a finite, chosen computational domain $\Omega$ 
and a boundary element method (BEM) for its unbounded exterior $\Omegap:=\Rbb^3\setminus \overline{\Omega}$. In this paper, we study such a coupling technique 
for a time-harmonic  acoustic scattering problem modelled by the Helmholtz equation in $\Rbb^3$ and given by 
\begin{equation} \label{original:problem}
- \div (\nu \nabla u) - (\kn)^2 u = f \quad \mbox{ in $\Rbb^3$},
\end{equation}
where the coefficients $\nu$ and~$n$ are constant outside a sufficiently large ball, $k$ denotes the wave number, and~$\f$ is the right-hand side.
Our focus is on a strategy that couples a high order discontinuous Galerkin finite element method (DGFEM) in the computational domain~$\Omega$ with a BEM on $\Gamma:=\partial\Omega$ to account for~$\Omegap$.
We consider approximation spaces made of piecewise polynomial functions.

The present work is a continuation of the recent work \cite{FEMBEM:mortar}, where the coupling of a 
conforming, high order FEM with a BEM is presented and analyzed. The coupling there is reminiscent of the symmetric coupling proposed 
in \cite{costabel88a} and \cite{han90} for Poisson-type problems but uses an additional mortar variable that 
has the physical meaning of a Robin trace for incoming waves. A feature of the mortar-based coupling is that 
the resulting system has a block structure where the two blocks
corresponding to the volume and to the BEM unknowns, respectively,
are individually invertible. This allows for 
the use of existing discretization techniques  for these blocks. Other FEM-BEM coupling strategies for Helmholtz 
problems are possible and are discussed in \cite{FEMBEM:mortar}. In contrast to the conforming setting of \cite{FEMBEM:mortar},
our focus here is on a DGFEM for the discretization in~$\Omega$, since the DGFEM has proved to be a very versatile discretization technique that can accommodate very well, for example, 
high order discretizations. High order methods are particularly suited for wave propagation problems~\cite{Ainsworth2004,melenk-sauter10, melenk-sauter11, MelenkParsaniaSauter_generalDGHelmoltz, melenk-loehndorf11}.
Further well known advantages of DG discretizations include the ease to realize adaptivity and accommodate nonuniform polynomial degree distributions.
Moreover, DG formulations for Helmholtz problems in bounded domains have the potential to be   unconditionally well posed~\cite{GHP_PWDGFEM_hversion,FengWu}.
We refer, e.g., to~\cite{FengWu2009,FengWu,FengXing2013,Amara2012,MelenkParsaniaSauter_generalDGHelmoltz,DuXu2016,ZhaoParkChung2020} for polynomial-based DG methods for the
  Helmholtz problem,
  to~\cite{GriesmaierMonk2011,ChenLuXu2013,CuiZhang2014} for hybridized DG (HDG) methods, and to~\cite{Demkowicz2012,Gopalakrishnan2014} for
  discontinuous Petrov-Galerkin (DPG) methods.

For Poisson-type problems, couplings of several variants of the DGFEM with the BEM have been proposed and analyzed. The first analysis appears to be that of the symmetric coupling 
of the local DG (LDG) method with BEM in \cite{Gatica2006,Bustinza2008,Gatica2010}. Generalizations to nonsymmetric couplings have been proposed in \cite{Of2012} and 
analyzed in \cite{Heuer2015}. The closely related coupling of finite volume methods with BEM is analyzed in \cite{Erath:2010-phd,Erath:2012-1,Erath:2013-3}. 
A fairly general framework that uses mortar variables for coupling the DGFEM and the BEM can be found in \cite{cockburn-sayas12,cockburn2012}. 
In the limit~$k \rightarrow 0$, which is not the focus of the present work, our method for~\eqref{original:problem} has similarities with those of~\cite{Gatica2006, Erath:2012-1} for the Laplace equation.

On a technical side, a main difficulty of the analysis of couplings of DG with BEM arises from the mapping properties of BEM operators that do not 
easily accommodate the discontinuous traces of DG functions.  This is one of the reasons for using mortar variables for the coupling both for Poisson-type problems 
discussed above and the present Helmholtz equation.  In our analysis, we tackle this issue with a new discontinuous-to-continuous operator in Theorem~\ref{lemma:Karakashian-style}.
From the many possible DG variants, we opted for an interior penalty-like DG method to keep the presentation as simple as possible, although other DG discretizations could be analyzed with similar techniques.
Following the lead in \cite{FEMBEM:mortar}, we employ a form of symmetric coupling using all four BEM operators.  The sesquilinear forms are carefully designed 
to ensure both consistency and adjoint consistency. In particular, as compared to~\cite{FEMBEM:mortar}, the discretization
  of the coupling condition required us to introduce an additional (consistent) term in order to prove a discrete G{\aa}rding inequality.

\paragraph*{Notation.}
For bounded Lipschitz domains $D \subset {\mathbb R}^d$, $d\ge 1$, we introduce the following norms and spaces: 
For integers~$s \in \mathbb N_0$ and complex-valued functions~$v$, we define the norms $\|v\|^2_{s,D} := \sum_{\boldalpha \in \Nbb_0^3 \colon|\boldalpha| \leq s} \|D^{\boldalpha} v\|^2_{0,D}$
and the seminorms $|v|^2_{s,D} := \sum_{\boldalpha \in \Nbb_0^3 \colon|\boldalpha| = s} \|D^{\boldalpha} v\|^2_{0,D}$.
The Hilbert spaces $H^s(D)$ and~$H^s_0(D)$ are defined as the closure of~$C^\infty (D)$ and~$C^\infty_0(D)$ with respect to the norm~$\|\cdot\|_{s,D}$.
We further set~$H^0(D):= L^2(D)$. 
For a noninteger~$s>0$, the spaces~$H^s(D)$ and~$H^s_0(D)$ 
are defined by interpolation between $H^{\lfloor s\rfloor}(D)$
and~$H^{\lceil s\rceil}(D)$
and between~$H^{\lfloor s\rfloor}_0(D)$ and~$H^{\lceil s\rceil}_0(D)$,
respectively.
For~$s > 0$, the space~$H^{-s}(D)$ is defined as the dual of~$H^s_0(D)$ with norm 
\[
\|v\|_{-s,D} := 
\sup_{w \in H^s_0(D)} \frac{ | \langle v, w \rangle | }{\|w\|_{s,D}},
\]
where $\langle \cdot,\cdot \rangle$ denotes the duality pairing, which coincides with the $L^2$ inner product whenever both~$v,w \in L^2(D)$.
The inner product in~$H^s(D)$, denoted by $(\cdot,\cdot)_{s,D}$, is linear in the first argument and antilinear in the second argument.

For closed, connected, smooth 2-dimensional surfaces $\Gamma \subset {\mathbb R}^3$ and~$s \ge 0$,
we define the Sobolev spaces~$H^s(\Gamma)$ as follows.
Let~$\{ \varphi_n, \lambda_n  \}_{n \in \mathbb N_0}$ be a sequence of eigenpairs of the Laplace-Beltrami operator on~$\Gamma$, so that $\{\varphi_n \}_{n \in \mathbb N_0}$ is an orthonormal basis of $L^2(\Gamma)$.
For~$s \ge 0$, we define the norm $\|v\|^2_{s,\Gamma}:= \sum_{n} |v_n|^2 (1+\lambda_n)^s$,  where $v = \sum_{n} v_n \varphi_n$ is expanded in the basis $\{\varphi_n\}_{n \in {\mathbb N}_0}$.
Then, $H^s(\Gamma):=\{v\in L^2(\Gamma): \|v\|_{s,\Gamma}<+\infty\}$.
The $\|\cdot\|_{s,\Gamma}$ norm is equivalent to the one obtained by using local charts as described in~\cite{mclean2000strongly}; see also~\cite[Sec.~{5.4}]{nedelec01}.
The mapping $v \mapsto (v_n)_{n \in {\mathbb N}_0}$, with $v_n =
(v,\varphi_{n})_{L^2(\Gamma)}$, is an isometric isomorphism 
between~$H^s(\Gamma)$ and the sequence space $\{(v_n)_{n \in {\mathbb N}_0}\,|\, \sum_n |v_n|^2 (1+\lambda_n)^s <\infty\}$.
Negative order Sobolev spaces are defined by duality and equipped with the norm
\[
\|v\|_{-s,\Gamma} := 
\sup_{w \in H^s(\Gamma)} \frac{| \langle v, w \rangle | }{\|w\|_{s,\Gamma}}.
\]
Moreover, $H^{-s}(\Gamma)$ is equivalent to the space~$\{(v_n)_{n \in {\mathbb N}_0}\,|\, \sum_n |v_n|^2 (1 + \lambda_n)^{-s}<\infty\}$,
endowed with the norm $\| v\|^2_{-s,\Gamma} = \sum_{n} |v_n|^2(1+\lambda_n)^{-s}$.
When identifying the spaces~$H^s(\Gamma)$ and~$H^{-s}(\Gamma)$ with
sequence spaces as above, the duality pairing takes the form
\[
\langle v, w\rangle = \sum_n v_n \overline{w_n},
\]
and, for $s\in \mathbb R$, the inner product in $H^{s}(\Gamma)$ takes the form
\[
(v, w)_{s ,\Gamma} = \sum_{n \in \mathbb N_0} v_n \overline{w_n} (1+\lambda_n)^{s}.
\]
The seminorm~$|\cdot|_{\frac12,\Gamma}$ in~$H^{\frac12}(\Gamma)$ is defined by 
$|v|_{\frac12,\Gamma} =\inf_{c \in  {\mathbb C}}  \|v -
c\|_{\frac12,\Gamma}$.

Let~$s$ and~$s' \in \Rbb$.
For a bounded linear operator $K:H^s(\Gamma) \rightarrow H^{s'}(\Gamma)$,
the adjoint operator~$K^\ast:H^{-s'}(\Gamma) \rightarrow H^{-s}(\Gamma)$
is defined by $\langle v, K^\ast w \rangle = \langle Kv, w\rangle$, where
the duality pairings are between the appropriate
spaces. This also implies $\langle  w,K v\rangle = \langle K^\ast w, v\rangle$.

As we deal with the Helmholtz problem, we also introduce the following
$k$-weighted Sobolev norms for integers $s\ge 1$ on domains $D\subset \mathbb R^3$:
\[
\| v\|_{s,k,D}^2:=\sum_{\boldalpha \in \mathbb N_0^3 \colon|\boldalpha| \leq s} k^{2(s-|\boldalpha|)}\|D^\alpha v\|^2_{0,D}.
\]
Finally, given~$a$, $b\ge 0$, we write~$a \lesssim b$
and~$a \gtrsim b$ to indicate the existence 
of a positive constant~$c$, whose dependence is specified
at each occurrence, such that~$a \le c \, b$ and~$a \ge c\,b$, respectively.

\paragraph*{Outline of the paper.}
The mortar formulation of the three dimensional Helmholtz problem is detailed in Section~\ref{section:continuous:problem}; here, we also recall several properties of boundary integral operators.
We introduce the DGFEM-BEM mortar method in Section~\ref{section:dG}.
Such a discretization is characterized by a sesquilinear form
satisfying a G{\aa}rding inequality and continuity estimates, which
we prove in Sections~\ref{section:Garding} and~\ref{section:continuity}, respectively.
In Section~\ref{section:adjoint}, we  provide results for the adjoint problem.
Then we cope with the well posedness of the method and the $\h$- and $\p$-error analysis in Section~\ref{section:error-analysis}.
We present numerical results verifying the theoretical results in Section~\ref{section:numerical-results} and state some conclusions in Section~\ref{section:conclusions}.
Three appendices conclude the paper:
in the first one, we show a consistency property of the proposed \DGFEMBEM{} mortar coupling;
in the second one, we construct a discontinuous-to-continuous reconstruction operator on curvilinear meshes
with optimal $\h$- and~$\p$-stability and approximation properties, which is of independent interest;
in the third one, we prove quantitative error estimates that are explicit in $\h$ and~$\p$.

\section{Helmholtz model problem, boundary integral operators, and mortar coupling} \label{section:continuous:problem}

In this section, we describe the model problem, see Section~\ref{subsection:model:problem}, and present its continuous mortar formulation  in Section~\ref{subsection:mortar-coupling}.
The setting is the same as that of \cite[Secs.~2 and~3]{FEMBEM:mortar}.
In order to make this paper self-contained, we report here all the necessary elements, including the definitions and some properties of the boundary integral operators for the 3D Helmholtz problem;
see Section~\ref{subsection:boundary-integral-operators}.

\subsection{Helmholtz model problem}\label{subsection:model:problem}
Let~$\Omega \subset \mathbb R^3$ be a bounded domain with a connected $C^\infty$-smooth
boundary~$\Gamma$ and~$\nGamma$ be the outward pointing unit vector normal to~$\Gamma$.
Denote~$\Omegap:= \mathbb R^3 \setminus \overline{\Omega}$  and
let~$\k \in \mathbb R^+$, $k\ge k_0>0$, denote the wave number.

We assume that~$\Omegap$ is occupied by a homogeneous medium with
both the refractive index $n \in L^\infty(\mathbb{R}^3; \mathbb{C})$ and a scalar-valued positive diffusion coefficient~$\Ad \in C^\infty(\mathbb{R}^3; \mathbb{R})$ 
normalized to~$1$, while~$n$ and~$\Ad$ may vary within~$\Omega$.
That is, the supports of~$1-n$ and $1 - \nu$ are contained
in~$\Omega$ and $\Ad$ satisfies 
$0<\numin\le \Ad(\xbold)\le\numax<+\infty$. 
In addition, we assume that $|n(\xbold)|\ge c_0>0$ a.e. in~$\mathbb R^3$.
Given~$\f \in L^2(\Omega)$ with support contained in $\Omega$, we set
\[
\ftilde = \begin{cases}
\f & \text{in } \Omega,\\
0 & \text{in } \Omegap.\\
\end{cases}
\]
Under the above assumptions on $n$, $\Ad$, and $\ftilde$, there exists an open
  neighborhood $\mathcal N(\Gamma)$ of $\Gamma$ such that
  $n\equiv 1$,
  $\Ad\equiv 1$, and
  $\ftilde\equiv 0$ 
  in $\Omegap\cup \mathcal N(\Gamma)$.

We consider the Helmholtz problem: Find $u: \mathbb R^3 \rightarrow
\mathbb C$ such that
\begin{equation} \label{starting:problem}
\begin{cases}
  -\div (\Ad \nabla u ) - (\kn)^2 u = \ftilde & \text{in } \mathbb R^3,\\
  \displaystyle{\lim_{\vert \xbold \vert \rightarrow +\infty}\vert \xbold \vert}
    (\partial_{\vert\xbold\vert} u - \i\k u )=0. &
\end{cases}
\end{equation}
We rewrite problem~\eqref{starting:problem} as a transmission problem. To that end, we define the following jump operators.
For $v \in H^1(\mathbb R^3 \setminus \Gamma)$, we denote the Dirichlet traces of~$v_{|\Omega}$ and~$v_{|\Omegap}$ on~$\Gamma$ 
by~$\gammazint(v)$ and~$\gammazext(v)$, respectively.
The two Neumann traces $(\partial_{\nGamma} = \nGamma \cdot \nabla)$
on~$\Gamma$ of a piecewise smooth function $v$ are denoted by~$\gammaoint(\varphi)$ and~$\gammaoext(\varphi)$.
For sufficiently smooth functions $v$ defined in $\mathbb R^3 \setminus \Gamma$, we then define the jumps 
\[
\llbracket v \rrbracket_\Gamma = \gammazint(v) - \gammazext(v), \quad
\quad \quad \llbracket \partial_{\nGamma} u
\rrbracket_\Gamma := \gammaoint(v) - \gammaoext(v).
\]
With these jumps in hand, we reformulate~\eqref{starting:problem} as 
looking  for solutions~$u: \mathbb R^3 \rightarrow \mathbb C$ of the following transmission problem:
\begin{equation} \label{transmission:problem}
\begin{cases}
-\div (\Ad \nabla u ) - (\kn)^2 u = \f & \text{in } \Omega,\\
-\Delta u - \k^2 u = 0 & \text{in } \Omegap,\\
\llbracket u \rrbracket_\Gamma =0, \; \llbracket
\partial_{\nGamma} u
\rrbracket_\Gamma =0,\\
\displaystyle{\lim_{\vert \xbold \vert \rightarrow +\infty}\vert \xbold \vert}
    (\partial_{\vert\xbold\vert} u - \i\k u )=0. &
\end{cases}
\end{equation}
Here, we required the boundary~$\Gamma$ to be globally smooth, whereas in Section~\ref{section:dG} below we allow for a piecewise smooth~$\Gamma$.
The global smoothness assumption is needed to promote the regularity of the solution to
problem~\eqref{transmission:problem} below, while the piecewise smoothness assumption is enough for the design of the method.

\subsection{Boundary integral operators} \label{subsection:boundary-integral-operators}
The fundamental solution to the 3D Helmholtz problem is
\[
\Gk (\xbold, \ybold) = \frac{e^{\i\k \vert \xbold - \ybold \vert}}{4\pi \vert \xbold - \ybold\vert}. 
\]
Based on that, we define the single and double layer potentials as follows:
\begin{alignat*}{2}
&  \Vtildek \varphi (\xbold) = \int_\Gamma G_k (\xbold - \ybold) \varphi(\ybold) ds(\ybold)                                          && \quad \forall \xbold \in \mathbb R^3 \setminus \Gamma,\quad \forall \varphi \in H^{-\frac{1}{2}}(\Gamma),\\
&  \Ktildek \varphi (\xbold) = \int_\Gamma \partial_{\nGamma(\ybold)}G_k  (\xbold - \ybold) \varphi(\ybold) ds(\ybold)  && \quad \forall \xbold \in \mathbb R^3 \setminus \Gamma,\quad \forall \varphi \in H^{\frac{1}{2}}(\Gamma).
\end{alignat*}
Starting from the potentials~$\Vtildek$ and~$\Ktildek$, we introduce the four standard boundary integral operators for the Helmholtz operator.
Their properties are widely studied in the literature; see, e.g., \cite{SauterSchwab_BEMbook, steinbach_BEMbook, costabel, mclean2000strongly} and the references therein.
The properties mentioned below have also been summarized in~\cite{FEMBEM:mortar}.

\paragraph*{Single layer operator.}
Define~$\Vk: H^{-\frac{1}{2}}(\Gamma) \rightarrow H^{\frac{1}{2}}(\Gamma)$ as
\begin{equation} \label{single:layer:operator}
\Vk \varphi := \gammazint (\Vtildek \varphi) \qquad \forall \varphi \in H^{-\frac{1}{2}}(\Gamma).
\end{equation}
For $C^\infty$-smooth~$\Gamma$, the operator $\Vk$ extends to~$\Vk :
  H^{-1+s}(\Gamma)\rightarrow H^s(\Gamma)$ for all $s\in\mathbb R$.

\paragraph*{Double layer operator.}
Define~$\Kk: H^{\frac{1}{2}}(\Gamma) \rightarrow H^{\frac{1}{2}}(\Gamma)$ as
\[
\left( -\frac{1}{2} + \Kk  \right) \varphi := \gammazint ( \Ktildek \varphi) \qquad \forall \varphi \in H^{\frac{1}{2}}(\Gamma).
\]
For $C^\infty$-smooth~$\Gamma$, the operator $\Kk$ extends to~$\Kk : H^{s}(\Gamma)
\rightarrow H^s(\Gamma)$ for all $s\in\mathbb R$.

\paragraph*{Adjoint double layer operator.}
Define~$\Kprimek: H^{-\frac{1}{2}}(\Gamma) \rightarrow H^{-\frac{1}{2}}(\Gamma)$ as
\[
\left( \frac{1}{2} + \Kprimek  \right) \varphi := \gammaoint ( \Vtildek \varphi) \qquad \forall \varphi \in H^{-\frac{1}{2}}(\Gamma).
\]
For $C^\infty$-smooth~$\Gamma$, the operator $\Kprimek$ extends to~$\Kprimek :
H^{-s}(\Gamma) \rightarrow H^{-s}(\Gamma)$ for all $s\in\mathbb R$.

\paragraph*{Hypersingular boundary integral operator.}
Define~$\Wk: H^{\frac{1}{2}}(\Gamma) \rightarrow H^{-\frac{1}{2}}(\Gamma)$ as
\[
-\Wk \varphi := \gammaoint(\Ktildek \varphi) \qquad \forall \varphi \in H^{\frac{1}{2}}(\Gamma).
\]
For $C^\infty$-smooth~$\Gamma$, the operator $\Wk$ extends to~$\Wk :
H^{s}(\Gamma)\rightarrow H^{-1+s}(\Gamma)$ for all $s\in\mathbb R$.
\medskip

Let $\Vz$, $\Kz$, $\Kprimez$, and~$\Wz$ be the corresponding integral operators for zero
wave number  $k = 0$. 
Then, for all $s\ge 0$, the difference operators are linear bounded operators in the following spaces
\begin{equation} \label{compact:part}
\begin{split}
& \Vk-\Vz: H^{-\frac{1}{2}+s}(\Gamma) \rightarrow H^{\frac{5}{2}+s} (\Gamma), \quad \quad \Kk - \Kz: H^{\frac{1}{2}+ s}(\Gamma) \rightarrow H^{\frac{5}{2} +s}( \Gamma),\\
& \Kprimek-\Kprimez: H^{-\frac{1}{2}+s}(\Gamma) \rightarrow H^{\frac{3}{2}+s} (\Gamma), \quad \quad \Wk - \Wz: H^{\frac{1}{2}+ s}(\Gamma) \rightarrow H^{\frac{3}{2} +s}( \Gamma).
\end{split}
\end{equation}
In other words, the difference operators possess enhanced shift properties with respect to those of each term in the difference;
see, e.g., \cite[Prop.~{2.2}]{FEMBEM:mortar} and~\cite[Thm.~{7.2}]{mclean2000strongly}.
Moreover, $\Vz$ and $\Wz$
satisfy the following properties: there exist positive constants~$c_{\Vz}$,
$C_{\Vz}$, $c_{\Wz}$, and~$C_{\Wz}$ such that
\begin{alignat}{2} 
c_{\Vz}\Vert \varphi \Vert_{-\frac{1}{2}, \Gamma}^2 &\le \langle \varphi ,\Vz \varphi \rangle \le C_{\Vz}\Vert \varphi \Vert_{-\frac{1}{2}, \Gamma}^2  && \quad \forall \varphi \in H^{-\frac{1}{2}} (\Gamma), \label{coercivity:operators:Laplace}\\
c_{\Wz} \vert \varphi\vert_{\frac{1}{2}, \Gamma}^2 &\le \langle \Wz \varphi , \varphi \rangle \le C_{\Wz}\vert \varphi \vert_{\frac{1}{2}, \Gamma}^2 && \quad \forall \varphi \in H^{\frac{1}{2}} (\Gamma)/\mathbb C.\notag
\end{alignat}
We also have the following properties:
\[
\Vz^* = \V_0, \qquad \Kz^* = \Kprimez, \qquad \Wz^*=\Wz.
\]

\subsection{Mortar coupling} \label{subsection:mortar-coupling}
 
In this section, we recall the mortar coupling described in~\cite{FEMBEM:mortar}.
Instead of looking for solutions to~\eqref{transmission:problem}, we aim to solve the following three coupled problems for~$u: \Omega \rightarrow \mathbb C$ and $\uext$, $\m:\Gamma \rightarrow \mathbb C$:
\begin{align} 
& \begin{cases}
- \div(\Ad \nabla  u)    - (\kn)^2 u = \f & \text{in } \Omega,\\
\partial _{\nGamma} u + i \k  u - \m =0 & \text{on } \Gamma,\\
\end{cases} \label{first:system:Helmholtz}\\
& \begin{cases}
\uext = \PItD \m & \quad \quad \quad \quad \quad\, \text{on } \Gamma
\end{cases},\label{second:system:Helmholtz}\\
& \begin{cases}
u - \left[ \left( \frac{1}{2} + \Kk  \right) \uext - \Vk (\m - i\k \uext)   \right] = 0.
\end{cases} \label{third:system:Helmholtz}
\end{align}
The operator~$\PItD: H^{-\frac{1}{2}}(\Gamma) \rightarrow
H^{\frac{1}{2}}(\Gamma)$ appearing in~\eqref{second:system:Helmholtz}
maps the impedance mortar variable~$\m$ to the Dirichlet trace~$\uext$
of the solution to the exterior problem. This operator was defined, e.g., in~\cite[pp.~{124--126}]{actaBEMhelmholtz}.
In order to characterize it explicitly, we introduce the combined integral operators
\begin{equation} \label{definition:Bk:and:Akprime}
\Bk := -\Wk - i k \left( \frac{1}{2} - \Kk \right),\quad \quad \Aprimek := \frac{1}{2} + \Kprimek + i k \Vk
\end{equation}
and recall their mapping properties, see, e.g., \cite[Thm.~{2.27}]{actaBEMhelmholtz}:
\[
    \Bk: H^{s+\frac{1}{2}}(\Gamma) \rightarrow H^{s-\frac{1}{2}}(\Gamma), \quad \quad \Aprimek: H^{s-\frac{1}{2}}(\Gamma) \rightarrow H^{s-\frac{1}{2}}(\Gamma)
\]
are bounded.
Then, equation~\eqref{second:system:Helmholtz} is equivalent to
\begin{equation} \label{formula:Chandler-Wilde}
\Bk  \uext + i \k \Aprimek (\uext) -   \Aprimek \m = 0;
\end{equation}
see~\cite[Prop.~{3.2}]{FEMBEM:mortar} and the references therein.

The variational formulation of
problem~\eqref{first:system:Helmholtz}--\eqref{third:system:Helmholtz}
reads as follows:
\begin{equation} \label{weak:formulation:mortar:Helmholtz}
\begin{cases}
\text{Find } (u,\m, \uext) \in H^1(\Omega) \times H^{-\frac{1}{2}}(\Gamma) \times H^{\frac{1}{2}}(\Gamma) \text{ such that}&\\
(\Ad\nabla u, \nabla v)_{0,\Omega} - ((\kn)^2 \,u, v)_{0,\Omega} + i (\k u,v)_{0,\Gamma} - \langle \m, v \rangle = (\f, v)_{0,\Omega}& \forall v \in H^1(\Omega),\\
\langle  (\Bk + i \k  \Aprimek)\uext  - \Aprimek \m, \vtilde  \rangle = 0 & \forall \vtilde \in H^{\frac{1}{2}}(\Gamma),\\
\langle  u, \lambda \rangle - \langle  (\frac{1}{2} + \Kk) \uext - \Vk (\m - i\k\uext), \lambda\rangle =  0 & \forall \lambda \in H^{-\frac{1}{2}}(\Gamma).
\end{cases}
\end{equation}
As in~\cite{FEMBEM:mortar}, we introduce
\begin{equation} \label{form:T:for:Helmholtz}
\begin{split}
\T(	& (u, \m, \uext) , (v,\lambda,\vtilde) ) = (\Ad\nabla u, \nabla v )_{0,\Omega} -  ((\kn)^2 \,u, v)_{0,\Omega} + i\k ( u, v)_{0,\Gamma} - \langle \m, v \rangle \\
	& - \langle  (-\Wk- i \k ( \frac{1}{2} - \Kk  )  + i \k   ( \frac{1}{2} +\Kprimek + i \k \Vk ) )\uext \\
	& -  ( \frac{1}{2} +\Kprimek + i \k \Vk ) \m, \vtilde   \rangle      +  \langle u, \lambda \rangle -  \langle  (\frac{1}{2} + \Kk) \uext - \Vk (\m - i \k \uext), \lambda \rangle.\\
\end{split}
\end{equation}
Then, we can rewrite problem~\eqref{weak:formulation:mortar:Helmholtz} in compact form:
\begin{equation} \label{equivalent:weak:formulation}
\begin{cases}
\text{Find } (u, \m, \uext) \in H^1(\Omega) \times H^{-\frac{1}{2}}(\Gamma) \times H^{\frac{1}{2}}(\Gamma)\text{ such that}\\
\T( (u,\m, \uext) , (v,\lambda,\vtilde) ) = (\f, v)_{0,\Omega} \quad \forall (v,\lambda,\vtilde) \in H^1(\Omega) \times H^{-\frac{1}{2}}(\Gamma) \times H^{\frac{1}{2}}(\Gamma).
\end{cases}
\end{equation}
In~\cite[Thm.~{3.5}]{FEMBEM:mortar}, the well posedness of~\eqref{equivalent:weak:formulation} was proven, under the assumption of smoothness of~$\Gamma$ and uniqueness of the solution to problem~\eqref{transmission:problem},
based on the following G{\aa}rding inequality:

\begin{thm}(\!\!\cite[Thm.~{3.6}]{FEMBEM:mortar}) \label{theorem:Garding:inequality:cont}
Let~$\T(\cdot, \cdot)$ be defined as in~\eqref{form:T:for:Helmholtz}, and assume that the interface~$\Gamma$ is smooth.
Then, there exists~$c > 0$ depending only on $k_0$ and $\Omega$ and,
for each $k \ge k_0$, there is a positive constant~$\cG(k)$ depending on~$k$ and $\Omega$,
such that, for all $(v,\lambda,\vext)\in H^1(\Omega) \times H^{-\frac{1}{2}}(\Gamma) \times H^{\frac{1}{2}} (\Gamma)$,
\[
\begin{split}
\Re ( \T( (v,\lambda,\vext),   (v,\lambda,\vext) ) )
& \ge c \left\{ 
\| \Ad^{1/2} \nabla v \|^2_{0,\Omega}
+ \Vert \lambda \Vert^2_{-\frac{1}{2},\Gamma}
+ \Vert \vext
\Vert^2_{\frac{1}{2},\Gamma} 
\right\}  \\
& \quad - \left\{k^2 \Vert n\,v \Vert^2_{0,\Omega} 
+\cG(k) \left(\Vert \lambda \Vert_{-\frac{5}{2}, \Gamma}^2  + \Vert \vext \Vert_{-\frac{1}{2}, \Gamma}^2\right) \right\}.
\end{split}
\]
\end{thm}

\section{\DGFEMBEM{} mortar coupling} \label{section:dG}
We  introduce a discontinuous Galerkin finite element method-boundary element method (\DGFEMBEM) 
for the discretization of problem~\eqref{first:system:Helmholtz}--\eqref{third:system:Helmholtz}.
As a \dG\ discretization of~\eqref{first:system:Helmholtz} in the interior domain~$\Omega$, we use the
method introduced in~\cite{MelenkParsaniaSauter_generalDGHelmoltz},
which is based on the same variational formulation as that
of~\cite{TDGPW_pversion}.
For the sake of completeness, we recall the main steps of its
derivation in Section~\ref{subsection:dGOmega} below, in case of a smooth
coefficient $\Ad$.
Equation~\eqref{second:system:Helmholtz} is discretized as
in~\cite{FEMBEM:mortar}, while the discretization
of~\eqref{third:system:Helmholtz} is obtained by a suitable
modification described in Section~\ref{subsection:second-third-equation} of what is proposed in~\cite{FEMBEM:mortar}.
The complete discrete formulation
is summarized in Section~\ref{subsection:complete}. 

\subsection{\dG\ discretization of~\eqref{first:system:Helmholtz}}\label{subsection:dGOmega}

We shall work with regular, shape regular meshes of the (curved) domain $\Omega$. 
That is, the meshes will have no hanging nodes and the parametrizations of common edges or faces 
induced by the element maps of neighboring elements match; see \cite[Def.~{2.2}]{LiMelenkWohlmuthZou} for 
the precise statement.

As in~\cite[Def.~{2.2}]{LiMelenkWohlmuthZou}, we define a {\it curved
  $d$-simplex}~$K$, $d=2,3$, as the image of a reference straight
$d$-simplex $\widehat K$ through a $C^1$-diffeomorphism 
$\Phi_K$ satisfying
\begin{equation}\label{eq:SR}
  \|D \Phi_K\|_{L^\infty(\widehat K)}\le\gamma_{SR} h_K,\qquad
  \|(D \Phi_K)^{-1}\|_{L^\infty(\widehat K)}\le\gamma_{SR}h_K^{-1},
  \end{equation}
where $D \cdot$ denotes the Jacobian and $\gamma_{SR}>0$ is the shape regularity constant.

Condition~\eqref{eq:SR} implies that $\Phi_K$ can be decomposed as
$\Phi_K=\Phi_K^\Delta+\Psi_K$, where $\Phi_K^\Delta$ is an affine bijection and
$\Psi_K$ is a $C^1$ mapping such that
\begin{equation}\label{eq:Bernardi}
c_K:=\sup_{\widehat\xbold\in \widehat K}\|D\Psi_K(\widehat\xbold)\cdot
(D \Phi_K^\Delta)^{-1}\|_{L^\infty(\widehat K)}\lesssim 1.
\end{equation}
To see~\eqref{eq:Bernardi}, it is enough to fix any $\widehat\xbold_0\in \widehat K$
and take $\Phi_K^\Delta:=\Phi_K(\widehat\xbold_0)+D
\Phi_K(\widehat\xbold_0)(\widehat\xbold-\widehat\xbold_0)$. This
gives~\eqref{eq:Bernardi} with $\gamma_{SR}^2+1$ on the right-hand side.
A face $F$ of a curved 3-simplex $K$ is the image through $\Phi_K$ of
a face $\widehat F$ of $\widehat K$.

Let~$\{\taun\}_h$ be a sequence of 
conforming, i.e., regular in the sense described above, decompositions of~$\Omega$ into curved 3-simplices
with mesh granularity~$\h$. For $h$ sufficiently small, $c_K\le c<1$ for all $K\in \taun$.
The union of the (open) internal and boundary faces of $\taun$ are denoted by~$\FhI$
and $\FhB$, respectively.
We assume that all the faces in~$\FhI$ are flat. The faces in~$\FhB$ are curved 2-simplices.

Given an element~$\E \in \taun$, denote its diameter by $\hE$ and the outward pointing unit vector normal to~$\partial K$ by~$\nE$.
We introduce the mesh size function $\msf:\overline\Omega\to \mathbb R$,
where~$\msf_{|_\E}=\hE$ for all~$\E\in \taun$, $\msf=\min\{\hEone,\hEtwo\}$ on each face in~$\FhI$ shared by $\E_1$ and ~$\E_2$, and~$\msf=\hE$ on each face in~$\FhB$ on~$\partial \Omega$.
We may fix~$\msf$ arbitrarily at mesh vertices and on edges because we shall not need it there.

To derive the DG formulation, we write the first equation of~\eqref{first:system:Helmholtz} in mixed form:
\[
\begin{cases}
\i \k\, \sigmabold = \Ad\,\nabla u  \\
-\i \k\div(\sigmabold) - (kn)^2 u = \f 
\end{cases}
\quad \text{in } \Omega.
\]
On each element~$\E \in \taun$, we multiply the above two
  equations by smooth functions $\taubold$ and $v$, respectively, and
  integrate by parts:
\begin{equation}\label{1st-weak-dG}
\begin{cases}
\displaystyle{\int_\E \i \k \,\sigmabold \cdot \overline{\taubold} +
  \int_\E u \, \overline{{\div(\Ad\, \taubold)}} - \int_{\partial \E}
  \Ad\, u \,\overline{\taubold \cdot \nE} = 0},
\\[0.2cm]
\displaystyle{\int_\E \i \k\,\sigmabold \cdot \overline{\nabla v}  - \int_{\partial \E} \i \k\,\sigmabold \cdot \nE \,\overline{v} - \int_\E (kn)^2 u\, \overline{v} = \int_\E  \f \,\overline {v}}. 
\end{cases}
\end{equation}
We replace the traces of~$u$ and~$\sigmabold \cdot \nE$ in the integral on~$\partial \E$ with suitable numerical fluxes~$\uhat$
and~$\sigmaboldhat \cdot \nE$, respectively, which will be defined later on in~\eqref{eq:numericalfluxes}.
Thus, we replace~$u_{|\partial \E}$ with~$\uhat_{|\partial \E}$ in
the first equation of~\eqref{1st-weak-dG}, apply one more
integration by parts, select~$\taubold = \nabla v$, and end up with
\begin{equation} \label{1st-equation-substitute}
\int_\E \i\k\, \sigmabold \cdot \overline{\nabla v} - \int_\E \Ad\,\nabla u \cdot \overline{\nabla v} + \int_{\partial \E} \Ad\, (u - \uhat) \,\overline{\nabla v \cdot \nE} = 0. 
\end{equation}
Next, we replace~$\sigmabold \cdot \nE{}_{|\partial \E}$
with~$\sigmaboldhat \cdot \nE{}_{|\partial \E}$ in the second equation
of~\eqref{1st-weak-dG}, and obtain
\begin{equation} \label{2nd-equation-substitute}
\int_\E \i \k \,\sigmabold \cdot \overline{\nabla v} - \int_{\partial
  \E} \i\k\, \sigmaboldhat \cdot \nE \overline {v} - \int_\E (\kn)^2
u\, \overline{v} = \int_\E \f \,\overline {v}. 
\end{equation}
Subtracting~\eqref{1st-equation-substitute}
from~\eqref{2nd-equation-substitute} and adding over all $K\in \taun$
lead to the following broken variational formulation:
\begin{equation} \label{a:linear:combination}
\begin{cases}
\text{Find } u \in \Hpw^1(\Omegah) \text{ such that for all } v \in \Hpw^1(\Omegah) \\[0.2cm]
\displaystyle{\sum_{\E\in\taun}}\left(\int_\E \Ad\,\nabla u \cdot \overline{\nabla v} - \int_{\partial \E} \Ad\,(u - \uhat) \,\overline{\nabla v \cdot \nE}\right) - \int_{\Gamma} \i \k\, \sigmaboldhat \cdot \nGamma \overline{v} - \int_{\Omega} (\kn)^2 u \,\overline{v}= \int_{\Omega} \f \,\overline{v},
\end{cases}
\end{equation}
where
\[
\Hpw^1(\Omegah):=\{v\in L^2(\Omega):\, v{}_{|\E}\in H^1(K)\ \forall K\in \taun\}.
\]
In order to complete the definition of the \dG\ method, we need to
choose finite dimensional subspaces of $\Hpw^1(\Omegah)$ and define the numerical fluxes. 

To that end, we introduce the following notation for spaces of mapped, piecewise polynomial functions of finite degree.
Let $D\subset\mathbb R^3$ be an open, bounded Lipschitz domain with piecewise $C^\infty$-smooth boundary, and~$D_h$ a partition of~$D$ into curved simplices with flat internal faces.
Let $\ell\in\mathbb N_0$, and denote by $\mathbb P_\ell(\cdot)$ the space of
polynomials of degree at most $\ell$ on the domain within the brackets. For $\ell\ge 1$ and $r=0,1$, we set
\[
\mathcal S^{\ell, r}(D, D_h)=\{v\in H^r (D):\,
     v_{|_K}\circ \Phi_K\in\mathbb P_\ell(\widehat K)\quad \forall K\in D_h
     \}.
   \]
For later use, we also define mapped, piecewise polynomial spaces on
surface meshes. To that end, we assume that $S\subset\mathbb R^3$ is a
closed, piecewise $C^\infty$-smooth surface and let $S_h$ be a partition of
$S$ into curved 2-simplices, which is the trace of a partition $D_h$ of
its interior $D$ as above. For $\ell\ge 1$ and $r=0,1$, we set
  \[
\mathcal S^{\ell, r}(S, S_h)=\{v\in H^r(S):\,
     v_{|_F}\circ \Phi_{K_F}\in\mathbb P_\ell(\widehat F)\quad \forall F\in S_h
     \},
   \] 
where $K_{F}$ is the element of $D_h$ with $F$ as a face.

As for the \dG{} discretization of~\eqref{a:linear:combination},
we choose discretization spaces made by discontinuous piecewise polynomial functions:
\[
\Vh := \mathcal S^{p,0}(\Omega,\Omega_h),
\]
and the numerical fluxes introduced in~\cite{MelenkParsaniaSauter_generalDGHelmoltz,TDGPW_pversion}.
We recall their definition in the case of a smooth coefficient $\Ad$.
We first introduce the following notation for the jump and the average functionals on~$\FhI$
for smooth, scalar functions~$v$ and vector-valued functions~$\taubold$.
At any~$\xbf \in \FhI$ shared by the two elements $\E_{\xbf}^1$
and~$\E_{\xbf}^2$, 
the jumps $\llbracket v
\rrbracket$ and $\llbracket \taubold \rrbracket$, 
and the averages $\ldc  v   \rdc$ and $\ldc  \taubold \rdc$
are defined as 
\[
\begin{split}
& \llbracket v \rrbracket (\xbf) := v_{|\E_{\xbf}^1} (\xbf) \,\n_{\E_{\xbf}^1}  +  v_{|\E_{\xbf}^2}(\xbf) \, \n_{\E_{\xbf}^2}, \quad \quad \ldc  v   \rdc (\xbf) = \frac{1}{2} (v_{|\E_{\xbf}^1} (\xbf) + v_{|\E_{\xbf}^2}(\xbf) ), \\
& \llbracket \taubold \rrbracket (\xbf)  := \taubold_{|\E_{\xbf}^1}(\xbf)  \cdot\n_{\E_{\xbf}^1} + \taubold_{|\E_{\xbf}^2} (\xbf) \cdot\n_{\E_{\xbf}^2}, \quad \quad  \ldc  \taubold \rdc (\xbf)  :=\frac{1}{2} (\taubold _{|\E_{\xbf}^1} (\xbf) + \taubold_{|\E_{\xbf}^2}(\xbf) ).
\end{split}
\]
Then, given functions~$\alpha,\beta>0$ in~$L^\infty(\FhI)$ and~$\delta \in (0, 1/2]$ in~$L^\infty(\FhB)$, we define the following numerical fluxes:
\begin{equation}\label{eq:numericalfluxes}
\begin{split}
 \i \k \,\sigmaboldhat = \i \k \,\sigmaboldhat (u)
 &:=
\begin{cases}
\nu\ldc  \nabla_h u   \rdc - \i \k\,\Ad \alpha\, \llbracket u \rrbracket
& \text{on }\FhI,\\
\nabla _\h u - (1-\delta) (\nablah u + \i \k \,u \,\nGamma - \m \,\nGamma)
& \text{on } \FhB,\\
\end{cases}\\
\uhat=\uhat(u) 
&:= 
\begin{cases}
\ldc u \rdc - (\i\k)^{-1} \beta\, \llbracket \nabla _\h u \rrbracket 			& \text{on }\FhI,\\
u + \delta (-(\i \k)^{-1} \nabla_h u \cdot \nGamma - u + (\i\k)^{-1}\m)		& \text{on }\FhB,\\
\end{cases}
\end{split}
\end{equation}
where~$\nabla_\h$ denotes the elementwise application of the gradient
operator, and
we recall that~$\m$ denotes the impedance boundary datum in~\eqref{first:system:Helmholtz}.

Given positive constants~$\afrak$, $\bfrak$, and~$\dfrak$, with~$\afrak$ sufficiently large, see Remark~\ref{remark:choice-dG-parameters} below, and~$\bfrak$ and~$\dfrak$ sufficiently small,
see~\eqref{smallness-delta}, the functions~$\alpha$, $\beta$ and~$\delta$ are chosen as
\begin{equation} \label{dG-parameters}
\begin{split}
\alpha(\xbf) = \afrak \,
    &\frac{\p^2}{\k \msf(\xbf)}, \qquad \beta(\xbf) = \bfrak \,\frac{\k\msf(\xbf)}{\p} \qquad  \forall \xbf \in \FhI, \quad \delta(\xbf) =\dfrak \frac{\k\msf(\xbf)}{\p^2} \quad  \forall \xbf \in \FhB.
\end{split}
\end{equation}
The assumption $\delta \in (0,1/2]$ implies $\dfrak \in (0,\p^2/(2\k\hE))$.

The fluxes defined in~\eqref{eq:numericalfluxes} are single-valued on interior mesh faces and consistent,
which entails the consistency of the resulting \dG{} scheme (Lemma~\ref{lemma:consistency}).  Furthermore, they 
satisfy the following \emph{combined consistency} property:
\[
\i\k \,\sigmaboldhat \cdot \nGamma + \i \k \,\uhat =\m \quad \quad\text{on } \FhB.
\]
In the error analysis, we deal with
the interior \dG{}-error $u-\uh$, which is locally smooth but globally only in $L^2(\Omega)$.
Thus, for $r>0$, we introduce the broken Sobolev spaces on $\Omega_h$ as
\begin{align*}
  \Hpw^r(\Omegah):=
  \{v\in L^2(\Omega):\,
    v{}_{|\E}\in H^r(K)\ \forall K\in \taun\}.
\end{align*}
We also define the following two DG norms, which will be used in the analysis: 
Given~$v\in \Hpw^{\frac{3}{2} + \tepsilon} (\Omegah)$, with~$\tepsilon >0$
arbitrarily small, we define
\begin{equation} \label{dG:norm:1}
\begin{split}
\Vert v \Vert ^2_{\DG}	
& := \Vert \nu^{1/2}\,\nablah v\Vert^2_{0,\Omega} +  \Vert \kn\, v \Vert^2_{0,\Omega} + \k^{-1} \Vert \nu^{1/2}\beta^{1/2} \llbracket \nabla _\h v\rrbracket \Vert^2_{0,\FhI}  + \k \Vert \nu^{1/2}\alpha^{1/2} \llbracket v \rrbracket \Vert^2_{0,\FhI}\\
& \quad + \k^{-1} \Vert \delta ^{1/2} \nablah v \cdot \nGamma \Vert^2_{0,\Gamma} + \k \Vert (1-\delta) ^{1/2} v \Vert^2 _{0,\Gamma},
\end{split}
\end{equation}
and
\[
\begin{split}
\Vert v \Vert ^2_{\DGp} := \Vert v \Vert ^2_{\DG} + \k^{-1} \Vert \nu^{1/2}\alpha^{-1/2} \ldc \nablah v  \rdc \Vert^2_{0,\FhI}.
\end{split}
\]
In Section~\ref{section:continuous:problem}, we required the boundary~$\Gamma$ to be globally smooth,
whereas in this section we can allow for a piecewise smooth~$\Gamma$.
The global smoothness assumption is needed to promote the regularity of the solution to
problem~\eqref{transmission:problem}, while the piecewise smoothness assumption is enough for the design of the method.

\subsection{BEM discretization of~\eqref{second:system:Helmholtz} and discretization of~\eqref{third:system:Helmholtz}}   \label{subsection:second-third-equation}
On $\Gamma$, we introduce the curved simplicial mesh~$\Gammah$, whose elements are given by the intersection of the elements in~$\taun$ and~$\Gamma$.
As already mentioned, for~\eqref{second:system:Helmholtz},
whose variational formulation is given by 
the second equation in~\eqref{weak:formulation:mortar:Helmholtz}, we use
the same discretization as in~\cite[eqns.~(3.8) and (4.1)]{FEMBEM:mortar}, namely, a 
standard conforming BEM method with approximation spaces
\[
\Zh:=\mathcal S ^{\p,1} (\Gamma, \Gammah)\qquad\text{and}\qquad \Wh := \mathcal S ^{\p-1,0} (\Gamma, \Gammah)
\]
for $\uext$ and $m$, respectively.

  Next, we focus on the discretization of~\eqref{third:system:Helmholtz},
whose variational formulation is given by 
the third equation in~\eqref{weak:formulation:mortar:Helmholtz}.
Compared to what was done in~\cite{FEMBEM:mortar}, we add suitable terms that will allow us to prove a discrete G{\aa}rding inequality, see Theorem~\ref{theorem:discrete-Garding} below,
and retain consistency and adjoint consistency, see Lemma~\ref{lemma:consistency} and Proposition~\ref{lemma:adjointconsistency} below.
To that end, it is convenient to write the integral terms on~$\Gamma$ appearing in the \dG\ discretization in the interior domain $\Omega$ explicitly.
Using the definition of the numerical fluxes~\eqref{eq:numericalfluxes} on~$\FhB$, we write
\[
\begin{split}
& -\int_\Gamma (\uh - \uhhat) \overline{\nabla \vh \cdot \nGamma} - \int_\Gamma \i\k \sigmaboldhhat \cdot \nGamma \overline{\vh} \\
& = -\int_\Gamma - \delta (-(\i\k)^{-1} \nablah \uh \cdot \nGamma - \uh + (\i\k)^{-1} \mh) \overline{\nablah \vh \cdot \nGamma}\\
& \quad -\int_\Gamma (\nablah \uh \cdot \nGamma - (1-\delta) (\nablah \uh \cdot \nGamma + \i \k \uh - \mh) )  \overline{\vh}\\
& = -\int_\Gamma \delta (\i\k)^{-1} \nablah \uh\cdot \nGamma \overline{\nabla _\h \vh \cdot \nGamma} - \int_\Gamma \delta \uh \overline{\nablah \vh \cdot \nGamma} + \int_\Gamma \delta (\i\k)^{-1} \mh \overline{\nablah \vh \cdot \nGamma}\\
& \quad - \int_\Gamma \delta\nabla_h \uh \cdot \nGamma \overline{\vh} + \int_{\Gamma} (1-\delta) \i \k \uh \overline{\vh} - \int_\Gamma (1-\delta) \mh \overline {\vh}.
\end{split}
\]
Therefore, the contribution from the interior discretization to the coupling, i.e., the terms involving~$\mh$, is
\begin{equation}\label{eq:coupledterms}
\begin{split}
\int_\Gamma \delta (\i \k)^{-1} \mh \overline{\nablah \vh \cdot \nGamma} - \int_\Gamma (1-\delta) \mh \overline{\vh} 	& = - \int_\Gamma \mh (\overline{\delta (\i\k)^{-1} \nablah \vh \cdot \nGamma + (1-\delta)\vh})\\
																						& = -(\mh ,\delta (\i\k)^{-1} \nablah \vh \cdot \nGamma + (1-\delta)\vh  )_{0,\Gamma}.\\
\end{split}
\end{equation}
We have to discretize the third equation in~\eqref{weak:formulation:mortar:Helmholtz} in such a way that we have terms that match some of the terms in~\eqref{eq:coupledterms} when proving a discrete G{\aa}rding inequality; see Proposition~\ref{proposition:continuity-coercivity:dG:weakened} below.

For $\mh\in \Wh$ and $\uhext\in \Zh$, we abbreviate, for convenience,
\begin{equation}\label{eq:Xh}
\Xh := (\frac{1}{2} + \Kk) \uhext - \Vk (\mh - \i\k \uhext),
\end{equation}
and introduce the following discretization of the third equation of~\eqref{weak:formulation:mortar:Helmholtz}:
\[
\langle -\delta (\i\k)^{-1} \nablah \uh
  \cdot\nGamma + (1-\delta) \uh + \delta (\i\k)^{-1} \mh, \lambdah \rangle -
\langle \Xh, \lambdah \rangle=0\quad\forall \lambdah\in \Wh.
\]
The term $-\delta (\i\k)^{-1} \nablah \uh \cdot\nGamma -\delta \uh$ is added in order to be able to prove the G{\aa}rding inequality, while
the term $\delta (\i\k)^{-1} \mh$ is added in order to restore consistency.
The signs of the terms are chosen in a way that gives a convenient structure to the adjoint problem; see Section~\ref{section:adjoint} below.
  
\subsection{Complete discrete formulation}  \label{subsection:complete}
On $\Vh\times \Vh$, we define local the sesquilinear forms~$\ahE
(\cdot, \cdot)$ for all~$\E\in \taun$ by
\[
\begin{split}
\ahE (\uh, \vh) : = & \int_\E \Ad\,\nablah \uh \cdot \overline{\nablah \vh} - \int_\E (\kn)^2 \uh \overline{\vh} \\
                    &- \sum_{\F \subset\partial K\cap\FhI}  \left( \int_\F \Ad\, (\uh-\uhhat) \overline{\nablah \vh \cdot \nE} + \int_\F  \i\k \sigmaboldhhat \cdot \nE \overline{\vh}  \right),
\end{split}
\]
with fluxes $\uhat$ and $\sigmaboldhat$ as
in~\eqref{eq:numericalfluxes},
and the global boundary sesquilinear form~$\bhGamma (\cdot, \cdot)$ by
\[
\begin{split}
\bhGamma (\uh, \vh)
& := - \int_\Gamma \delta (\i\k)^{-1} \nablah \uh \cdot \nGamma \overline{\nablah \vh \cdot \nGamma} 
- \int_\Gamma \delta \uh \overline {\nablah \vh \cdot \nGamma}\\
& \quad \quad - \int_\Gamma \delta \nabla_h \uh \cdot \nGamma \overline{\vh} + \int_{\Gamma} (1-\delta)\i\k \uh \overline{\vh}.
\end{split}
\]
With~$\Vh=\mathcal S^{p,0}(\Omega,\Omega_h)$, $\Wh=\mathcal S ^{\p-1,0} (\Gamma, \Gammah)$, and $\Zh=\mathcal S^{p,1}(\Gamma,\Gamma_h)$,
the full \DGFEMBEM{} method reads as follows:
\begin{equation} \label{dGBEM:long-version}
\begin{cases}
\text{Find } (\uh, \mh,\uhext) \in \Vh \times \Wh \times \Zh 
\text{ such that, for all }(\vh,\lambdah,\vexth) \in \Vh \times \Wh \times \Zh, \\[0.2cm]
\displaystyle{\sum_{\E \in \taun} \ahE (\uh,\vh) + \bhGamma (\uh,\vh) - (\mh, \delta (\i\k)^{-1} \nabla_h \vh \cdot \nGamma + (1-\delta)\vh)_{0,\Gamma} = (\f,\vh)_{0,\Omega}},\\[0.2cm]
\langle (\Bk + \i\k \Aprimek) \uhext - \Aprimek \mh, \vexth \rangle =0, \\[0.2cm]
\langle -\delta (\i\k)^{-1} \nabla_h \uh \cdot \nGamma + (1-\delta) \uh + \delta (\i\k)^{-1}\mh, \lambdah\rangle  - \langle \Xh, \lambdah \rangle = 0,\\
\end{cases}
\end{equation}
where the combined integral operators~$\Bk$ and~$\Aprimek$ are as in~\eqref{definition:Bk:and:Akprime}
and $\Xh$ is as in~\eqref{eq:Xh}.

The definition of~$\uhat$ and~$\sigmaboldhat$ in~\eqref{eq:numericalfluxes} entails
\[
\begin{split}
& \sum_{\E \in \taun} \ahE (\uh,\vh)= \sum_{\E \in \taun}\left(\int_\E \Ad\,\nabla \uh \cdot \overline{\nabla \vh} - \int_\E (\kn)^2\uh \overline{\vh}\right) \\
& -\int_{\FhI}\nu\left(\llbracket \uh \rrbracket \cdot \ldc \overline{\nabla_h \vh} \rdc +\ldc \nabla_h \uh \rdc \cdot \llbracket \overline{\vh} \rrbracket \right)
-\int_{\FhI}\nu \beta (ik)^{-1} \llbracket \nabla_h \uh \rrbracket  \llbracket \overline{\nabla_h \vh} \rrbracket
            +\int_{\FhI} \Ad\alpha\, ik \llbracket  \uh \rrbracket \cdot \llbracket \overline{\vh} \rrbracket.
\end{split}
\]
By introducing the sesquilinear form
\begin{equation} \label{trisesquilinear-form}
\begin{split}
& \Th ( (\uh,\mh,\uhext) , (\vh,\lambdah, \vexth) )  \\
& := \sum_{\E \in \taun} \ahE (\uh,\vh) + \bhGamma(\uh,\vh) - (\mh, \delta (\i\k)^{-1} \nablah \vh \cdot \nGamma + (1-\delta)\vh)_{0,\Gamma} \\
& \quad \quad - \langle (-\Wk -\i\k (\frac{1}{2}-\Kk) + \i\k (\frac{1}{2}+\Kprimek+\i\k\Vk))\uhext - (\frac{1}{2} + \Kprimek + \i\k \Vk)\mh, \vexth \rangle\\
& \quad \quad + \langle -\delta (\i\k)^{-1} \nablah \uh \cdot \nGamma + (1 - \delta) \uh + \delta (\i\k)^{-1} \mh, \lambdah \rangle\\
& \quad \quad - \langle (\frac{1}{2} +\Kk) \uhext - \Vk (\mh - \i\k \uhext) , \lambdah   \rangle,
\end{split}
\end{equation} 
method~\eqref{dGBEM:long-version} can be written in compact form as follows:
\begin{equation} \label{dgBEM:short-version}
\begin{cases}
\text{Find } (\uh, \mh,\uhext) \in \Vh \times \Wh \times \Zh \text{ such that}\\
\Th\left( (\uh,\mh,\uhext) , (\vh,\lambdah, \vexth) \right) = (\f,\vh)_{0,\Omega} \quad \forall (\vh, \lambdah, \vexth) \in \Vh \times \Wh \times \Zh.
\end{cases}
\end{equation}
\begin{lem}
 \label{lemma:consistency}
Let the exact solution~$(u,\m,\uext)$ to~\eqref{first:system:Helmholtz}--\eqref{third:system:Helmholtz}
belong to  $ H^{\frac{3}{2}+t}(\Omega)\times  L^2(\Gamma) \times H^{\frac{1}{2}}(\Gamma)$, for some 
$t>0$.
Then, the discrete \DGFEMBEM{} coupling~\eqref{dgBEM:short-version}, or equivalently~\eqref{dGBEM:long-version}, is consistent, i.e.,
\begin{align}  \label{dgBEM:consistency}
\Th\left( (u,\m,\uext) , (\vh,\lambdah, \vexth) \right) = (\f,\vh)_{0,\Omega} \quad \forall (\vh, \lambdah, \vexth) \in \Vh \times \Wh \times \Zh.
\end{align}
\end{lem}
\begin{proof}
 See Appendix~\ref{appendix:sectionA}.
\end{proof}
An immediate consequence of~\eqref{dgBEM:short-version} and~\eqref{dgBEM:consistency} is the following Galerkin orthogonality property:
For all $(\vh, \lambdah, \vexth) \in \Vh \times \Wh \times \Zh$,
\begin{equation} \label{Galerkin:orthogonality}
\Th ( (u-\uh, \m-\mh, \uext-\uhext), (\vh, \lambdah, \vexth))=0 .
\end{equation}

\section{A G{\aa}rding inequality}\label{section:Garding}
In this section, we establish in Theorem~\ref{theorem:discrete-Garding} a G{\aa}rding inequality for the form $\Th( (\cdot,\cdot,\cdot), (\cdot,\cdot,\cdot))$ defined in~\eqref{trisesquilinear-form}.
We start with a remark and some preliminary results.

\begin{remark} \label{remark:choice-dG-parameters}
For any $\E\in\taun$, introduce~$\Ctrace(\p,\E)$ as the smallest constant such that
\begin{equation} \label{trace:inequality}
\Vert \nabla \vh \Vert_{0,\partial \E} \le \Ctrace(\p,\E) \Vert \nabla
\vh \Vert_{0,\E}\qquad \forall \vh \in \Vh.
\end{equation}
For straight elements,
it is well known that~$\Ctrace(\p,\E) \lesssim \p\hE^{-1/2}$; see, e.g., \cite[Thm.~{4.76}]{SchwabpandhpFEM}.
Under the shape regularity assumption~\eqref{eq:SR},
this is valid also for curved elements. In fact, given $\vh \in
\Vh$, let $\widehat v_h$ be the pull-back of ${\vh}_{|_{\E}}$
through the mapping $\Phi_K: \Ehat\to \E$.
Since $\widehat v_h$ is a polynomial and $\Ehat$ is a straight
simplex, we have
\[
\Vert \nabla \vh \Vert_{0,\partial \E}\lesssim \Vert \widehat\nabla \widehat v_h \Vert_{0,\partial \Ehat}
\lesssim p \Vert \widehat\nabla \widehat v_h \Vert_{0, \Ehat}\lesssim ph_\E^{-1/2}\Vert \nabla \vh \Vert_{0,\E}.
\]
In the light of this, we demand the following assumptions: for~$\hE$ sufficiently
small and~$\p$ sufficiently large,
\begin{equation} \label{condition:alpha}
  \alpha(\xbf) \ge \frac{\aaleph}{\k}
  \max_{\E \in \{  \E^-_{\xbf}, \E^+_{\xbf}\}}  \Ctrace^2(\p,\E) \quad \forall \xbf \in \FhI,
\end{equation}
where~$\aaleph$ is a constant, which will be fixed in the proof of Proposition~\ref{proposition:continuity-coercivity:dG:weakened} below; see equation~\eqref{aleph}.
\eremk
\end{remark}

The following coercivity/continuity result is valid.
\begin{prop} \label{proposition:continuity-coercivity:dG}
Let $\alpha$ satisfy~\eqref{condition:alpha} and $0<\delta<1/2$.
Then, there exists a positive constant~$\ccoer$ independent of~$\h$, $\k$, $\p$, $\alpha$, $\beta$, and~$\delta$, such that
\begin{equation} \label{coercivity:Omega}
\left \vert \sum_{\E \in \taun} \ahE(\vh,\vh) + \bhGamma(\vh,\vh)  \right\vert     \ge \ccoer \Vert \vh \Vert^2_{\DG} - \Vert \kn  \vh \Vert^2_{0,\Omega}
\quad \quad \forall \vh\in\Vh.
\end{equation}
Moreover, for any~$\tepsilon >0$, there exists a positive constant~$\cc$ independent of~$\h$, $\k$, $\p$, $\alpha$, $\beta$, and~$\delta$, such that
\begin{align}
\label{continuity:abplusplus} \left \vert \sum_{\E \in \taun} \ahE(u,v) + \bhGamma(u,v)  \right\vert  &\le \cc \Vert u \Vert_{\DGp} \Vert v \Vert_{\DGp}  \quad \quad \forall u,\, v \in \Hpw^{\frac{3}{2}+\tepsilon} (\Omegah), \\
\label{continuity:abdgplusdg} \left \vert \sum_{\E \in \taun} \ahE(u,\vh) + \bhGamma(u,\vh)  \right\vert  &\le \cc \Vert u \Vert_{\DGp} \Vert \vh \Vert_{\DG}  \quad \quad \forall u\in \Hpw^{\frac{3}{2}+\tepsilon} (\Omegah),\; \forall \vh \in \Vh , \\
\label{continuity:abdgdgplus} \left \vert \sum_{\E \in \taun}
  \ahE(\uh,v) + \bhGamma(\uh,v)  \right\vert  & \le \cc  \Vert \uh \Vert_{\DG} \Vert v \Vert_{\DGp}  \quad \quad \forall \uh\in \Vh,\; \forall v \in \Hpw^{\frac{3}{2}+\tepsilon}(\Omegah). 
\end{align}
\end{prop}
\begin{proof}
The proof of~\cite[Prop.~{3.1}]{MelenkParsaniaSauter_generalDGHelmoltz} applies also in our context.
As for~\eqref{coercivity:Omega}, we also refer to the proof of Proposition~\ref{proposition:continuity-coercivity:dG:weakened} below.
As for~\eqref{continuity:abplusplus}--\eqref{continuity:abdgdgplus}, we use the trace inequality~\eqref{trace:inequality} and assumption~\eqref{condition:alpha}.
\end{proof}

The coercivity bound~\eqref{coercivity:Omega} can be refined, as
described in the following result, which is instrumental in the proof
of the G{\aa}rding inequality in Theorem~\ref{theorem:discrete-Garding} below.
\begin{prop} \label{proposition:continuity-coercivity:dG:weakened}
Given~$\varepsilon>0$, there exist~$\afrakz>0$, $\bfrakz>0$, and~$\dfrakz>0$ independent of~$\k$
such that, for all~$\afrak \ge \afrakz$, $\bfrak \le \bfrakz$, and~$\dfrak \le \dfrakz$ in~\eqref{dG-parameters}, and
for all~$\vh\in\Vh$, the following bound is valid:
\begin{equation} \label{coercivity:Omega:weakened}
\begin{split}
& (\Re + \varepsilon \Im) ( \sum_{\E \in \taun} \ahE(\vh,\vh) + \bhGamma(\vh,\vh) )   \ge  \frac{1}{2}\Vert \Ad^{1/2}\nabla _\h \vh \Vert^2_{0,\Omega} -  \Vert \kn \vh \Vert^2_{0,\Omega} \\
& \quad \quad \quad \quad \quad + \frac{1}{2}\varepsilon \left( \k^{-1} \Vert \Ad^{1/2}\beta^{1/2} \llbracket \nabla_h\vh \rrbracket \Vert^2_{0,\FhI}   + \k \Vert \Ad^{1/2}\alpha^{1/2} \llbracket \vh \rrbracket \Vert^2_{0,\FhI}  \right.\\
& \quad \quad \quad \quad \quad \quad \quad \quad  \left. + \k^{-1} \Vert \delta^{1/2} \nablah \vh \cdot \nGamma \Vert^2_{0,\Gamma} +  \k \Vert (1-\delta) \vh\Vert^2_{0,\Gamma}  \right).
\end{split}
\end{equation}
\end{prop}
\begin{proof}
The proof is a modification of that of~\cite[Prop.~{3.1}]{MelenkParsaniaSauter_generalDGHelmoltz}. 
We begin by observing that
\begin{equation} \label{starting:equation:weak}
\begin{split}
& (\Re + \varepsilon \Im) ( \sum_{\E \in \taun} \ahE(\vh,\vh) + \bhGamma(\vh,\vh)  )   	\\
& = \Vert \Ad^{1/2}\nabla _\h \vh \Vert^2_{0,\Omega} - \Vert \kn  \vh \Vert^2_{0,\Omega} -2\Re (  \int_{\FhI} \Ad \llbracket \vh \rrbracket \overline{\ldc  \nablah \vh  \rdc}  )    -2\Re (  \int_{\Gamma} \delta \vh \overline{\nablah \vh\cdot \nGamma} )   \\
& \quad +\varepsilon \Big(\k^{-1} \Vert \Ad^{1/2}\beta^{1/2} \llbracket \nablah \vh \rrbracket \Vert^2_{0,\FhI} + \k^{-1} \Vert \delta^{1/2} \nabla_\h \vh \cdot \nGamma \Vert^2_{0,\FhB} \\
& \quad\quad\quad + \k \Vert \Ad^{1/2}\alpha^{1/2} \llbracket \vh \rrbracket \Vert^2_{0,\FhI} + \k \Vert (1-\delta)^{1/2} \vh \Vert^2_{0,\Gamma} \Big).
\end{split}
\end{equation}
Using the Young inequality with weight~$\varepsilon \k /2$ entails
\[
\begin{split}
\left\vert 2\Re \left(  \int_{\FhI}  \Ad\llbracket \vh \rrbracket \overline{\ldc  \nablah \vh  \rdc}   \right)    \right\vert   
& \le \frac{\varepsilon \k }{2} \left\Vert  \Ad^{1/2}\alpha^{1/2} \llbracket \vh \rrbracket   \right\Vert^2_{0,\FhI}  + \frac{2}{\varepsilon \k} \left\Vert  \Ad^{1/2}\alpha^{-1/2} \ldc \nablah \vh \rdc \right\Vert_{0, \FhI}^2\\
& \le \frac{\varepsilon \k}{2} \left\Vert \Ad^{1/2}\alpha^{1/2} \llbracket \vh \rrbracket   \right\Vert^2_{0,\FhI}  + \sum_{\E \in \taun} \frac{1}{\varepsilon \k} \left\Vert  \Ad^{1/2}\alpha^{-1/2} \nablah \vh \right\Vert_{0,\partial K \setminus \Gamma}^2.
\end{split}
\]
For the second summand, we use~\eqref{trace:inequality} and~\eqref{condition:alpha} to obtain
\[
\begin{split}
 \sum_{\E \in \Th}\frac{1}{\varepsilon \k} \left\Vert \Ad^{1/2}\alpha^{-1/2} \nablah \vh  \right\Vert_{0,\partial \E \setminus \Gamma}^2 	
 &\le \sum_{\E \in \Th} \frac{\numax}{\varepsilon \aaleph}
 \Vert \nablah \vh \Vert^2_{0,\E}
 \le  \sum_{\E \in \Th} \frac{\numax}{ \varepsilon \aaleph \numin}
 \Vert \nu^{1/2} \nablah \vh \Vert^2_{0,\E}.
\end{split}
\]
Fix
\begin{equation} \label{aleph}
\aaleph = 2\numax/(\varepsilon \numin)
\end{equation}
and get
\[
\sum_{\E \in \Th} \frac{1}{\varepsilon \k} \left\Vert \Ad^{1/2}\alpha^{-1/2} \nablah \vh  \right\Vert_{0,\partial \E \setminus \Gamma}^2 
\le \frac{1}{2} \Vert \Ad^{1/2} \nablah \vh\Vert^2_{0,\Omega} .
\]
We deduce
\begin{equation} \label{new:bound1}
\left\vert 2\Re \left(  \int_{\FhI} \Ad \llbracket \vh \rrbracket \overline{\ldc  \nablah \vh  \rdc} \right) \right\vert 
\le \frac{\varepsilon \k}{2} \left\Vert \Ad^{1/2}\alpha^{1/2} \llbracket \vh \rrbracket   \right\Vert^2_{0,\FhI} 
+  \frac{1}{2} \Vert \Ad^{1/2} \nabla _\h \vh \Vert^2_{0,\Omega}.
\end{equation}
\medskip

\noindent
We deal with the fourth term on the right-hand side of~\eqref{starting:equation:weak} analogously:
For any constant~$t>0$, we have
\[
\left\vert  2\Re\left(  \int_{\Gamma} \delta \vh \overline{\nablah \vh\cdot \nGamma}   \right)   \right\vert 
											\le t\k \left\Vert \frac{\delta }{1-\delta} \right\Vert_{\infty,\Gamma} \Vert (1-\delta)^{1/2} \vh \Vert^2_{0,\Gamma} + \frac{1}{t\k} \Vert \delta^{1/2} \nablah \vh\cdot \nGamma \Vert^2_{0,\Gamma}.
\]
Take~$t=\frac{1}{2}\Vert \frac{\delta} {1- \delta }\Vert^{-1}_{\infty, \Gamma}$ and get
\begin{equation} \label{new:bound2}
\left\vert  2\Re\left(  \int_{\Gamma} \delta \vh \overline{\nablah \vh\cdot \nGamma}   \right)   \right\vert 
\le \frac{\k}{2} \Vert (1-\delta)^{1/2} \vh \Vert^2_{0, \Gamma} + \frac2\k\left\Vert \frac{ \delta }{1-\delta }\right\Vert_{\infty, \Gamma} \Vert \delta^{1/2} \nablah \vh\cdot \nGamma \Vert^2_{0,\Gamma}.
\end{equation}
Inserting~\eqref{new:bound1} and~\eqref{new:bound2} into~\eqref{starting:equation:weak}, we get
\[
\begin{split}
& (\Re + \varepsilon \Im) \left( \sum_{\E \in \taun} \ahE(\vh,\vh) +
  \bhGamma(\vh,\vh)  \right)  \ge \frac{1}{2} \Vert \Ad^{1/2}\nablah \vh \Vert^2_{0,\Omega} - \Vert \kn \vh \Vert^2 _{0,\Omega} \\
& \quad\quad + \varepsilon \k^{-1} \Vert \Ad^{1/2}\beta^{1/2} \llbracket \nablah \vh \rrbracket \Vert^2_{0,\FhI}
+\left( \varepsilon - 2 \left\Vert \frac{\delta }{1-\delta } \right\Vert_{\infty,\Gamma} \right) \k^{-1} \Vert \delta^{1/2} \nablah \vh \cdot \nGamma \Vert^2_{0,\Gamma} \\
& \quad\quad + \frac{1}{2} \varepsilon \k \Vert \Ad^{1/2}\alpha^{1/2}\llbracket \vh \rrbracket \Vert^2_{0, \FhI} +
\frac{1}{2}\varepsilon \k \Vert (1-\delta) ^{1/2} \vh \Vert^2_{0,\Gamma}.\\
\end{split}
\]
Take~$\dfrak$ such that
\begin{equation} \label{smallness-delta}
2 \left\Vert \frac{\delta} {1- \delta} \right\Vert_{\infty,\Gamma} \le \frac{1}{2} \varepsilon
\end{equation}
and deduce
\[
\begin{split}
& (\Re + \varepsilon \Im) \left( \sum_{\E \in \taun} \ahE(\vh,\vh) + \bhGamma(\vh,\vh)  \right)  \ge \frac{1}{2} \Vert \Ad^{1/2}\nablah \vh \Vert^2_{0,\Omega} - \Vert \kn \vh \Vert^2 _{0,\Omega} \\
& \quad \quad + \frac{1}{2}\varepsilon \left(\k^{-1} \Vert \Ad^{1/2}\beta^{1/2} \llbracket \nablah \vh \rrbracket \Vert^2_{0,\FhI} + \k^{-1} \Vert \delta^{1/2} \nablah \vh \cdot \nGamma \Vert^2_{0, \Gamma} \right.\\
& \quad \quad\quad \quad\quad \left. + \k \Vert \Ad^{1/2}\alpha^{1/2} \llbracket \vh \rrbracket \Vert^2_{0, \FhI} + \k \Vert (1-\delta) ^{1/2} \vh \Vert^2_{0,\Gamma} \right),\\
\end{split}
\]
whence the assertion follows.
\end{proof}
An explicit choice of~$\varepsilon$ in the 
bound~\eqref{coercivity:Omega:weakened} is given in the proof of the G{\aa}rding inequality in Theorem~\ref{theorem:discrete-Garding}.

Next, we present a discontinuous-to-continuous reconstruction operator for piecewise smooth functions on curvilinear simplicial meshes.
\begin{thm} \label{lemma:Karakashian-style}
Let $\taun$ be a shape-regular mesh on $\Omega$ as defined in Section~\ref{subsection:dGOmega}.
Then, there exists $c>0$ depending only on~$\Omega$ and~$\gamma_{SR}$ in~\eqref{eq:SR} such that, for each~$\ell \in \Nbb$, there exists
a linear operator~$\Pkar : \Hpw^1(\taun) \rightarrow H^1(\Omega)$ that satisfies, for all~$v \in \Hpw^1(\taun)$,
\begin{align}
\label{K-propertyA}
\Vert \nabla \Pkar v \Vert_{0,\Omega} &\le c \left(  \Vert \nablah v \Vert_{0,\Omega} + \Vert \msf^{-1/2} \ell \llbracket v \rrbracket \Vert_{0,\FhI} \right),\\
\label{K-propertyB}
\Vert \Pkar v \Vert_{0,\Omega} & \le c \left(  \Vert \msf \ell^{-2}\nablah v \Vert_{0,\Omega} + \Vert v \Vert_{0,\Omega} + \Vert \msf^{1/2} \ell^{-1}\llbracket v \rrbracket \Vert_{0,\FhI}\right),\\
\label{K-propertyC}
\Vert \msf^{-1/2} \ell\,(I-\Pkar) v \Vert_{0,\Gamma} &\le c \left(  \Vert \nablah v \Vert_{0,\Omega} + \Vert \msf^{-1/2} \ell \llbracket v \rrbracket \Vert_{0,\FhI} \right).
\end{align}
\end{thm}
\begin{proof}
We postpone the proof to Appendix~\ref{appendix:proof-Lemma-splitting} below.
\end{proof}

\begin{remark}
The parameter~$\ell$ appearing in the statement of Theorem~\ref{lemma:Karakashian-style} is not necessarily related to the polynomial degree of the \dG{} space under consideration.
However, it will be apparent in Theorem~\ref{theorem:discrete-Garding} and Proposition~\ref{proposition:continuity:trisesquilinear-form} below that
a natural choice in our framework is in fact~$\ell=\p$.
\eremk
\end{remark}

\begin{remark}
Theorem~\ref{lemma:Karakashian-style} relates to similar results in the literature; see, e.g., \cite[Sec.~{5.2}]{MR2299768} and~\cite[Prop.~{5.2}]{MR2290408}.
With respect to the first reference, we provide here optimal estimates also on curvilinear simplicial meshes;
moreover, differently from the second reference, we also present stability estimates for the elemental $L^2$ norm.
Furthermore, we define the reconstruction operator for piecewise sufficiently smooth functions, without restricting to piecewise polynomial functions.
The price to pay is that the image of this operator is not an $H^1$-conforming piecewise polynomial space over the decomposition~$\Omegah$,
but rather on a sufficiently fine shape regular refinement of~$\Omegah$; see Appendix~\ref{appendix:proof-Lemma-splitting} below for more details.
\eremk
\end{remark}

We are left to prove the main result of the section, namely the following discrete G{\aa}rding inequality for the form defined in~\eqref{trisesquilinear-form}.
\begin{thm} \label{theorem:discrete-Garding}
Let~$\Th( (\cdot,\cdot,\cdot), (\cdot,\cdot,\cdot))$ be defined as
in~\eqref{trisesquilinear-form} and the interface~$\Gamma$ be smooth.
Then, the following G{\aa}rding inequality is valid:
there exist a constant $\varepsilon >0$ only depending on $\Omega$ and~$\gamma_{SR}$ in~\eqref{eq:SR} (see~\eqref{various:epsilon:1}),
three constants~$\afrakz>0$, $\bfrakz>0$, and~$\dfrakz>0$
depending additionally on $\Ad$, and a positive constant~$\cG(\k)$
depending additionally on~$\k$ such that, for all~$\afrak \ge \afrakz$, $\bfrak \le \bfrakz$, and~$\dfrak \le \dfrakz$ in~\eqref{dG-parameters},
\begin{equation} \label{Garding:inequality}
\begin{split}
& (\Re + \varepsilon \Im)\Th ( (\vh, \lambdah, \vexth), (\vh, \lambdah, \vexth))  \\
& \gtrsim \Vert \Ad^{1/2}\nablah \vh \Vert ^2_{0,\Omega}  + \varepsilon \k^{-1} \Vert \delta ^{1/2} \nablah \vh \cdot \nGamma \Vert^2_{0,\Gamma} + \varepsilon \k \Vert \Ad^{1/2}\alpha^{1/2} \llbracket \vh \rrbracket \Vert^2_{0,\FhI}  \\
& \quad + \varepsilon \k^{-1} \Vert \Ad^{1/2}\beta^{1/2} \llbracket \nablah \vh \rrbracket \Vert^2_{0,\FhI} + \varepsilon \k \Vert (1-\delta) \vh \Vert^2_{0,\Gamma} + \Vert \lambdah \Vert^2_{-\frac{1}{2}, \Gamma} + \Vert \vexth \Vert^2_{\frac{1}{2}, \Gamma} \\
& \quad  -\Vert \kn \, \vh \Vert^2_{0,\Omega}
- \cG(\k) \left( \Vert \lambdah \Vert^2_{-\frac{3}{2}, \Gamma}  +  \Vert \vexth \Vert^2_{-\frac{1}{2}, \Gamma} \right) \quad \forall (\vh,\lambdah, \vexth) \in \Vh \times \Wh \times \Zh.
\end{split}
\end{equation}
The hidden constant depends on $c_{\Vz}$ and $c_{\Wz}$
in~\eqref{coercivity:operators:Laplace} but not on $\k$.
\end{thm}
\begin{proof}
Observe that
\[
\begin{split}
& \Th ( (\vh, \lambdah, \vexth), (\vh, \lambdah, \vexth)) \\
& = \sum_{\E \in \taun}  \ahE (\vh,\vh) + \bhGamma (\vh,\vh)  -   (\lambdah, \delta (\i\k)^{-1} \nablah \vh \cdot \nGamma + (1-\delta)\vh)_{0,\Gamma}  \\
&  \quad - \langle (-\Wk -\i\k (\frac{1}{2}-\Kk) + \i\k (\frac{1}{2}+\Kprimek+\i\k\Vk)) \vexth - (\frac{1}{2} + \Kprimek + \i\k \Vk) \lambdah, \vexth \rangle\\
& \quad - \overline{\langle \lambdah, \delta (\i\k)^{-1} \nablah \vh \cdot \nGamma - (1 - \delta) \vh - \delta (\i\k)^{-1} \lambdah \rangle} -  \overline{ \langle \lambdah, (\frac{1}{2} +\Kk) \vexth - \Vk (\lambdah - \i\k \vexth)    \rangle } .
\end{split}
\]
Equivalently, we write
\[
\begin{split}
& \Th ( (\vh, \lambdah, \vexth), (\vh, \lambdah, \vexth))  = \sum_{\E \in \taun} \ahE(\vh, \vh)   +\bhGamma(\vh, \vh)  \\
&  \quad +  \langle \Wk \vexth, \vexth \rangle + \i\k \langle (\frac{1}{2} - \Kk) \vexth, \vexth \rangle - \i\k \langle (\frac{1}{2} + \Kprimek + \i\k \Vk) \vexth, \vexth \rangle \\
&  \quad +\langle (\frac{1}{2} + \Kprimek) \lambdah, \vexth \rangle + \i\k \langle \Vk \lambdah, \vexth \rangle  - \i \k^{-1} \Vert \delta^{1/2} \lambdah \Vert^2_{0,\Gamma} \\
& \quad  -  \overline{\langle \lambdah, (\frac{1}{2} + \Kk) \vexth\rangle} + \overline{\langle \lambdah, \Vk \lambdah \rangle}  - \i\k \overline{\langle \lambdah, \Vk \vexth \rangle} \\
& \quad -2\Re\left(  \langle \lambdah, \delta (\i\k)^{-1} \nablah \vh \cdot \nGamma \rangle    \right) -2 \i \, \Im \left(  \langle (1-\delta)\lambdah, \vh \rangle   \right),
\end{split}
\]
and thus
\[
\begin{split}
& \Th ( (\vh, \lambdah, \vexth), (\vh, \lambdah, \vexth))  
= \sum_{\E \in \taun} \ahE(\vh, \vh)   +\bhGamma(\vh, \vh)   \\
&  \quad + \overline {\langle \lambdah, \Vk \lambdah \rangle} + \langle \Wk \vexth, \vexth \rangle + \k^2 \langle \Vk \vexth , \vexth \rangle - \i\k \langle (\Kprimek + \Kk) \vexth, \vexth \rangle \\
& \quad  + \left[  \langle (\frac{1}{2} + \Kprimek) \lambdah, \vexth \rangle - \overline{\langle \lambdah, (\frac{1}{2} + \Kk)\vexth\rangle} \right] - \i\k^{-1}  \Vert  \delta^{1/2} \lambdah \Vert^2_{0,\Gamma}  \\
& \quad + \i\k \langle \Vk \lambdah, \vexth \rangle - \i\k   \overline{\langle \lambdah, \Vk \vexth \rangle }  -2\Re\left(  \langle \lambdah, \delta (\i\k)^{-1} \nablah \vh \cdot \nGamma \rangle    \right) -2 \i \, \Im \left(  \langle (1-\delta)\lambdah, \vh \rangle   \right) .
\end{split}
\]
For some~$\varepsilon>0$ to be fixed sufficiently small below, we take the $\Re + \varepsilon \Im$ part on both sides and get
\begin{equation} \label{long-identity}
\begin{split}
& (\Re + \varepsilon \Im)\Th ( (\vh, \lambdah, \vexth), (\vh, \lambdah, \vexth)) = \sum_{\E \in \taun}  (\Re + \varepsilon \Im) [\ahE(\vh, \vh) + \bhGamma(\vh, \vh)]    \\
&  + (\Re + \varepsilon \Im) [\overline {\langle \lambdah, \Vk \lambdah \rangle} + \langle \Wk \vexth, \vexth \rangle + \k^2 \langle \Vk \vexth , \vexth \rangle - \i\k \langle (\Kprimek + \Kk) \vexth, \vexth \rangle] \\
&  + (\Re + \varepsilon \Im) \left[  \langle (\frac{1}{2} + \Kprimek) \lambdah, \vexth \rangle \right]  - (\Re + \varepsilon \Im) \left[ \overline{\langle \lambdah, (\frac{1}{2} + \Kk)\vexth\rangle} \right] \\
& - \varepsilon \k^{-1}  \Vert \delta^{1/2}\lambdah \Vert^2_{0,\Gamma}  +(\Re + \varepsilon \Im) \left[ + \i\k \langle \Vk \lambdah, \vexth \rangle - \i\k   \overline{\langle \lambdah, \Vk \vexth \rangle} \right] \\
& -2\Re\left(  \langle \lambdah, \delta (\i\k)^{-1} \nablah \vh \cdot \nGamma \rangle    \right) -2\varepsilon  \Im \left(  \langle (1-\delta)\lambdah, \vh \rangle   \right)  \\
&  =: \sum_{\E \in \taun}  (\Re + \varepsilon \Im) [\ahE(\vh, \vh)  + \bhGamma(\vh, \vh)] + \sum_{i=1}^{10} T_i  . \\
\end{split}
\end{equation}
We deal with the terms~$T_i$, for~$i=1,\dots, 10$, separately.

The continuity of~$\Vk-\Vz: H^{-\frac{3}{2}}(\Gamma) \rightarrow H^{\frac{3}{2}}(\Gamma)$, see~\eqref{compact:part},
the fact that~$ \langle \lambdah, \Vz \lambdah \rangle$ is real, and~$\varepsilon \lesssim 1$ imply
\begin{equation} \label{T1}
\begin{split}
& T_1 := (\Re + \varepsilon \Im) (\overline{\langle \lambdah, \Vk \lambdah \rangle}) \\
& \; = \langle \lambdah, \Vz \lambdah \rangle + (\Re+\varepsilon\Im)(\overline{\langle \lambdah, (\Vk-\Vz) \lambdah \rangle}) \\
& \overset{\eqref{coercivity:operators:Laplace}}{\geq} c_{\Vz} \Vert \lambdah \Vert^2_{-\frac{1}{2}, \Gamma} - (1+\varepsilon)\Vert \lambdah \Vert_{-\frac{3}{2}, \Gamma} \Vert (\Vk - \Vz) \lambdah \Vert_{\frac{3}{2}, \Gamma} \geq c_{\Vz} \Vert \lambdah \Vert^2_{-\frac{1}{2},\Gamma} - c_1(\k)\Vert \lambdah \Vert^2_{-\frac{3}{2},\Gamma} . \\
\end{split}
\end{equation}
Analogously, the continuity of~$\Wk-\Wz: H^{-\frac{1}{2}}(\Gamma) \rightarrow H^{\frac{1}{2}}(\Gamma)$, see~\eqref{compact:part},
and the fact that~$\langle \Wz \vexth, \vexth \rangle$ is real and~$\varepsilon \lesssim 1$ imply
\begin{equation} \label{T2}
\begin{split}
&  T_2 :=  (\Re + \varepsilon\Im) (\langle \Wk \vexth , \vexth \rangle) \\
& \; =  \langle \Wz \vexth, \vexth \rangle + (\Re+\varepsilon\Im) (\langle (\Wk - \Wz) \vexth , \vexth \rangle )\\
& \overset{\eqref{coercivity:operators:Laplace}}{\ge} c_{\Wz} \vert \vexth \vert^2_{\frac{1}{2},\Gamma} - (1+\varepsilon) \Vert (\Wk - \Wz) \vexth \Vert_{\frac{1}{2},\Gamma}  \Vert \vexth \Vert _{-\frac{1}{2},\Gamma}
\ge c_{\Wz} \Vert \vexth \Vert^2_{\frac{1}{2}, \Gamma} - c_2(\k)\Vert \vexth \Vert^2_{-\frac{1}{2},\Gamma}.
\end{split}
\end{equation}
By the discussion after~\eqref{single:layer:operator} the operator $\Vk:H^{-\frac{1}{2}}(\Gamma)\rightarrow H^{\frac{1}{2}}(\Gamma)$ is continuous so that 
we get for $\varepsilon \in (0,1]$ 
\begin{align}
\label{T3} 
T_3 & : = \k^2  (\Re + \varepsilon \Im)  (\langle \Vk \vexth, \vexth \rangle) \ge - c_3(\k) \Vert\vexth \Vert^2_{-\frac{1}{2},\Gamma}.
\end{align}
Owing to the continuity of $\Kprimek: H^{-\frac{1}{2}}(\Gamma) \rightarrow H^{-\frac{1}{2}}(\Gamma)$ and $\Kk: H^{-\frac{1}{2}}(\Gamma) \rightarrow H^{-\frac{1}{2}}(\Gamma)$,
and~$\varepsilon \lesssim 1$, we note that
\begin{equation} \label{T4}
\begin{split}
T_4 	
& := - (\Re + \varepsilon\Im) (\i\k \langle \Kprimek \vexth, \vexth \rangle  + \i\k \langle \Kk \vexth, \vexth \rangle )\\
& \ge  -\k (1+\varepsilon) \Vert \Kprimek \vexth \Vert_{-\frac{1}{2},\Gamma} \Vert \vexth \Vert_{\frac{1}{2}, \Gamma} -\k (1+\varepsilon) \Vert \Kk \vexth \Vert_{-\frac{1}{2}, \Gamma} \Vert \vexth \Vert_{\frac{1}{2},\Gamma} \\
& \ge -\frac{c_{\Wz}}{5} \Vert \vexth \Vert ^2_{\frac{1}{2},\Gamma} - c \k^2 \Vert \Kprimek \vexth \Vert^2_{-\frac{1}{2}, \Gamma} -\frac{ c_{\Wz}}{5} \Vert \vexth \Vert ^2_{\frac{1}{2},\Gamma} - c \k^2 \Vert \Kk \vexth \Vert^2_{-\frac{1}{2}, \Gamma} \\
& \ge -\frac{ 2}{5} c_{\Wz} \Vert \vexth \Vert^2_{\frac{1}{2},\Gamma} - c_4 (\k) \Vert \vexth \Vert^2_{-\frac{1}{2}, \Gamma}.
\end{split}
\end{equation}
Next, we focus on the term~$T_5$. We observe that
\[
\begin{split}
T_5 	& : =  (\Re+ \varepsilon\Im) ( \langle (\frac{1}{2} + \Kprimek) \lambdah, \vexth \rangle - \overline{ \langle \lambdah, (\frac{1}{2} + \Kk)\vexth\rangle} ) \\
	& =  \frac{1}{2} (\Re+ \varepsilon\Im) ( \langle \lambdah , \vexth \rangle - \overline{\langle \lambdah , \vexth \rangle}) + (\Re+ \varepsilon\Im)  ( \langle \Kprimek \lambdah , \vexth \rangle  - \overline{\langle \lambdah , \Kk \vexth \rangle}   ) \\
	& =: T_{5,1} +  T_{5,2}.
\end{split}
\]
First, we focus on the term~$T_{5,1}$:
\begin{equation} \label{T51}
\begin{split}
T_{5,1} =  \varepsilon \Im ( \langle \lambdah , \vexth \rangle) \ge - \varepsilon \Vert \lambdah \Vert_{-\frac{1}{2}, \Gamma} \Vert \vexth \Vert_{\frac{1}{2}, \Gamma} \ge -  \frac{1}{2}  \varepsilon \left( \Vert \lambdah \Vert^2_{-\frac{1}{2}, \Gamma}+ \Vert \vexth \Vert^2_{\frac{1}{2}, \Gamma}\right).
\end{split}
\end{equation}
To show a bound on the term~$T_{5,2}$, we use~\cite[eqn.~(1.2)]{melenk2012mapping}, \eqref{compact:part}, and~$\varepsilon \lesssim 1$:
\begin{equation} \label{T52}
\begin{split}
T_{5,2}	& =  (\Re+\varepsilon \Im) ( \langle (\Kprimek - \Kprimez) \lambdah, \vexth \rangle - \overline{\langle \lambdah, (\Kk - \Kz)\vexth\rangle }    ) \\
		& \geq - (1+\varepsilon) \Vert (\Kprimek - \Kprimez) \lambdah \Vert_{\frac{1}{2}, \Gamma} \Vert \vexth \Vert_{-\frac{1}{2}, \Gamma} - \Vert \lambdah \Vert_{-\frac{3}{2}, \Gamma} \Vert (\Kk - \Kz) \vexth \Vert_{\frac{3}{2}, \Gamma}\\
		& \gtrsim - c_{5,2}(\k) \Vert \lambdah \Vert^2_{-\frac{3}{2}, \Gamma} - c_{5,2}(\k)  \Vert \vexth \Vert^2_{-\frac{1}{2}, \Gamma}.
\end{split}
\end{equation}
We show a bound on the term~$T_6$ using the polynomial inverse inequality of~\cite[Lemma~A.1]{MR3667020} with constant~$\cinvG$ and~\eqref{dG-parameters}:
\begin{equation} \label{T6}
T_6 := - \varepsilon\k^{-1}  \Vert \delta^{1/2} \lambdah \Vert_{0,\Gamma}^2 \ge  -  \varepsilon \cinvG \dfrak \Vert \lambdah \Vert_{-\frac{1}{2},\Gamma}^2.
\end{equation}
Using the continuity of~$\Vk : H^{-\frac{3}{2}}(\Gamma) \rightarrow H^{-\frac{1}{2}}(\Gamma)$ and~$\varepsilon \lesssim 1$, we get
\begin{equation} \label{T7}
\begin{split}
T_7 	& : = \k  (\Re + \varepsilon \Im) (\i \langle \Vk \lambdah, \vexth \rangle) \geq - (1+\varepsilon)\k \Vert \Vk \lambdah \Vert_{-\frac{1}{2},\Gamma} \Vert \vexth \Vert_{\frac{1}{2}, \Gamma} \\
	& \ge -\frac{1}{10}  c_{\Wz} \Vert \vexth \Vert^2_{\frac{1}{2}, \Gamma} - c_7 (\k) \Vert \lambdah \Vert^2_{-\frac{3}{2},\Gamma}.
\end{split}
\end{equation}
Besides, using the continuity of $\Vk : H^{\frac{1}{2}}(\Gamma) \rightarrow H^{\frac{3}{2}}(\Gamma)$, we prove that
\begin{equation} \label{T8}
\begin{split}
T_8 	& := - \k   (\Re + \varepsilon \Im) (\i \langle \lambdah, \Vk \vexth \rangle) \geq - (1+\varepsilon) \k \Vert \lambdah \Vert_{-\frac{3}{2},\Gamma} \Vert \Vk \vexth \Vert_{\frac{3}{2}, \Gamma} \\
	    & \ge -\frac{ c_{\Wz}}{5}  \Vert \vexth \Vert^2_{\frac{1}{2}, \Gamma} - c_8(\k) \Vert \lambdah \Vert^2_{-\frac{3}{2},\Gamma}. \\
\end{split}
\end{equation}
Next, we focus on the term~$T_{9}$. 
Using again the polynomial inverse inequality of~\cite[Lemma~A.1]{MR3667020}, the Young inequality with weight~$\varepsilon/4$, and~\eqref{dG-parameters},
we arrive at
\begin{equation} \label{T9}
\begin{split}
T_9 	& := -2\Re\left(  \langle \lambdah, \delta (\i\k)^{-1} \nablah \vh \cdot \nGamma \rangle    \right) \ge -  2\k^{-1/2} \Vert \delta^{1/2} \lambdah \Vert_{0,\Gamma}\, \k^{-1/2} \Vert \delta ^{1/2} \nablah \vh \cdot \nGamma \Vert_{0,\Gamma}\\
	& \ge -\frac{4}{\varepsilon} \k^{-1} \Vert \delta^{1/2} \lambdah \Vert^2_{0,\Gamma} - \frac{\varepsilon}{4} \k^{-1} \Vert \delta^{1/2} \nablah \vh \cdot \nGamma \Vert^2_{0,\Gamma} \\
	& \overset{\eqref{dG-parameters}}{\ge} -   \frac{4 \cinvG \dfrak}{\varepsilon}  \Vert \lambdah \Vert^2_{-\frac{1}{2}, \Gamma} - \frac{\varepsilon \k^{-1}}{4}  \Vert \delta ^{1/2} \nablah \vh \cdot \nGamma \Vert^2_{0,\Gamma}.\\
\end{split}
\end{equation}
As for the term~$T_{10}$, we proceed as follows. Recall that
\[
T_{10} := -2\varepsilon  \Im \left \langle (1-\delta) \lambdah, \vh \right \rangle .
\]
Let~$\Pkar$ be the operator introduced in Theorem~\ref{lemma:Karakashian-style}, with~$\ell=\p$.
Then, we use a trace inequality and again the polynomial inverse inequality of~\cite[Lemma~A.1]{MR3667020} to deduce
\[
\begin{split}
\langle \lambdah , \vh \rangle 
& = \langle \lambdah, \Pkar \vh \rangle + \langle \lambdah, (1-\Pkar) \vh    \rangle\\
& \le \Vert \lambdah \Vert_{-\frac{1}{2}, \Gamma} \Vert \Pkar \vh \Vert_{\frac{1}{2}, \Gamma} + \Vert  \msf^{1/2} \p^{-1}\lambdah \Vert_{0,\Gamma} \Vert \msf^{-1/2}\p  (I-\Pkar) \vh \Vert_{0,\Gamma}\\
& \overset{\mathclap{\eqref{K-propertyA}-\eqref{K-propertyC}}}{\lesssim} \Vert \lambdah\Vert_{-\frac{1}{2},\Gamma} \left(\Vert \nablah \vh\Vert_{0,\Omega}  +\Vert \vh\Vert_{0,\Omega} + \Vert\msf^{-1/2}\p\llbracket \vh\rrbracket \Vert_{0,\FhI} \right)\\
&  \lesssim \Vert \lambdah \Vert_{-\frac{1}{2}, \Gamma} \left( \numin^{-1/2} \Vert \nu^{1/2} \nablah \vh \Vert_{0,\Omega}+ (k_0c_0)^{-1} \Vert \kn\,\vh \Vert_{0,\Omega} 
       +\numin^{-1/2} \k^{1/2} \afrak^{-1/2} \Vert \nu^{1/2}\alpha^{1/2} \llbracket \vh \rrbracket \Vert_{0,\FhI} \right),
\end{split}
\]
where the last inequality follows from the bounds~$\nu\ge \numin$ and~$|kn|\ge k_0c_0$, and from
the definition of~$\alpha$ in~\eqref{dG-parameters}.

Let~$\epsilontilde>0$ be a positive constant, which will be fixed below; see~\eqref{various:epsilon:2}.
The Young inequality gives
\begin{equation} \label{T10}
\begin{split}
&  -2\varepsilon  \Im (\langle (1-\delta) \lambdah , \vh \rangle ) \geq -\epsilontilde^{-1} \varepsilon \Vert \lambdah \Vert_{-\frac{1}{2}, \Gamma}^2\\
& \quad\quad\quad\quad -c_{10} \epsilontilde \varepsilon (\numin^{-1}\Vert \nu^{1/2}\nablah \vh \Vert_{0,\Omega}^2  +  (k_0c_0)^{-2} \Vert \kn\, \vh\Vert_{0,\Omega} ^2+ \numin^{-1}\k \Vert \nu^{1/2} \alpha^{1/2} \llbracket \vh \rrbracket \Vert_{0,\FhI}^2 ),
\end{split}
\end{equation}
where $c_{10}$ depends on $\afrak^{-1/2}$.
Provided that $\afrak$ is sufficiently large and $\dfrak$, $\bfrak$ are sufficiently small, depending on $\varepsilon$,
we insert~\eqref{coercivity:Omega:weakened}, \eqref{T1}, \eqref{T2}, \eqref{T3}, \eqref{T4}, \eqref{T51}, \eqref{T52}, \eqref{T7}, \eqref{T8}, \eqref{T9}, and~\eqref{T10} into~\eqref{long-identity}, and arrive at 
\[
\begin{split}
& (\Re + \varepsilon \Im)\Th ( (\vh, \lambdah, \vexth), (\vh, \lambdah, \vexth))  \\
& \geq c_{\Vz} \Vert \lambdah \Vert^2_{-\frac{1}{2}, \Gamma} - {c}_1(\k) \Vert \lambdah \Vert^2_{-\frac{3}{2}, \Gamma} + c_{\Wz} \Vert \vexth \Vert^2_{\frac{1}{2}, \Gamma} - c_2(\k) \Vert \vexth \Vert^2_{-\frac{1}{2},\Gamma} - c_3(\k) \Vert \vexth \Vert^2_{-\frac{1}{2}, \Gamma}\\
&  \quad -\frac{2}{5}  c_{\Wz} \Vert \vexth \Vert^2_{\frac{1}{2}, \Gamma} - c_4(\k) \Vert \vexth \Vert^2_{-\frac{1}{2}, \Gamma} -\frac{\varepsilon}{2} \Vert \lambdah \Vert^2_{-\frac{1}{2},\Gamma} - \frac{\varepsilon}{2} \Vert \vexth \Vert^2_{\frac{1}{2}, \Gamma} \\
& \quad- c_{5,2}(\k)\Vert \lambdah \Vert^2_{-\frac{3}{2},\Gamma}  - c_{5,2}(\k)\Vert \vexth \Vert^2_{-\frac{1}{2}, \Gamma} - \varepsilon  \cinvG \dfrak \Vert \lambdah \Vert^2_{-\frac{1}{2},\Gamma}\\
& \quad -\frac{1}{10} c_{\Wz} \Vert \vexth \Vert^2_{\frac{1}{2}, \Gamma} - c_7 (\k) \Vert \lambdah \Vert^2_{-\frac{3}{2}, \Gamma} - \frac{1}{5} c_{\Wz} \Vert \vexth \Vert^2_{\frac{1}{2}, \Gamma} - c_8(\k) \Vert \lambdah \Vert^2_{-\frac{3}{2},\Gamma}\\
& \quad - \cinvG \dfrak 4 \varepsilon^{-1} \Vert \lambdah
\Vert^2_{-\frac{1}{2}, \Gamma} -\epsilontilde^{-1}\varepsilon \Vert
\lambdah \Vert^2_{-\frac{1}{2}, \Gamma}\\
& \quad- c_{10}  \epsilontilde \varepsilon \left( \numin^{-1}\Vert \nu^{1/2}\nablah \vh
  \Vert_{0,\Omega}^2  +  (k_0c_0)^{-2} \Vert \kn\, \vh
  \Vert_{0,\Omega} ^2+ \numin^{-1} \k \Vert \nu^{1/2}\alpha^{1/2} \llbracket \vh \rrbracket \Vert_{0,\FhI}
  ^2 \right) \\
& \quad + \frac{1}{2}\Vert \Ad^{1/2}\nabla _\h \vh \Vert^2_{0,\Omega} - \Vert \kn\,  \vh \Vert^2_{0,\Omega} + \frac{1}{2} \left( \varepsilon \k^{-1} \Vert \Ad^{1/2}\beta^{1/2} \llbracket \nablah \vh \rrbracket \Vert^2_{0,\FhI} + \varepsilon \k \Vert \Ad^{1/2}\alpha^{1/2} \llbracket \vh \rrbracket \Vert^2_{0,\FhI} \right) \\
& \quad + \frac{1}{2}\left( \varepsilon \k^{-1} \Vert \delta^{1/2} \nablah \vh \cdot \nGamma \Vert^2_{0,\Gamma}  +  \varepsilon \k \Vert (1-\delta) \vh \Vert^2_{0,\Gamma}\right) - \frac{\varepsilon}{4} \k^{-1} \Vert \delta^{1/2} \nablah \vh \cdot \nGamma \Vert^2_{0,\Gamma} .
\end{split}
\]
Simple computations yield
\[
\begin{split}
& (\Re + \varepsilon \Im)\Th ( (\vh, \lambdah, \vexth), (\vh, \lambdah, \vexth))  \\
& \geq (1/2 - c_{10} \epsilontilde
\varepsilon\numin^{-1}) \Vert \Ad^{1/2} \nablah \vh \Vert
^2_{0,\Omega} - (c_{10}\epsilontilde
\varepsilon(k_0c_0)^{-2} + 1
)
\Vert \kn\,\vh \Vert^2_{0,\Omega} + \k^{-1} \varepsilon/4 \Vert \delta ^{1/2} \nablah \vh \cdot \nGamma \Vert^2_{0,\Gamma}\\
& \quad + \varepsilon \k(1/2 -c_{10}\epsilontilde \numin^{-1}) \Vert \Ad^{1/2}\alpha^{1/2} \llbracket \vh \rrbracket \Vert^2_{0,\FhI} + \frac{\varepsilon}{2\k} \Vert \Ad^{1/2}\beta^{1/2} \llbracket \nablah\vh \rrbracket \Vert^2_{0,\FhI} + \frac{1}{2} \varepsilon \k \Vert (1-\delta) \vh \Vert^2_{0,\Gamma}\\
& \quad + \left( c_{\Vz} - \frac{\varepsilon}{2} - \varepsilon \cinvG \dfrak -  4 \cinvG \dfrak \varepsilon^{-1} - \epsilontilde^{-1} \varepsilon \right) \Vert \lambdah \Vert^2_{-\frac{1}{2}, \Gamma}  - \left( c_1(\k) +c_{5,2}(\k)+ c_7(\k) + c_8 (\k) \right) \Vert \lambdah \Vert^2_{-\frac{3}{2}, \Gamma}\\
& \quad + \left( \frac{3}{10} c_{\Wz}- \frac{\varepsilon}{2} \right) \Vert \vexth \Vert^2_{\frac{1}{2}, \Gamma} - (c_2(\k) + c_3(\k) + c_4(\k) + c_{5,2}(\k)) \Vert \vexth \Vert^2_{-\frac{1}{2}, \Gamma}.
\end{split}
\]
We select
\begin{equation} \label{various:epsilon:2}
\epsilontilde :=\frac{ \numin}{4c_{10}},
\end{equation}
and fix $\varepsilon$ as 
\begin{equation} \label{various:epsilon:1}
\varepsilon := \min\left\{ \frac{c_{\Vz}}{3(1/2+ \cinvG + 4 c_{10} \numin^{-1})} ,\frac{c_{\Wz}}{10},1 \right\},
\end{equation}
where we recall that the constants~$c_{\Vz}$ and~$c_{\Wz}$ are
from~\eqref{coercivity:operators:Laplace},~$\numin$ is a lower bound
of the coefficient~$\Ad$ (see
Section~\ref{subsection:model:problem}),~$\cinvG$ is the inverse
inequality constant in~\eqref{T6}, and~$c_{10}$ is from~\eqref{T10}.

Using~\eqref{various:epsilon:2} and~\eqref{various:epsilon:1}, we investigate the constants of the terms
appearing in the DG norm: 
\begin{itemize}
\item[*] $ (1/2 - c_{10} \epsilontilde \varepsilon \numin^{-1}) \ge 1/2-\varepsilon/4>1/4$;
\item[*] $- (c_{10}\epsilontilde \varepsilon(k_0c_0)^{-2} + 1)  \ge -(\frac{1}{4}\numin\varepsilon (k_0c_0)^{-2}+1)$;
\item[*] $\varepsilon \k(1/2 -c_{10}\epsilontilde  \numin^{-1})=\varepsilon k/4$;
\item[*] by taking~$\mathfrak d$ in~\eqref{dG-parameters} also fulfilling
\[
\mathfrak{d} \leq \mathfrak{d}_0\leq\frac{c_{\Vz}}{12 \cinvG} \varepsilon ,
\]
we also have
\[
\begin{split}
\left( c_{\Vz} - \frac{\varepsilon}{2} -4  \cinvG \dfrak 
- \cinvG \dfrak \varepsilon^{-1} -\epsilontilde^{-1} \varepsilon \right)
& =  c_{\Vz}- \varepsilon \left(\frac{1}{2} + \cinvG \dfrak + 4 c_{10}\numin^{-1} \right) -  4\cinvG \dfrak \varepsilon^{-1}\\
& \ge2 c_{\Vz}/3 -4\cinvG \dfrak \varepsilon^{-1} = c_{\Vz}/3;
\end{split}
\]
this term is positive as well;
\item[*] $\left( \frac{3}{10} c_{\Wz}- \varepsilon \right)>c_{\Wz}/5$.
\end{itemize}
The assertion follows.
\end{proof}

\section{Continuity of~$\Th( (\cdot,\cdot,\cdot), (\cdot,\cdot,\cdot))$}\label{section:continuity}
In this section, we prove the continuity of~$\Th((\cdot,\cdot,\cdot), (\cdot,\cdot,\cdot))$. To that end, we introduce the two following energy norms, which extend the $\DG$ and $\DGp$ norms to the \DGFEMBEM{} coupling:
\begin{align*}
\energynorm{(u,\m,\uext)}^2   &:= \|{u}\|_{\DG}^2 + \|\m\|_{-\frac{1}{2},\Gamma}^2 + \|\uext\|_{\frac{1}{2},\Gamma}^2, \\
\energynormp{(u,\m,\uext)}^2 &:= \|{u}\|_{\DGp}^2 + \|\m\|_{-\frac{1}{2},\Gamma}^2 + \| \msf^{1/2}\p^{-1} \ \m\Vert_{0,\Gamma}^2 + \|\uext\|_{\frac{1}{2},\Gamma}^2.
\end{align*}

\begin{prop} \label{proposition:continuity:trisesquilinear-form}
For all $(u,\m,\uext), (v,\lambda,\vext) \in \Hpw^{\frac{3}{2} +\tepsilon} (\Omegah)\times  L^2(\Gamma) \times H^{\frac{1}{2}}(\Gamma)$ for some regularity parameter $\tepsilon>0$, the following continuity bound is valid:
\begin{align} \label{eq:continuity}
\big|\Th ( (u,\m,\uext) , (v,\lambda, \vext) )\big| 
  &\lesssim \energynormp{(u,\m,\uext)}  \energynormp{(v,\lambda,\vext)} ,
\end{align}
where the hidden constant depends on $\k$.
If $(u,\m,\uext)$ or $(v,\lambda,\vext)$ is in $\Vh \times \Wh \times \Zh$,
then we can replace the corresponding $\energynormp{\cdot}$ norm in~\eqref{eq:continuity} with $\energynorm{\cdot}$.
\end{prop}
\begin{proof}
We present the estimates of the terms in the sesquilinear form $\Th ( (u,\m,\uext) , (v,\lambda, \vext) )$ defined in~\eqref{trisesquilinear-form} separately.

First, to estimate the term $\sum_{\E \in \taun} \ahE (u,v) + \bhGamma(u,v)$, we use~\eqref{continuity:abplusplus}.

For the terms involving the integral operators, 
we use the definitions of the combined integral operators in~\eqref{definition:Bk:and:Akprime}, and
the mapping properties described in Section~\ref{subsection:boundary-integral-operators}.
More precisely, we write
\begin{multline*}
  \big| - \langle (-\Wk -\i\k (\frac{1}{2}-\Kk) + \i\k (\frac{1}{2}+\Kprimek+\i\k\Vk))\uext - (\frac{1}{2} + \Kprimek + \i\k \Vk)\m, \vext \rangle\big|\\
  \begin{aligned}[t]
& \leq\Vert \Bk \uext \Vert_{-\frac{1}{2}, \Gamma} \Vert \vext\Vert_{\frac{1}{2}, \Gamma}+\k\Vert \Aprimek \uext \Vert_{-\frac{1}{2}, \Gamma} \Vert \vext\Vert_{\frac{1}{2}, \Gamma} + \Vert \Aprimek \m \Vert_{-\frac{1}{2}, \Gamma}\Vert \vext\Vert_{\frac{1}{2}, \Gamma}\\
&\lesssim \Vert \uext\Vert_{\frac{1}{2}, \Gamma}\Vert \vext\Vert_{\frac{1}{2}, \Gamma}
+\Vert \m\Vert_{-\frac{1}{2}, \Gamma}\Vert \vext\Vert_{\frac{1}{2}, \Gamma},
\end{aligned}
\end{multline*}
where we have used $\Vert \uext\Vert_{-\frac{1}{2}, \Gamma}\leq \Vert \uext\Vert_{\frac{1}{2}, \Gamma}$ and
\begin{align*}
\big|- \langle (\frac{1}{2} +\Kk) \uext - \Vk (\m - \i\k \uext) , \lambda   \rangle\big|
&\leq \Vert (\frac{1}{2} +\Kk) \uext - \Vk (\m - \i\k \uext)\Vert_{\frac{1}{2},\Gamma} \Vert\lambda\Vert_{-\frac{1}{2}, \Gamma}\\
&\lesssim \big(\Vert \uext\Vert_{\frac{1}{2}, \Gamma}+\Vert \m\Vert_{-\frac{1}{2}, \Gamma}\big)\Vert\lambda\Vert_{-\frac{1}{2},\Gamma}.
\end{align*}

Next, we focus on the coupling terms. Several of the following estimates are already established in the proof of 
Theorem~\ref{theorem:discrete-Garding}.
However, we cannot use the polynomial inverse inequality here.
With the Cauchy-Schwarz inequality and the definition of $\delta$ in~\eqref{dG-parameters}, we get
\begin{align*}
\big|  (\m, \delta (\i\k)^{-1} \nablah v \cdot \nGamma)_{0,\Gamma} \big|
& \le \k^{-1} \Vert \delta^{1/2}\m\Vert_{0,\Gamma} \Vert \delta^{1/2}\nablah v \cdot \nGamma\Vert_{0,\Gamma} \\
  &\leq \dfrak^{1/2}  \Vert \msf^{1/2}\p^{-1}\ \m\Vert_{0,\Gamma} \k^{-1/2} \Vert \delta^{1/2}\nablah v \cdot \nGamma\Vert_{0,\Gamma}
\end{align*}
The next coupling term is dealt with as follows:
\begin{align*}
 \big|\langle -\delta (\i\k)^{-1} \nablah u \cdot \nGamma, \lambda \rangle\big|
 &\leq \k^{-1}\Vert \delta^{1/2}\nablah u \cdot
   \nGamma\Vert_{0,\Gamma} \Vert \delta^{1/2}\lambda\Vert_{0,\Gamma}\\
  &\leq  \k^{-1/2}\Vert \delta^{1/2}\nablah u \cdot \nGamma\Vert_{0,\Gamma}\dfrak^{1/2}\Vert  \msf^{1/2}\p^{-1}\ \lambda\Vert_{0,\Gamma}.
\end{align*}
Furthermore, we get
\begin{align*}
 \big|\langle \delta (\i\k)^{-1} \m, \lambda \rangle\big|
 &\leq \k^{-1} \Vert \delta^{1/2}m\Vert_{0,\Gamma} \Vert \delta^{1/2}\lambda\Vert_{0,\Gamma} \leq \dfrak^{1/2} \Vert  \msf^{1/2}\p^{-1}\ \m\Vert_{0,\Gamma}\dfrak^{1/2} \Vert  \msf^{1/2}\p^{-1}\ \lambda\Vert_{0,\Gamma}.
\end{align*}
As for the two remaining coupling terms, we employ the reconstruction
operator $\Pkar:\Hpw^1(\Omegah)\to H^1(\Omega)$ introduced in Theorem~\ref{lemma:Karakashian-style}, and write
\begin{align*}
\big|- \langle \m, (1-\delta)v \rangle  \big|
&\lesssim \vert \langle \m, \Pkar v \rangle \vert + \vert \langle \m, (1-\Pkar) v \rangle \vert \\
& \le \Vert \m \Vert_{-\frac{1}{2}, \Gamma} \Vert \Pkar v \Vert_{\frac{1}{2}, \Gamma} + \Vert  \msf^{1/2} \p^{-1}\m \Vert_{0,\Gamma} \Vert \msf^{-1/2}\p  (I-\Pkar) v \Vert_{0,\Gamma}.
\end{align*}
Properties~\eqref{K-propertyA}--\eqref{K-propertyC} with $\ell=p$, 
$\msf p^{-2}\leq 1$, and the definition of $\alpha$ in~\eqref{dG-parameters} lead to
 \begin{align*}
\big|  \langle\m,
& (1-\delta)v \rangle \big| \lesssim  \Vert \m\Vert_{-\frac{1}{2},\Gamma} (\Vert \nablah v\Vert_{0,\Omega}   +\Vert\msf^{-1/2}p\llbracket v\rrbracket \Vert_{0,\FhI} +\Vert v\Vert_{0,\Omega})\\
&\quad +\Vert  \msf^{1/2} \p^{-1}\m \Vert_{0,\Gamma} (\Vert \nablah v\Vert_{0,\Omega} + \Vert\msf^{-1/2}p\llbracket v\rrbracket \Vert_{0,\FhI})\\
&\lesssim  \Vert \m\Vert_{-\frac{1}{2},\Gamma} (\numin^{-1/2}\Vert \nu^{1/2}\nablah v\Vert_{0,\Omega}
+ \numin^{-1/2}\k^{1/2}\Vert\nu^{1/2}\alpha^{1/2}\llbracket v \rrbracket \Vert_{0,\FhI}+ (k_0c_0)^{-1}\Vert \kn v\Vert_{0,\Omega}  )\\
&\quad +\Vert  \msf^{1/2} \p^{-1}\m \Vert_{0,\Gamma}( \numin^{-1/2}\Vert \nu^{1/2}\nablah v\Vert_{0,\Omega}    +\numin^{-1/2}\k^{1/2}\Vert\nu^{1/2}\alpha^{1/2}\llbracket v\rrbracket \Vert_{0,\FhI}),
 \end{align*}
where the last inequality follows from the bounds $\nu\ge \numin$ and~$|kn|\ge k_0c_0$.
The hidden constant depends additionally on~$\afrak^{-\frac{1}{2} }$. 

We proceed in the same way to estimate the term $ \big| \langle (1 - \delta) u, \lambda \rangle \big|$, and
the assertion follows combining the above bounds.

When dealing with discrete functions, estimate~\eqref{eq:continuity}
can be improved using the polynomial inverse inequality of~\cite[Lemma~A.1]{MR3667020}:
\[
\Vert  \msf^{1/2} \p^{-1}\lambda_h \Vert_{0,\Gamma}\lesssim \Vert\lambda_h \Vert_{-\frac{1}{2} ,\Gamma} \qquad \forall \lambda_h\in \Wh .
\]
Thus, we can replace~$\energynormp{\cdot}$ by~$\energynorm{\cdot}$ for discrete functions.
\end{proof}

\section{Adjoint problem}\label{section:adjoint}

In this section, we introduce and analyze the adjoint problem of~\eqref{dGBEM:long-version}.

The dual problem to~\eqref{weak:formulation:mortar:Helmholtz} is:
given~$(\rr, \rm, \rext )\in L^2(\Omega) \times H^{-\frac{3}{2}}(\Gamma) \times H^{-\frac{1}{2}}(\Gamma)$,
\begin{equation} \label{dual:problem:2}
\begin{cases}
\text{find } (\psi, \psim, \psitilde) \in H^1(\Omega)\times H^{-\frac{1}{2}}(\Gamma) \times H^{\frac{1}{2}}(\Gamma) \text{ such that} \\
(\Ad\nabla v, \nabla \psi )_{0,\Omega} -  ((\kn)^2 \, v, \psi)_{0,\Omega} + \i\k( v, \psi)_{0,\Gamma} - \langle \lambda, \psi \rangle \\
\quad -\langle  ( \Bk + i \k \Aprimek) \vtilde -  \Aprimek \lambda, \psitilde   \rangle + \overline{ \langle \psim, v \rangle} - \overline{\langle \psim, (\frac{1}{2} + \Kk) \vtilde - \Vk (\lambda - i\k \vtilde)   \rangle} \\
=  \left ((w,\rr)_{0,\Omega} +  (\xi, \rm)_{-\frac{3}{2},\Gamma} +  (\wext, \rext)_{-\frac{1}{2},\Gamma} \right)\\
\quad \quad \quad\forall (v,\lambda,\vtilde) \in H^1(\Omega)\times H^{-\frac{1}{2}}(\Gamma) \times H^{\frac{1}{2}}(\Gamma).\\
\end{cases}
\end{equation}

We recall some technical results from~\cite{FEMBEM:mortar}.

\begin{lem} (\cite[Lemma~{3.6}]{FEMBEM:mortar}) \label{lemma:devil}
The following identities are valid: For all $\varphi\in H^{-\frac{1}{2}}(\Gamma)$ and for all~$\psi \in H^{\frac{1}{2}}(\Gamma)$,
\begin{align}
& \Vk^* \varphi = \overline {\Vk \overline \varphi}, \quad \Kk^* \psi = \overline{\Kk' \overline \psi}, \label{adjoint:of:blo:1}\\
& (\Kk')^* \varphi = \overline{\Kk \overline \varphi}, \quad \Wk^* \psi= \overline {\Wk \overline \psi}, \label{adjoint:of:blo:2}
\end{align}
where we recall that $\cdot^*$ denotes the adjoint operator.
\end{lem}


\begin{lem}(\cite[Lemma~{3.7}]{FEMBEM:mortar}) \label{lemma:Laplace:Beltrami}
Let $s \in {\mathbb R}^+$. 
Given~$\rm \in H^{s-\frac{3}{2}}(\Gamma)$ and~$\rext \in H^{s-\frac{1}{2}}(\Gamma)$,
there exist~$\Rm\in H^{s+\frac{3}{2}}(\Gamma)$ and~$\Rext \in H^{s+\frac{1}{2}} (\Gamma)$ such that
\[
\Vert \Rm \Vert_{s+\frac{3}{2}, \Gamma} = \Vert \rm \Vert_{s-\frac{3}{2}, \Gamma},\quad \Vert \Rext \Vert_{s+\frac{1}{2}, \Gamma} = \Vert \rext \Vert_{s-\frac{1}{2}, \Gamma}
\]
and
\begin{align}
\langle \xi, \Rm \rangle			& = (\xi, \rm)_{-\frac{3}{2}, \Gamma} 		&& \quad \forall \xi \in H^{-\frac{1}{2}}(\Gamma), \label{regularity:Melenk:trick} \\
(\wext , \Rext ) _{0,\Gamma}	& = (\wext,\rext)_{-\frac{1}{2}, \Gamma} 	&& \quad \forall \wext \in L^2(\Gamma).\notag
\end{align}
\end{lem}

Indeed, the global problem~\eqref{dual:problem:2} can be split into three problems as detailed in the following result.
\begin{lem}(\cite[Lemma~{3.8}]{FEMBEM:mortar}) \label{lemma:rewriting:dual:problem:strong:formulation}
Let $(\rr, \rm, \rext )\in L^2(\Omega) \times H^{-\frac{3}{2}}(\Gamma) \times H^{-\frac{1}{2}}(\Gamma)$ and~$\Rm$ and~$\Rext$ be the representers of~$\rm$ and~$\rext$ constructed 
in Lemma~\ref{lemma:Laplace:Beltrami}.  
Then, problem~\eqref{dual:problem:2} is equivalent to the variational
formulation of the following three coupled problems:
Find~$(\psi, \psim, \psitilde) \in H^1(\Omega) \times H^{-\frac{1}{2}}(\Gamma) \times H^{\frac{1}{2}}(\Gamma)$ such that
\begin{align}
&\begin{cases}
-\div(\Ad \nabla\overline  \psi) - (\kn)^2 \overline \psi =  \overline r 				& \text{in }\Omega, \\
\nabla \overline\psi \cdot \nGamma + i\k \overline\psi + \overline{\psim} = 0 	& \text{on }\Gamma,\\
\end{cases}\label{dual:problem:first:equation:strong}\\
& \begin{cases}  -\overline \psi + (\frac{1}{2} + \Kk +i\k \Vk) \overline{\psitilde} + \Vk \overline{\psim} = \overline \Rm \qquad\text{on }\Gamma,
\end{cases}\label{dual:problem:second:equation:strong} \\
& \begin{cases} (\Wk + i\k (\frac{1}{2} - \Kprimek) - i\k (\frac{1}{2} + \Kk + i\k\Vk)) \overline{\psitilde} - ( (\frac{1}{2} + \Kprimek) + i \k \Vk)  \overline{\psim} = \overline{\Rext} \qquad\text{on }\Gamma.
\end{cases}\label{dual:problem:third:equation:strong:3}
\end{align}
\end{lem}

Well posedness as well as regularity results for problem~\eqref{dual:problem:2} are given in the following theorem. 
\begin{thm}(\cite[Thm.~{3.12}]{FEMBEM:mortar}) \label{theorem:regularity:duality}
Given~$s\in \mathbb R^+_{0}$ and
\[
\rr \in H^{s}(\Omega),\quad\quad \rm \in H^{s - \frac{3}{2}}(\Gamma), \quad\quad \rext \in H^{s-\frac{1}{2}}(\Gamma),
\]
let~$(\psi,\psim, \psitilde)$ be the solution to~\eqref{dual:problem:first:equation:strong}--\eqref{dual:problem:third:equation:strong:3}.
Then, $(\psi,\psim, \psitilde)$ satisfies
\[
\psi \in H^{s+2}(\Omega), \quad\quad \psim \in H^{s+\frac{1}{2}}(\Gamma), \quad\quad \psitilde \in H^{s+\frac{3}{2}}(\Gamma),
\]
together with the {\sl a priori} estimates
\begin{equation}\label{bounds:norms:solutions:dual}
\begin{split}
& \Vert \psi \Vert_{s+2,\Omega} +\Vert \psim \Vert_{s+\frac{1}{2}, \Gamma} + \Vert \psitilde \Vert_{s+\frac{3}{2}, \Gamma} \lesssim_{\k} \left(\Vert \rr \Vert_{s,\Omega} + \Vert \rm \Vert_{s - \frac{3}{2}, \Gamma} + \Vert \rext \Vert_{s-\frac{1}{2}, \Gamma} \right).  \\
\end{split}
\end{equation}
\end{thm}

In the next proposition, we prove that the adjoint formulation
of~\eqref{dGBEM:long-version} is in fact an approximation of the
adjoint problem~\eqref{dual:problem:2}, i.e., of the coupled problems~\eqref{dual:problem:first:equation:strong}--\eqref{dual:problem:third:equation:strong:3}.
\begin{prop}[adjoint consistency]\label{lemma:adjointconsistency}
Let the right-hand side  $(\rr, \rm, \rext )$ of~\eqref{dual:problem:2} belong to
$L^2(\Omega) \times H^{-\frac{3}{2}}(\Gamma) \times H^{-\frac{1}{2}}(\Gamma)$.
Then, the solution $(\psi, \psim, \psitilde)$ of~\eqref{dual:problem:2} belongs to $H^{2}(\Omega)\times  H^{\frac{1}{2}}(\Gamma) \times H^{\frac{3}{2}}(\Gamma)$ and satisfies
\begin{align}  \label{eq:adjointconsistency}
  \begin{split}
  &\Th ( (w,\xi,\wext) , (\psi,\psim,\psitilde))=(w,\rr)+(\xi,\rm)_{-\frac{3}{2},\Gamma}+(\wext,\rext)_{-\frac{1}{2},\Gamma} 
  \end{split}
\end{align}
for all $(w,\xi,\wext)\in \Hpw^{\frac{3}{2} +\tepsilon}(\Omegah)\times  L^2(\Gamma) \times
 H^{\frac{1}{2}}(\Gamma)$, for some $\tepsilon>0$.
\end{prop}
\begin{proof}
 According to Theorem~\ref{theorem:regularity:duality}, $(\psi, \psim, \psitilde)$ belongs to $H^{2}(\Omega)\times H^{\frac{1}{2}}(\Gamma) \times H^{\frac{3}{2}}(\Gamma)$.\\
\textbf{STEP 1:} \emph{$(\psi, \psim, \psitilde)$ satisfies  $\Th( (w,0,0), (\psi,\psim,\psitilde))=(w,\rr)_{0,\Omega}$ for all~$w \in \Hpw^{\frac32+\tepsilon} (\Omegah)$.}\\
Since $\psi\in H^{2}(\Omega)$, on each internal face we have
\begin{align}
 \label{eq:jumpmeanpsizero}
 \llbracket \psi \rrbracket=0, \quad \llbracket \nabla \psi\rrbracket=0, \quad\ldc\nabla \psi\rdc=\nabla\psi.
\end{align} 
We multiply the first equation in~\eqref{dual:problem:first:equation:strong} by
$w\in\Hpw^{\frac{3}{2} +\tepsilon}(\Omegah)$ and integrate by parts
elementwise 
to get
\begin{align} \label{eq:adjointconsist1}
\sum_{\E \in \taun} \left ( -\int_{\partial \E} \Ad w\overline{\nabla \psi \cdot \nGamma} +\int_{\E}\Ad\,\nabla w \cdot \overline{\nabla \psi}\right)  -\int_\Omega  (\kn)^2 w \overline{\psi} =\int_\Omega w\overline{\rr}.
\end{align}
With the aid of the boundary condition
in~\eqref{dual:problem:first:equation:strong}, the definition of the parameter~$\delta$ in~\eqref{compact:part}, and the fact that $\Ad=1$ on $\Gamma$,
we manipulate the boundary term in~\eqref{eq:adjointconsist1} as follows:
\begin{align*}
-\sum_{\E \in \taun}\int_{\partial \E} \Ad w
& \overline{\nabla \psi \cdot \nGamma} =-\int_{\FhI}\Ad \llbracket w\rrbracket\cdot \overline{\nabla \psi}-\int_{\FhB} w \overline{\nabla \psi \cdot \nGamma}  \\
& = -\int_{\FhI}\Ad \llbracket w\rrbracket\cdot \overline{\nabla \psi} +\int_{\FhB}\i\k w\overline{\psi}+\int_{\FhB}w\overline{\psim} -\int_{\FhB}\delta w\overline{\nabla \psi\cdot\nGamma} \\
& \quad - \int_{\FhB}\delta\i\k w\overline{\psi} - \int_{\FhB}\delta w\overline{\psim} -\int_{\FhB} \delta (\i\k)^{-1}\nabla w\cdot\nGamma \overline{\nabla \psi\cdot\nGamma}\\
& \quad -\int_{\FhB}\delta \nabla w\cdot\nGamma\overline{\psi}  - \int_{\FhB}\delta (\i\k)^{-1} \nabla w\cdot\nGamma\overline{\psim}.
\end{align*}
Inserting the above identity into~\eqref{eq:adjointconsist1} and adding some terms with property~\eqref{eq:jumpmeanpsizero}, we see that STEP~1 is valid.\\
\smallskip

\noindent \textbf{STEP 2:} \emph{$(\psi, \psim, \psitilde)$ satisfies $\Th ( (0,\xi,0) , (\psi,\psim,\psitilde))=(\xi,\rm)_{-\frac{3}{2},\Gamma}$ for all~$\xi\in L^2(\Gamma)$}.\\
First, we multiply~\eqref{dual:problem:second:equation:strong} by $\xi\in L^2(\Gamma)$:
\begin{align} \label{eq:adjointconsist2}
-\int_\Gamma \xi \overline{\psi} +\int_\Gamma \xi (1/2+\Kk+\i\k\Vk)\overline{\psitilde} +\int_\Gamma \xi\Vk\overline{\psim} =\int_\Gamma \xi \overline{\Rm}.
\end{align}
Identities~\eqref{adjoint:of:blo:1} and~\eqref{adjoint:of:blo:2} lead to
\begin{align*}
\int_\Gamma 
& \xi (1/2+\Kk+\i\k\Vk)\overline{\psitilde} =\langle \xi,\overline{(1/2+\Kk\overline{\psitilde}}\rangle +\i\k\langle \xi,\overline{\Vk\overline{\psitilde}}\rangle\\
& = \langle \xi, (1/2+\Kk)^* \psitilde\rangle +\i\k\langle \xi,\Vk^*\psitilde\rangle =\langle (1/2+\Kk'+\i\k\Vk)\xi,\psitilde\rangle
\end{align*}
and
\begin{align*}
 \int_\Gamma \xi\Vk\overline{\psim}=
 \langle \xi,\overline{\Vk\overline{\psim}}\rangle=\langle \xi,\Vk^*\psim\rangle=\langle \Vk\xih,\psim\rangle.
\end{align*}
Inserting these two terms into~\eqref{eq:adjointconsist2} and adding
the boundary condition in~\eqref{dual:problem:first:equation:strong}
with the parameter~$\delta$, yields STEP~2.
To deal with the right-hand side of~\eqref{eq:adjointconsist2}, we have used~\eqref{regularity:Melenk:trick}.\\
\smallskip

\noindent \textbf{STEP 3:} \emph{$(\psi, \psim, \psitilde)$ satisfies  $\Th ( (0,0,\wext) , (\psi,\psim,\psitilde))=(\wext,\rext)_{-\frac{1}{2},\Gamma}$ for all~$\wext\in H^{\frac{1}{2}}(\Gamma)$}.\\
We multiply~\eqref{dual:problem:third:equation:strong:3} by $\wext\in H^{\frac{1}{2}}(\Gamma)$:
\[
\begin{split}
& \int_\Gamma \wext (\Wk+\i\k(1/2-\Kk')-\i\k(1/2+\Kk+\i\k\Vk))\overline{\psitilde} +\int_\Gamma\wext ((1/2+\Kk')+\i\k\Vk)\overline{\psim}\\
&=\int_\Gamma \wext\overline{\Rext}.
\end{split}
\]
We use again identities~\eqref{adjoint:of:blo:1} and~\eqref{adjoint:of:blo:2} and write
\begin{align*}
&\int_\Gamma \wext (\Wk+\i\k(1/2-\Kk')-\i\k(1/2+\Kk+\i\k\Vk))\overline{\psitilde}\\
& = \langle \wext, \overline{\Wk\overline{\psitilde}} \rangle +\i\k\langle \wext, \overline{(1/2-\Kk')\overline{\psitilde}} \rangle-\i\k\langle \wext, \overline{(1/2+\Kk)\overline{\psitilde}} \rangle 
                        -(\i\k)^2\langle \wext,\overline{\Vk\overline{\psitilde}} \rangle \\
& = \langle \wext, \Wk^* \psitilde\rangle +\i\k\langle \wext,(1/2-\Kk)^*\psitilde \rangle -\i\k\langle \wext, (1/2+\Kk')^*\psitilde \rangle -(\i\k)^2\langle \wext,\Vk^*\psitilde \rangle \\
& = \langle (\Wk+\i\k(1/2-\Kk)-\i\k(1/2+\Kk'+\i\k\Vk))\wext,\psitilde\rangle,
\end{align*}
and
\begin{align*}
& \int_\Gamma\wext ((1/2+\Kk')+\i\k\Vk)\overline{\psim} =  \langle \wext,\overline{(1/2+\Kk')\overline{\psim}}\rangle +\i\k\langle \wext,\overline{\Vk\overline{\psim}}\rangle \\
& = \langle \wext,(1/2+\Kk)^*\psim\rangle +\i\k\langle \wext,\Vk^*\psim\rangle =\langle (1/2+\Kk+\i\k\Vk)\wext,\psim\rangle.
\end{align*}
Thus, using the above terms and~\eqref{regularity:Melenk:trick} for the right-hand side of~\eqref{eq:adjointconsist2}
shows STEP~3.
Combining STEPS 1--3 gives the assertion.
\end{proof}

\section{Error analysis} \label{section:error-analysis}
In this section, we prove the well posedness of scheme~\eqref{dgBEM:short-version} as well as the convergence rate of the $\h$- and $\p$-versions of the method.
We require the following approximability property.

\begin{assumption}\label{ass:approx}
Let $(\psi, \psim, \psitilde)\in H^2(\Omega)\times H^{\frac12}(\Gamma)\times H^{\frac32}(\Gamma)$ satisfy
$\Vert \psi \Vert_{2,\Omega} + \Vert \psim \Vert_{\frac{1}{2}, \Gamma}
+ \Vert \psitilde \Vert_{\frac{3}{2}, \Gamma}\leq 1$.
Then, for every $\varepsilon>0$, there exists $\eta_0(\varepsilon)>0$ such that for $\h$ and $\p$ satisfying $\h\p^{-1}\in (0, \eta_0(\varepsilon)]$ there exists $(\psih, \psimh, \psitildeh)\in \Vh\times \Wh\times \Zh$ such that
 \begin{align*}
   \Big(\big\Vert \psi-\psih \big \Vert_{\DGp}+\big\Vert  \psim-\psimh \big\Vert_{-\frac{1}{2} ,\Gamma} +\big\Vert \psitilde-\psitildeh \big\Vert_{\frac{1}{2}, \Gamma} 
 + \big\Vert  \frac{\msf^{1/2}}{p}( \psim-\psimh) \big\Vert_{0,\Gamma}\Big) &\leq\varepsilon.
 \end{align*}
  \end{assumption}

\begin{thm} \label{theorem:abstract-analysis}
Let the solution $(u,\m,\uext)$ to~\eqref{equivalent:weak:formulation} be in
$H^{\frac{3}{2} +\tepsilon} (\Omega)\times  L^2(\Gamma) \times H^{\frac{1}{2}}(\Gamma)$
for some $\tepsilon>0$, and $(\uh,\mh, \uhext)\in \Vh\times\Wh\times\Zh$ be the discrete solution of method~\eqref{dgBEM:short-version} with flux parameters defined in~\eqref{dG-parameters} and satisfying the assumptions of Theorem~\ref{theorem:discrete-Garding}.
Furthermore, let Assumption~\ref{ass:approx} be valid.
Then, there exists $\eta_0>0$ such that for $\h$, $\p$ satisfying 
  $\h\p^{-1}\in(0,\eta_0]$ and for all $(\vh,\lambdah, \vexth)$ in $\Vh\times\Wh\times\Zh$,
 \begin{multline*}
 \Vert u-\uh \Vert_{\DG} + \Vert \m-\mh\Vert_{-\frac{1}{2},\Gamma} +\Vert \uext-\uhext\Vert_{\frac{1}{2}, \Gamma}\\
\lesssim \Vert u-\vh \Vert_{\DGp}+\Vert \m-\lambdah\Vert_{-\frac{1}{2} ,\Gamma} +\Vert \uext-\vexth\Vert_{\frac{1}{2}, \Gamma} +  \Vert \msf^{1/2}\p^{-1} (\m-\lambdah)\Vert_{0,\Gamma}.
\end{multline*}
The hidden constant depends on~$\k$.
\end{thm}
\begin{proof}
We use Schatz' argument \cite{schatz1974observation}; see also~\cite{GHP_PWDGFEM_hversion,MelenkParsaniaSauter_generalDGHelmoltz,FEMBEM:mortar}.
For convenience, we write $\x:=(u,\m,\uext)$ and $\xh:=(\uh,\mh,\uhext)$. 
For all $\yh:=(\vh,\lambdah,\vexth)$ in $\Vh \times \Wh \times \Zh $ we get

\begin{align} \label{eq:error1}
\energynorm{\x-\xh}& \leq \energynorm{\x-\yh} + \energynorm{\yh-\xh}.
\end{align}
 We use the discrete G{\aa}rding inequality~\eqref{Garding:inequality} to estimate \begin{multline} \label{eq:error2}
  \energynorm{\yh-\xh}^2
\lesssim \Th ( \yh-\xh,\yh-\xh)\\
+2\Vert \kn(\vh-\uh)\Vert_{0,\Omega}^2+ \cG(\k) \left( \Vert \lambdah-\mh \Vert^2_{-\frac{3}{2}, \Gamma} + \Vert \vexth -\uhext\Vert^2_{-\frac{1}{2}, \Gamma}\right).
\end{multline}
We estimate the first term on the right-hand side of~\eqref{eq:error2}. 
Using \eqref{Galerkin:orthogonality} to replace $\xh$ by $\x$ in the first argument,
applying Proposition~\ref{proposition:continuity:trisesquilinear-form}, where the second argument is discrete, and using the Young inequality lead to
\begin{align} \label{eq:error3}
\begin{split}
\Th( \yh-\xh,\yh-\xh)
&=\Th( \yh-\x,\yh-\xh)
\\
&\lesssim \varepsilon_1^{-1}
\energynormp{\x-\yh}^2
+\varepsilon_1 \energynorm{\xh-\yh}^2.
\end{split}
\end{align}
where $\varepsilon_1>0$ will be fixed later on.
Next, we estimate the compact perturbation term appearing in~\eqref{eq:error2}.
The triangle inequality yields
\begin{align} \label{eq:error5}
\begin{split}
2\Vert \kn(\vh-\uh)\Vert_{0,\Omega}^2
&+ \cG(\k) \left( \Vert \lambdah-\mh \Vert^2_{-\frac{3}{2}, \Gamma} +  \Vert \vexth -\uhext\Vert^2_{-\frac{1}{2}, \Gamma}\right)\\
&\leq  2\Vert \kn(u-\vh)\Vert_{0,\Omega}^2+ \cG(\k) \left( \Vert m-\lambdah \Vert^2_{-\frac{3}{2}, \Gamma} +  \Vert \uext-\vexth\Vert^2_{-\frac{1}{2}, \Gamma}\right)\\
& \quad+ 2\Vert \kn(u-\uh)\Vert_{0,\Omega}^2+ \cG(\k) \left( \Vert \m-\mh \Vert^2_{-\frac{3}{2}, \Gamma} + \Vert \uext -\uhext\Vert^2_{-\frac{1}{2}, \Gamma}\right).
\end{split}
\end{align}
We apply a standard duality argument for the last two terms.
More precisely, we consider~\eqref{eq:adjointconsistency} with~$\rr=2(\kn)^2(u-\uh)$, $\rm=\cG(\k)(\m-\mh)$, $\rext=\cG(\k)(\uext-\uhext)$, and $\x-\xh$ for the test function.
We collect the solution to the adjoint problem into the vector
$\Psi:=(\psi,\psim,\psitilde)$ and we get

\begin{align*}
& 2\Vert \kn(u-\uh)\Vert_{0,\Omega}^2 + \cG(\k) \left( \Vert \m-\mh \Vert^2_{-\frac{3}{2}, \Gamma}  
   +  \Vert \uext -\uhext\Vert^2_{-\frac{1}{2}, \Gamma}\right)=\Th ( \x-\xh, \Psi).
\end{align*}
Next, we use the Galerkin orthogonality~\eqref{Galerkin:orthogonality} to  subtract
an arbitrary $\Psi_h:=(\psih,\psimh,\psitildeh)\in \Vh \times \Wh \times \Zh$ to the right-hand side,
and apply the continuity estimate~\eqref{eq:continuity} (the
first argument in the second term is discrete):
\begin{align}
\begin{split}
\label{eq:error6}
   2\Vert \kn(u-\uh)\Vert_{0,\Omega}^2&+ \cG(\k) \left( \Vert \m-\mh \Vert^2_{-\frac{3}{2}, \Gamma}  
                                        +  \Vert \uext -\uhext\Vert^2_{-\frac{1}{2}, \Gamma}\right)\\
	&\quad=\Th ( \x-\xh, \Psi-\Psi_h)  =\Th ( \x-\yh, \Psi-\Psi_h)
	 +\Th ( \yh-\xh, \Psi-\Psi_h)\\									
  &\quad \lesssim \big(\energynormp{\x-\yh}+ \energynorm{\x_h-y_h}\big) \energynormp{\Psi-\Psi_h}.
 \end{split}
\end{align}
From~\eqref{bounds:norms:solutions:dual}, we see that
\begin{align*}
 \Vert \psi \Vert_{2,\Omega} + \Vert \psim \Vert_{\frac{1}{2}, \Gamma} + \Vert \psitilde \Vert_{\frac{3}{2}, \Gamma} \lesssim  \Vert u-\uh \Vert_{0,\Omega} + \Vert \m-\mh \Vert_{-\frac{1}{2}, \Gamma} + \Vert \uext-\uhext\Vert_{\frac{1}{2}, \Gamma}.
\end{align*}
This, together with Assumption~\ref{ass:approx}, yields 
\begin{align*}
  \energynormp{\Psi-\Psi_h}
  &\lesssim    
    \varepsilon \big(\Vert \psi \Vert_{2,\Omega} + \Vert \psim \Vert_{\frac{1}{2}, \Gamma} + \Vert \psitilde \Vert_{\frac{3}{2}, \Gamma} \big)
    \lesssim
    \varepsilon \energynorm{\x-\xh}.
\end{align*}
We insert this bound into~\eqref{eq:error6} and merge the resulting bound with~\eqref{eq:error5}:
\begin{align}  \label{eq:error4}
  \begin{split}
2\Vert \kn(\vh-\uh)\Vert_{0,\Omega}^2
&+ \cG(\k) \left( \Vert \lambdah-\mh \Vert^2_{-\frac{3}{2}, \Gamma} + \Vert \vexth -\uhext\Vert^2_{-\frac{1}{2}, \Gamma}\right)\\
&\lesssim  \energynorm{\x-\yh}^2 + (\energynormp{\x-\yh}+\energynorm{\x_h-\yh})\varepsilon \energynorm{\x-\x_h}\\
&\lesssim (1+\varepsilon) \energynormp{\x-\yh}^2 + \varepsilon \energynorm{\x_h-\yh}^2.
\end{split}
\end{align}
Eventually, we insert~\eqref{eq:error3} and~\eqref{eq:error4} in~\eqref{eq:error2} and, writing~$c$ for the  constant implied in all the previous estimates, we get
\begin{align*}
\energynorm{\xh-\yh}^2
&\leq  c \big(\varepsilon_1^{-1}+1+\varepsilon\big)\energynormp{\x-\yh}^2
    +  c \big(\varepsilon_1+\varepsilon\big)\energynorm{\xh-\yh}^2.
\end{align*}
Assuming that $\varepsilon$ in Assumption~\ref{ass:approx} is sufficiently small
and taking $\varepsilon_1$ small enough, we shift the second term to the left-hand side:
\begin{equation} \label{eq:error7}
(1- c (\varepsilon+\varepsilon_1)) \energynorm{\xh-\yh}^2 \lesssim \energynormp{x-y_h}^2.
\end{equation}
Inserting~\eqref{eq:error7} in~\eqref{eq:error1} concludes the proof.
\end{proof}

The quasi-optimality result Theorem~\ref{theorem:abstract-analysis} can lead
to quantitative error estimates that are explicit in the mesh size $\h$ and the polynomial
degree $\p$. To obtain higher order rates of convergence, the element maps $\Phi_K$ 
need to have more regularity than what has been assumed so far at the outset of 
Section~\ref{subsection:dGOmega}. To be concrete, one can make the following 
assumption as in \cite{bernardi_optimal_interpolation}.
\begin{assumption}
\label{assumption:bernardi89}
Given $s \in \mathbb{N}$, there is a constant $\widetilde c_B > 0 $ such that 
\[
\|D^l \Phi_K \|_{L^\infty(\widehat{K})} \leq \widetilde c_B  h_K^l, \qquad 
2 \leq l \leq s+1. 
\]
\end{assumption}
\begin{remark}
Scenarios for the constructions of triangulations and element maps that ensure the validity of Assumption~\ref{assumption:bernardi89} are provided in \cite{bernardi_optimal_interpolation}.  
\eremk
\end{remark}
\begin{cor}
\label{corollary:hp-convergence}
Let $s \in \mathbb{N}$ and Assumption~\ref{assumption:bernardi89} be valid.  Set $h:= \max_K (h_K)$. 
Let the solution  $(u,\m,\uext)$ to~\eqref{equivalent:weak:formulation}  
belong to $H^{s+1}(\Omega) \times H^{s-\frac{1}{2}}(\Gamma) \times H^{s+\frac12}(\Gamma)$
and $(\uh,\mh, \uhext)\in \Vh\times\Wh\times\Zh$ be the discrete solution
of method~\eqref{dgBEM:short-version} with flux parameters defined
in~\eqref{dG-parameters} and satisfying the assumptions of Theorem~\ref{theorem:discrete-Garding}.
Then, there are constants $\eta_0 = \eta_0(k)$ and $c(k) > 0$ such
that, under the scale resolution  condition $hp^{-1} \in (0,\eta_0]$, the following bound is valid
\begin{align*}
& \Vert u - \uh \Vert_{\DG} + \Vert \m - \mh \Vert_{-\frac12, \Gamma} + \Vert \uext - \uhext \Vert_{\frac{1}{2},\Gamma}\\
& \qquad \leq c (k) h^{\min (\p,s)} \p^{-s+\frac12} ( \Vert  u \Vert_{s+1,\Omega} + \Vert \m \Vert_{ s-\frac{1}{2},\Gamma} + \Vert \uext \Vert_{s+\frac12,\Gamma}).
\end{align*}
\end{cor} 
\begin{proof}
See Appendix~\ref{appendix:proof-Corollary}.
\end{proof}

\begin{remark}[suboptimality in $\p$]
The suboptimality by half an order in the polynomial degree~$\p$ is due to the $p$-scaling of 
the parameter $\alpha$ in the definition of the \dG{} norm in~\eqref{dG:norm:1}.
Under further assumptions on the mesh, it is possible to argue as 
in~\cite[Sec.~{4.2.2}]{MelenkParsaniaSauter_generalDGHelmoltz} to obtain $\p$-optimal estimates. 
\eremk
\end{remark}
\begin{remark}[exponential convergence]
Exponential convergence that is explicit in $\h$ and $\p$ for analytic solutions can be proved if 
the element maps $\Phi_K$ are assumed to be of the form $\Phi_K = R_K \circ A_K$ with an analytic map $R_K$ and an
affine map $A_K$. We refer to \cite[Assump.~{4.1}]{MelenkParsaniaSauter_generalDGHelmoltz} for details. 
See also \cite[Sec.~{3.3.2}]{melenk02} for the concept of ``patchwise structured meshes''.  
\eremk
\end{remark}

\section{Numerical results} \label{section:numerical-results}
In this section, we present numerical results validating the
convergence rate detailed in Corollary~\ref{corollary:hp-convergence}.

We implemented method~\eqref{dGBEM:long-version} by combining the NGSolve package~\cite{NGSolve} with the
BEM++ library~\cite{BEM++paper,BEM++}.
In particular, we proceeded as in~\cite{FEMBEM:mortar}, yet replacing the interior discretization with the novel discontinuous Galerkin part.
In order to solve the resulting algebraic linear system, we used a GMRES iteration with a preconditioner based on $\mathcal{H}$-matrix $LU$-decomposition provided by the H2Lib library~\cite{h2lib}.

We considered sequences of quasiuniform tetrahedral meshes $\Omega_h$ in~$\Omega$ and 
used the trace of the corresponding interior finite element mesh as a partition $\Gamma_h$ of~$\Gamma$.
As for the choice of the discretization spaces, we picked~$\Vh = S^{p,0}(\Omega,\Omega_h)$ as the space of discontinuous piecewise polynomials 
of order $p$ over the tetrahedral meshes $\Omega_h$, whereas we picked~$\Zh = S^{p,1}(\Gamma,\Gamma_h)$ and~$\Wh = S^{p-1,0}(\Gamma,\Gamma_h)$ 
as the spaces of continuous and discontinuous piecewise polynomials of orders~$p$ and $p-1$ over the triangulation $\Gamma_h$ of~$\Gamma$, respectively.

We are interested in studying the convergence of the following relative errors:
\[
\frac{\Vert u - \uh \Vert_{0,\Omega}}{\Vert u \Vert_{0,\Omega}}, \qquad
\frac{\Vert \nablah(u - \uh) \Vert_{0,\Omega}}{\Vert \nabla u \Vert_{0,\Omega}}, \qquad
\h^{\frac12}\frac{\Vert \m - \mh \Vert_{0,\Gamma}}{\Vert m \Vert_{0,\Gamma}}, \qquad
\h^{-\frac12} \frac{\Vert \uext - \uhext \Vert_{0,\Gamma}}{\Vert \uext \Vert_{0, \Gamma}}.
\]
For the $\h$-version of the method, the last two error measures scale like the relative errors 
in the~$H^{-\frac12}(\Gamma)$ and the~$H^{\frac12}(\Gamma)$, respectively. 
The stabilization parameters of the DG method~\eqref{dG-parameters} are taken to be $\afrakz:=10$, and $\bfrakz:=\dfrakz:=0.1$.

We investigated the performance of method~\eqref{dGBEM:long-version} for 
the domain~$\Omega:=(-1,1)^3$ and the coefficients~$\nu=1$ and~$n=1$
in~\eqref{starting:problem}, and prescribe the exact smooth solution 
\begin{align} \label{exact:solution}
  u(x,y,z):=\begin{cases}
    \sin(\k\,x)\cos(\k y) & (x,y,z) \in \Omega \\
    \frac{e^{i \k \sqrt{x^2+y^2+z^2}}}{\sqrt{x^2+y^2+z^2}} & \text{otherwise}.
  \end{cases}
\end{align}
The function~$u$ solves the Helmholtz equation in $\Rbb^3$ but has nonzero Dirichlet and Neumann jumps.
This case is not covered by the theory in Sections~\ref{section:continuous:problem}--\ref{section:error-analysis},
but can be incorporated into method~\eqref{dGBEM:long-version} via a suitable modification of the right-hand sides.

The coupling strategy based on the mortar variable $\m$ aims at solvability for all wave numbers~$\k$. To underline
this feature, we select the wave numbers~$\k$ as $\k:=n \sqrt{3}\pi$ for $n =1$, $2$, which are the first two nonzero eigenvalues of the Dirichlet and Neumann Laplacian on the unit cube.
Figures~\ref{fig:conv_h1} and~\ref{fig:conv_h2} show that the 
method~\eqref{dGBEM:long-version} delivers optimal convergence rates
of the errors after some pre-asymptotic phase, which is expected due to dispersion errors (``pollution'' effect) typical of 
wave propagation problems. 
These rates partly surpass those predicted by Corollary~\ref{corollary:hp-convergence},
which only considers a convergence of the combined error, i.e., the rate of all contributions would be dominated by the 
lowest order contribution, namely, the $H^1(\Omega)$ seminorm.
A similar superconvergence phenomenon is well known for the simpler Poisson problem and analyzed in details in~\cite{mpw2017simultaneous}.

\begin{figure}[htb]
  \centering
  \begin{subfigure}{0.45\textwidth}
    \input{conv_h_k1_0.tex}
  \end{subfigure}
  \quad
  \begin{subfigure}{0.45\textwidth}
    \input{conv_h_k1_1}
  \end{subfigure}
  \quad
    \begin{subfigure}{0.45\textwidth}
    \input{conv_h_k1_2}
  \end{subfigure}
  \quad
  \begin{subfigure}{0.45\textwidth}
    \input{conv_h_k1_3}
  \end{subfigure}
\caption{$h$-version. Wave number $\k=2\sqrt{3}\pi$.}
\label{fig:conv_h1}  
\end{figure}
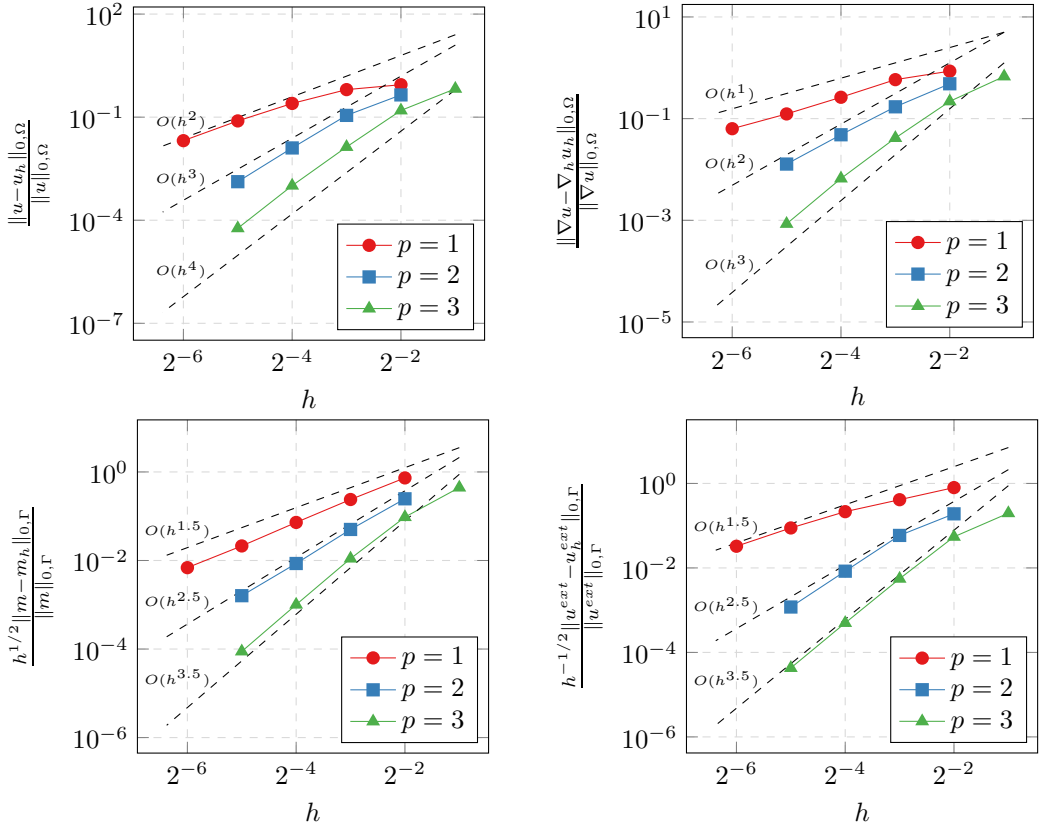

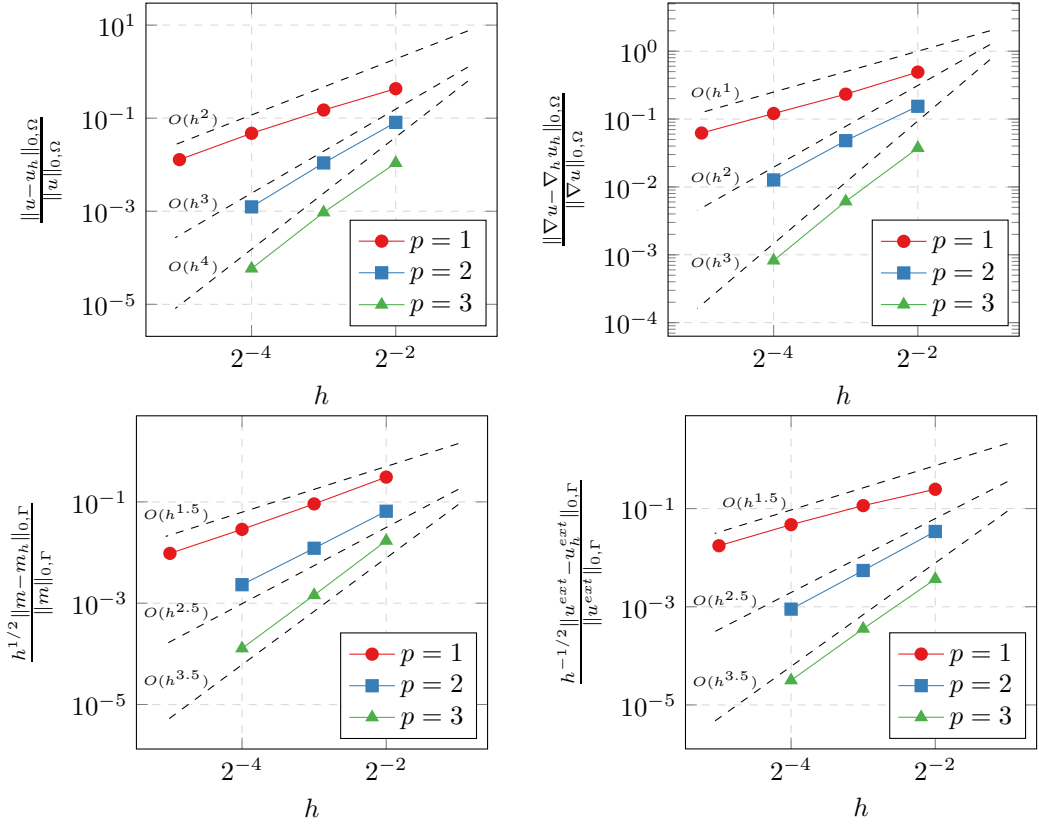
\begin{figure}[htb]
  \centering
  \begin{subfigure}{0.45\textwidth}
  \input{conv_h_k2_0}\end{subfigure}
  \begin{subfigure}{0.45\textwidth}
  \input{conv_h_k2_1}\end{subfigure} \quad%
  \begin{subfigure}{0.45\textwidth}
  \input{conv_h_k2_2}\end{subfigure}  \quad%
  \begin{subfigure}{0.45\textwidth}
  \input{conv_h_k2_3}\end{subfigure}
\caption{$h$-version. Wave number $\k=\sqrt{3}\pi$.}
\label{fig:conv_h2}  
\end{figure}

We also considered  the $p$-version of the method with wave numbers~$\k:=4\sqrt{3}{\pi}\sim 21.7$ and~$\k:=2\sqrt{3}{\pi} \sim 10.88$.
We fixed an underlying uniform mesh of size $h\approx 1/4$ and considered the exact solution as in~\eqref{exact:solution}.
For both wave numbers, we observe exponential convergence after a small preasymptotic regime;
see Figure~\ref{fig:conv_p}.

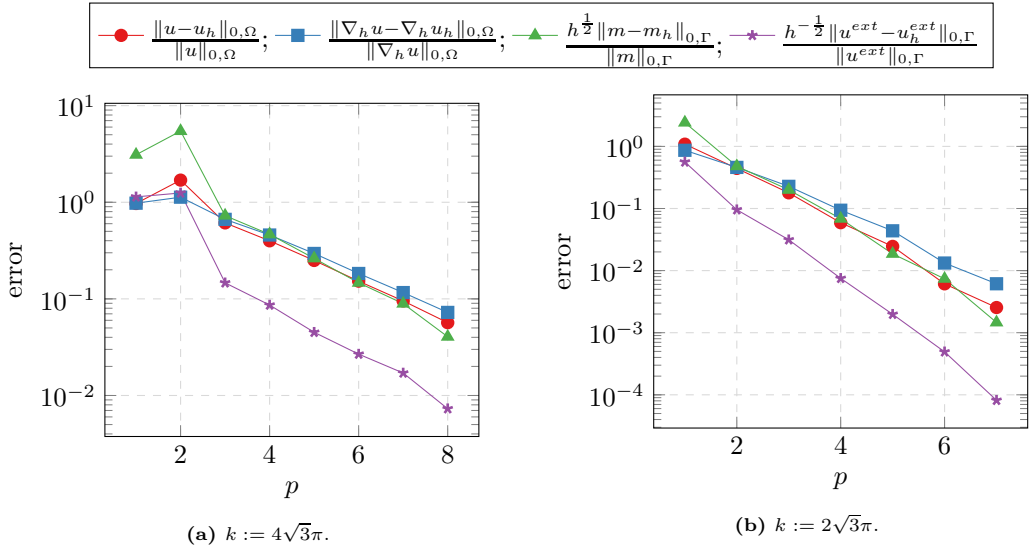
\begin{figure}[htb]
  \centering
  \quad\ref{thatlegend} \vspace{2mm}\\
  \begin{subfigure}{0.45\textwidth}
\begin{tikzpicture}
    \pgfplotsset{   
      cycle list/Set1,
      cycle multiindex* list={
        mark list\nextlist
        Set1\nextlist
      },
    }

  \begin{semilogyaxis}[
    width=6.5cm, height=6cm,     
    grid = major,
    grid style={dashed, gray!30},
    axis background/.style={fill=white},
    xlabel=$p$,
    ylabel=error,
    legend style={legend pos=south west},
    legend pos= outer north east,
    legend to name=thatlegend,
    legend columns=-1,
    legend entries={$\frac{\|{u-\uh}\|_{0,\Omega}}{\|u\|_{0,\Omega}}$;,
      $\frac{\|{\nablah u-\nablah \uh}\|_{0,\Omega}}{\|\nablah u\|_{0,\Omega}}$;,
      $\frac{\h^{\frac12} \|{\m -\mh\|}_{0,\Gamma}}{\|\m\|_{0,\Gamma}}$;,
      $\frac{\h^{-\frac12} \|{\uext - \uhext}\|_{0,\Gamma}}{\|\uext\|_{0,\Gamma}}$},
    ]
    \def\thisk{20}
    \def\thatk{23}
 
    \addplot+[solid,mark size=2pt,mark options={line width=1.0pt}] table    
    [
    x expr=\thisrow{order},
    y expr=\thisrow{errltwo},
    col sep=comma,
    restrict expr to domain={\thisrow{num_refines}}{2:2},
    restrict expr to domain={\thisrow{k}}{\thisk:\thatk}
    ]{output_dg_p.csv};

    \addplot+[solid,mark size=2pt,mark options={line width=1.0pt}] table    
    [
    x expr=\thisrow{order},
    y expr=\thisrow{errhone},
    col sep=comma,
    restrict expr to domain={\thisrow{num_refines}}{2:2},
    restrict expr to domain={\thisrow{k}}{\thisk:\thatk}
    ]{output_dg_p.csv};

    \addplot+[solid,mark size=2pt,mark options={line width=1.0pt}] table    
    [
    x expr=\thisrow{order},
    y expr=\thisrow{errm},
    col sep=comma,
    restrict expr to domain={\thisrow{num_refines}}{2:2},
    restrict expr to domain={\thisrow{k}}{\thisk:\thatk}
    ]{output_dg_p.csv};

    \addplot+[solid,mark size=2pt,mark options={line width=1.0pt}] table    
    [
    x expr=\thisrow{order},
    y expr=\thisrow{errphi},
    col sep=comma,
    restrict expr to domain={\thisrow{num_refines}}{2:2},
    restrict expr to domain={\thisrow{k}}{\thisk:\thatk}
    ]{output_dg_p.csv};
  \end{semilogyaxis}
\end{tikzpicture}
\caption{$\k:=4\sqrt{3}{\pi}$.}
\end{subfigure}
\quad
\begin{subfigure}{0.45\textwidth}
  \begin{tikzpicture}
    \pgfplotsset{   
      cycle list/Set1,
      cycle multiindex* list={
        mark list\nextlist
        Set1\nextlist
      },
    }

  \begin{semilogyaxis}[
    width=6.5cm, height=6cm,     
    grid = major,
    grid style={dashed, gray!30},
    axis background/.style={fill=white},
    xlabel=$p$,
    ylabel=error,
    legend pos= outer north east,
    ]
    \def\thisk{9}
    \def\thatk{11}
    \addplot+[solid,mark size=2pt,mark options={line width=1.0pt}] table    
    [
    x expr=\thisrow{order},
    y expr=\thisrow{errltwo},
    col sep=comma,
    restrict expr to domain={\thisrow{num_refines}}{2:2},
    restrict expr to domain={\thisrow{k}}{\thisk:\thatk}
    ]{output_dg_p.csv};

    \addplot+[solid,mark size=2pt,mark options={line width=1.0pt}] table    
    [
    x expr=\thisrow{order},
    y expr=\thisrow{errhone},
    col sep=comma,
    restrict expr to domain={\thisrow{num_refines}}{2:2},
    restrict expr to domain={\thisrow{k}}{\thisk:\thatk}
    ]{output_dg_p.csv};

    \addplot+[solid,mark size=2pt,mark options={line width=1.0pt}] table    
    [
    x expr=\thisrow{order},
    y expr=\thisrow{errm},
    col sep=comma,
    restrict expr to domain={\thisrow{num_refines}}{2:2},
    restrict expr to domain={\thisrow{k}}{\thisk:\thatk}
    ]{output_dg_p.csv};

        \addplot+[solid,mark size=2pt,mark options={line width=1.0pt}] table    
    [
    x expr=\thisrow{order},
    y expr=\thisrow{errphi},
    col sep=comma,
    restrict expr to domain={\thisrow{num_refines}}{2:2},
    restrict expr to domain={\thisrow{k}}{\thisk:\thatk}
    ]{output_dg_p.csv};

  \end{semilogyaxis}
\end{tikzpicture}
\caption{$\k:=2\sqrt{3}{\pi}$.}
\end{subfigure}
  \caption{$p$-version.}
  \label{fig:conv_p}  
\end{figure}

\section{Conclusions} \label{section:conclusions}

We introduced a \DGFEMBEM{} mortar coupling for three dimensional Helmholtz problems
with variable coefficients.
Upon showing that the discrete sesquilinear form satisfies a G{\aa}rding inequality and continuity bounds,
we showed quasi-optimality of the $\h$- and $\p$-versions of the scheme.
The theoretical results are validated by numerical examples.
Notably, theoretical and numerical results are valid regardless of whether the wave number is a Dirichlet or Neumann Laplace eigenvalue.

As a pivot result of independent interest, we constructed a discontinuous-to-continuous reconstruction operator on tetrahedral meshes, with optimal $\h$- and $\p$-stability properties in the $H^1$ seminorm and in the $L^2$ norm, covering the case of curvilinear meshes.

\section*{Acknowledgements}
JMM, IP, and AR gratefully acknowledge funding by the Austrian Science Fund (FWF) through the project F~65 ``Taming Complexity in Partial Differential System''.
IP and AR also acknowledge funding by the FWF through the project P~29197-N32.
LM and AR acknowledge support from the FWF project P33477.

{\footnotesize
\bibliography{bibliogr}
}
\bibliographystyle{plain}

\appendix

\section{Consistency of method~\eqref{dGBEM:long-version}} \label{appendix:sectionA}

\begin{proof}[Proof of Lemma~\ref{lemma:consistency}]
Proving assertion~\eqref{dgBEM:consistency} is equivalent to proving that
the continuous solution $(u,\m,\uext)$ solves also the three equations in~\eqref{dGBEM:long-version}.
Since $u\in H^{\frac{3}{2}+\tepsilon}(\Omega)$ we have that
\begin{align}  \label{eq:jumpmeanzero}
 \llbracket u \rrbracket=0, \qquad \llbracket \nabla u\rrbracket=0, \qquad\ldc\nabla u\rdc=\nabla u  \qquad \text{on }\FhI.
\end{align} 
We multiply \eqref{first:system:Helmholtz} by $\overline{\vh}\in\Vh$ and integrate elementwise by parts to get
\begin{align*}
 \sum_{\E \in \taun} \left ( -\int_{\partial \E} \Ad\nabla u \cdot \nGamma \overline{\vh} 
 +\int_{\E}\Ad\,\nabla u \cdot \overline{\nabla \vh}\right) 
 -\int_\Omega  (\kn)^2 u \overline{\vh} =\int_\Omega f\overline{\vh} .
\end{align*}
With the aid of the boundary condition in~\eqref{first:system:Helmholtz},
inserting the parameter~$\delta$, and using the fact that $\Ad=1$ on $\Gamma$,
 we manipulate the boundary term as follows:
\begin{align*}
&  -\sum_{\E \in \taun} \int_{\partial \E} \Ad\nabla u \cdot \nGamma \overline{\vh} \\
& = -\int_{\Gamma}\delta \nabla u\cdot\nGamma \overline{\vh}  -\int_{\Gamma}(1-\delta) \m \overline{\vh}
  +\int_{\Gamma}  \i\k (1-\delta)u \overline{\vh} 
  -\int_{\FhI}\Ad \nabla u \cdot \llbracket \overline{\vh}\rrbracket \\
& \quad +\int_{\Gamma}(\i\k)^{-1}\delta \m \overline{\nabla \vh\cdot\nGamma}
  -\int_{\Gamma}(\i\k)^{-1}\delta \nabla u\cdot\nGamma \overline{\nabla \vh\cdot\nGamma}
  -\int_{\Gamma}\delta u \overline{\nabla \vh\cdot\nGamma} .
\end{align*}
Properties~\eqref{eq:jumpmeanzero} and the above identity lead to the consistency of the first equation of~\eqref{dGBEM:long-version}, i.e.,
\begin{align*}
 {\sum_{\E \in \taun} \ahE (u,\vh) + \bhGamma (u,\vh) - (\m, \delta (\i\k)^{-1} \nabla_h \vh \cdot \nGamma + (1-\delta)\vh)_{0,\Gamma} = (\f,\vh)_{0,\Omega}} \qquad \forall \vh \in \Vh.
\end{align*}
To show the consistency of the second equation of~\eqref{dGBEM:long-version},
we multiply~\eqref{formula:Chandler-Wilde}, which is an equivalent formulation of~\eqref{second:system:Helmholtz},
by $\overline{\vexth}\in\Zh$ and integrate over $\Gamma$:
\begin{align*}
 \langle (\Bk + \i\k \Aprimek) \uext - \Aprimek \m, \vexth \rangle =0 \qquad \forall \vexth \in \Zh.
\end{align*}
Eventually, multiplying~\eqref{third:system:Helmholtz} by $\overline{\lambdah}\in\Wh$ and integrating over~$\Gamma$, we get
\begin{align*}
 \langle u , \lambdah\rangle - \langle ( \frac{1}{2} + \Kk  ) \uext - \Vk (\m - i\k \uext), \lambdah\rangle = 0.
\end{align*}
Similarly as above, the boundary condition in~\eqref{first:system:Helmholtz} leads to
\begin{align*}
\langle- \delta(\i\k)^{-1}\nabla u\cdot\nGamma,\lambdah\rangle +\langle -\delta u,\lambdah\rangle + \langle\delta(\i\k)^{-1}m,\lambdah\rangle=0.
\end{align*}
Summing up the last two equations shows the consistency of the third equation in~\eqref{dGBEM:long-version}.
\end{proof}

\section{An $\h\p$-stable, discontinuous-to-continuous reconstruction operator on curvilinear simplicial meshes} \label{appendix:proof-Lemma-splitting}

Here, we prove Theorem~\ref{lemma:Karakashian-style}.

Let the mesh~$\taun$ satisfy the shape regularity assumption~\eqref{eq:SR} and~$v \in \Hpw^1(\taun)$.
We construct the operator~$\Pcal:\Hpw^1(\taun) \rightarrow H^1(\Omega)$ as the composition~$\Pcal:= \Pcalt \circ \Pcalo$ of two operators 
$\Pcalt$, $\Pcalo$ that we define below.
Preliminarily, for each~$\E \in \taun$, we construct a
quasi-uniform, shape regular simplicial decomposition~$\tautildenE$ of~$\E$, such that the size of each 
element~$\Etilde$ of~$\tautildenE$ is comparable to~$\htilde_\E := \hE/\ell^2$.
Denote the union of all~$\tautildenE$ by~$\tautilden$.
By using a standard refinement strategy on the original mesh, we can additionally ensure that $\tautilden$ does not contain hanging nodes.
We also introduce
\begin{equation}
\Vtildeh:=\{v \in S^{1,0}(\Omega,\tautilden)\,|\, v{}_{|_K} \in S^{1,1}(\E,\tautildenE) \quad \forall K \in \taun\},
\end{equation}
the space of the mapped, piecewise linear polynomials over~$\tautilden$, which are continuous in each~$\E \in \taun$ but possibly discontinuous at the interfaces of~$\taun$.

We define~$\Pcalo: \Hpw^1(\taun) \rightarrow \Vtildeh$ as follows.
For each~$\E\in \taun$, $\Pcalo (\vh{}_{|\E})\in {\cal S}^{1,1}(K, \tautildenE)$ is the
quasi-interpolant of $v$ defined in~\cite[Sec.~4]{bernardi_optimal_interpolation}.
As for $\Pcalt: \Vtildeh \rightarrow {\cal S}^{1,1}(\Omega, \tautilden)\subset H^1(\Omega)$, we choose
the lowest-order, Oswald-type operator introduced by Karakashian and
Pascal in~\cite{KarakashianPascal}.
This operator interpolates the arithmetical averages of the degrees of freedom at each vertex of the mesh $\tautilden$.
Thus, we are actually going to prove Theorem~\ref{lemma:Karakashian-style}
with $\Pcal:\Hpw^1(\taun) \rightarrow {\cal S}^{1,1}(\Omega, \tautilden)\subset H^1(\Omega)$.
For simplicity, throughout this section we assume that $\h/{\ell^2} \lesssim 1$ and $\ell \in \Nbb$. The
other cases follow similarly but would incur some cumbersome notation/case distinctions.

Before proving~\eqref{K-propertyA}--\eqref{K-propertyC}, we recall two propositions, which summarize the properties of
the operators~$\Pcalo$ and $\Pcalt$.
\begin{prop}
For any element $K \in \taun$, the quasi-interpolant $\Pcalo:\Hpw^1(\taun)\rightarrow \Vtildeh$ satisfies the following estimates:
\begin{align}
\Vert \nabla \Pcalo v \Vert_{0,\E} & \lesssim \Vert \nabla v \Vert_{0,\E} ,	\label{approximation:SZ1} \\
\Vert v -\Pcalo v \Vert_{0,K} & \lesssim \Vert \msf \ell^{-2} \nabla v \Vert_{0,\E}, \label{approximation:SZ2} \\
 \Vert \llbracket \Pcalo v \rrbracket \Vert_{0,\partial \E \setminus \Gamma} &\lesssim  \Vert \llbracket v \rrbracket \Vert_{0,\partial \E\setminus \Gamma}+ \Vert \msf^{1/2}\ell^{-1} \nabla_h v \Vert_{0,\omega_K},\label{approximation:SZ3}
\end{align}
where $\omega_K$ in~\eqref{approximation:SZ3} denotes the set of elements sharing a face with~$K$.
\end{prop}
\begin{proof}
Bounds~\eqref{approximation:SZ1} and~\eqref{approximation:SZ2} follow from~\cite[Thm.~{4.1}]{bernardi_optimal_interpolation} locally on~$K$ as the domain to obtain a function on the subtriangulation~$\tautildenE$.
We can apply~\cite[Thm.~{4.1}]{bernardi_optimal_interpolation} since~$\tautildenE$ fulfills~\eqref{eq:SR} and 
thus~\eqref{eq:Bernardi}, which is the condition required there.
  
To show~\eqref{approximation:SZ3}, we fix a facet~$F$ shared by the elements~$K$ and~$K'$. We get
\begin{align*}
    \Vert \llbracket \Pcalo v \rrbracket \Vert_{0,\F}
    &\leq \Vert \llbracket  v \rrbracket \Vert_{0,\F} + \Vert \llbracket v - \Pcalo v \rrbracket \Vert_{0,F} \\
    & \leq \Vert \llbracket v \rrbracket \Vert_{0,\F}
      + \Vert  \big(v - \Pcalo v \big)_{|K}  \Vert_{0,F} + \Vert  \big(v - \Pcalo v \big)_{|\E'}  \Vert_{0,\F}.
\end{align*}
For brevity, we only consider the third term on the right-hand side. Transforming to the reference element, applying a multiplicative trace estimate and transforming back gives
\begin{align*}
\Vert  \big(v - \Pcalo v \big)_{|\E'}  \Vert_{0,F}
&\lesssim \Vert  \msf^{-1/2} (v - \Pcalo v) \Vert_{0,\E'} +  \Vert v - \Pcalo v \Vert_{0,\E'}^{1/2} \Vert \nabla (v - \Pcalo v) \Vert_{0,K'}^{1/2}.
\end{align*}
Inserting~\eqref{approximation:SZ1} and~\eqref{approximation:SZ2} yields~\eqref{approximation:SZ3}.
\end{proof}
\begin{prop}
The Oswald-type operator~$\Pcalt: \Vtildeh \to  {\cal S}^{1,1}(\Omega, \tautilden)$ satisfies the following properties:
\begin{align}  \label{eq:Burman-Ern-01}
\Vert  \widetilde{v}_h - \Pcalt \widetilde{v}_h \Vert _{0,\Omega}
\lesssim \Vert \msf^{1/2} \ell^{-1}\llbracket \widetilde{v}_h \rrbracket \Vert_{0,\FhI},\quad \Vert \nablah( \widetilde{v}_h - \Pcalt \widetilde{v}_h )\Vert _{0,\Omega}
\lesssim \Vert \msf^{-1/2} \ell \llbracket \widetilde{v}_h   \rrbracket \Vert_{0,\FhI}.
\end{align}
\end{prop}
  \begin{proof}
We claim that
\[
\Vert  \widetilde{v}_h - \Pcalt \widetilde{v}_h \Vert _{0,\Omega}^2
\lesssim   \sum_{\E \in \taun} {\Vert \htilde_{\E}^{1/2}\llbracket \widetilde{v}_h
    \rrbracket \Vert_{0,\partial \E\setminus \Gamma}^{2}} ,\quad
\Vert  \nablah(\widetilde{v}_h - \Pcalt \widetilde{v}_h) \Vert _{0,\Omega}^2
      \lesssim   \sum_{\E \in \taun} {\Vert \htilde_{\E}^{-1/2}\llbracket \widetilde{v}_h  \rrbracket \Vert_{0,\partial \E\setminus\Gamma}^{2}.
}
\]
This follows as in the proof of~\cite[Thm.~{2.2}]{KarakashianPascal}, which only makes use of the definition of the Lagrangian degrees of freedom of $\Pcalt \widetilde{v}_h$ as arithmetical averages of the degrees of freedom of $\widetilde{v}_h$ and of the scaling properties of the basis functions.
We remark that~\cite[Thm.~{2.2}]{KarakashianPascal} states the estimate in the $H^1$ seminorm; the estimate in the $L^2$ norm follows along the same lines; 
see also~\cite[Lemma~{5.3}]{MR2299768}.
Then, the estimates in~\eqref{eq:Burman-Ern-01} follow from the definition of $\htilde_\E = \hE/\ell^2$ and the fact that function~$\widetilde{v}_h$ is continuous within each element~$K \in \taun$,
i.e., no extra jumps are introduced along the edges of the refined triangulation~$\tautilden$.
\end{proof}

As an immediate consequence of the shape regularity of~$\Omegah$ and the locality of the operator~$\Pcalo$, we get
\begin{align}\label{Burman-Ern-global}
  \Vert \msf^{-1} \ell^2\big(\Pcalo v - \Pcalt (\Pcalo v) \big) \Vert _{0,\Omega}
  & \overset{\eqref{eq:Burman-Ern-01}}{\lesssim} \Vert \msf^{-{1/2}}
    \ell \llbracket  \Pcalo v \rrbracket \Vert_{0,
    \FhI}\overset{\eqref{approximation:SZ3}}{\lesssim} \Vert
    \msf^{-{1/2}} \ell \llbracket v \rrbracket \Vert_{0, \FhI} +
    \Vert \nabla_\h v \Vert_{0,\Omega}.
\end{align}
We prove further properties of the operator~$\Pcalt$.
First, proceeding as in Remark~\ref{remark:choice-dG-parameters}, we have the following inverse estimate for mapped, affine functions:
\begin{equation} \label{inverse:h1-L2-polynomial}
\Vert \nabla q \Vert_{0,\Etilde} \lesssim \htilde_{\E}^{-1} \Vert q \Vert_{0,\Etilde} =   \Vert  \msf^{-1}\ell^2 q \Vert_{0,\Etilde}
\quad \forall \Etilde\in \tautildenE,\; \forall q \in \mathcal S^{1,1}(K, \tautildenE).
\end{equation}
Next, we observe that
\begin{equation} \label{Burman-Ern-2}
\begin{split}
\Vert & \nablah \Pcalt(\Pcalo v) \Vert_{0,\Omega} \le \Vert \nablah(\Pcalo v) \Vert_{0,\Omega} + \Vert \nablah(\Pcalo v - \Pcalt(\Pcalo v)) \Vert_{0,\Omega} \\
& \overset{\eqref{approximation:SZ1}}{\lesssim} \Vert \nablah v \Vert_{0,\Omega} + \Vert \nablah(\Pcalo v - \Pcalt(\Pcalo v)) \Vert_{0,\Omega}
\overset{\eqref{inverse:h1-L2-polynomial}}{\lesssim} \Vert \nablah v \Vert_{0,\Omega} + \Vert \msf^{-1}\ell^2 \big( \Pcalo v - \Pcalt(\Pcalo v) \big) \Vert_{0,\Omega} \\
& \overset{\eqref{Burman-Ern-global}}{\lesssim} \Vert \nablah v \Vert_{0,\Omega} + \Vert  \msf^{-1/2}\ell \llbracket v \rrbracket \Vert_{0,\FhI}.
\end{split}
\end{equation}
From this and the triangle inequality, we get~\eqref{K-propertyA}.

In order to prove~\eqref{K-propertyB}, we observe that the following
approximation property of the operator~$\Pcalt$ is valid:
\begin{equation} \label{Burman-Ern-3}
\begin{split}
\Vert v - \Pcalt(\Pcalo v) \Vert _{0,\Omega}
& \quad\;\; \le  \Vert v - \Pcalo v \Vert _{0,\Omega} + \Vert \Pcalo v - \Pcalt(\Pcalo v) \Vert _{0,\Omega}\\
& \overset{\eqref{approximation:SZ2},\eqref{Burman-Ern-global}}{\lesssim}
\Vert \msf \ell^{-2}\nablah v \Vert_{0,\Omega}  +    \Vert \msf^{1/2} \ell^{-1} \llbracket v \rrbracket \Vert_{0, \FhI}.
\end{split}
\end{equation}
Then, \eqref{K-propertyB} follows by the triangle inequality.

We are left to prove~\eqref{K-propertyC}. To that end, we use a scaling argument.
Given $v \in \Hpw^1(\taun)$, for any $\E\in\taun$, let $\widehat v$ be the polynomial pull-back of ${v}_{|_{\E}}$ through the mapping $\Phi_K: \Ehat\to \E$.
We denote the counterparts of  $\Pcalo$ and~$\Pcalt$ acting on the polynomials on $\Ehat$ by~$\Pcalohat$ and~$\Pcalthat$, respectively.
For any boundary face $\F \in \FhB$, we denote the pull-back of~$\F$ through~$\PhiE$ by~$\Fhat$, where~$\E$ is the only element such that $\F\subset\partial\E$.
For all~$\F \in\FhB$, we apply a scaling argument, the multiplicative trace inequality, and the Young inequality to get
\begin{align*}
\Vert v 	 - \Pcalt(\Pcalo v) \Vert_{0,\F}^2 & \lesssim \Vert \msf (\widehat v - \Pcalthat(\Pcalohat\widehat v) )\Vert_{0,\Fhat}^2 \\
& \lesssim \Vert \msf (\widehat v - \Pcalthat (\Pcalohat \widehat v)) \Vert_{0,\Ehat}^2 
         + \Vert \msf (\widehat v - \Pcalthat (\Pcalohat \widehat v)) \Vert_{0,\Ehat}   \Vert \msf \widehat\nabla (\widehat v - \Pcalthat (\Pcalohat \widehat v)) \Vert_{0,\Ehat}  \\
&\stackrel{\ell \ge 1}{\lesssim }
\Vert \msf \ell (\widehat v - \Pcalthat(\Pcalohat \widehat v)) \Vert_{0,\Ehat}^2 +  \Vert \msf \ell^{-1} \widehat\nabla (\widehat v - \Pcalthat(\Pcalohat \widehat v)) \Vert_{0,\Ehat}^2.
\end{align*}
Scaling back to~$\E$, summing over all the elements, and using
  the locality of the operators $\Pcalo$ and $\Pcalt$, as well the shape regularity of the 
meshes to insert the factor $\msf^{-{1/2}}\ell$, we deduce
\[
\begin{split}
& \Vert \msf^{-{1/2}} \ell \big(v - \Pcalt(\Pcalo v) \big) \Vert_{0,\Gamma}^2 \\[2pt]
& \qquad \lesssim \sum_{\E \in\taun \text{ with }\overline \E \cap \partial \Omega \in \FhB}  \Vert \msf^{1/2} \ell^2 (\widehat v - \Pcalthat(\Pcalohat \widehat v)) \Vert_{0,\Ehat}^2 + \Vert \msf^{1/2}\widehat\nabla (\widehat v - \Pcalthat(\Pcalohat \widehat v)) \Vert_{0,\Ehat}^2 \\[2pt]
& \qquad \lesssim \sum_{\E \in \taun \text{ with } \overline \E \cap \partial \Omega \in \FhB} \Vert \msf^{-1} \ell^2 (v - \Pcalt (\Pcalo v)) \Vert_{0,\E}^2 + \Vert \nabla( v - \Pcalt(\Pcalo v)) \Vert_{0,\E}^2\\[2pt]
& \quad \overset{\eqref{Burman-Ern-3}, \eqref{Burman-Ern-2}}{\lesssim}  \Vert \nablah v \Vert_{0,\Omega}^2 +   \Vert \msf^{-1/2} \ell \llbracket v \rrbracket \Vert_{0, \FhI}^2,
\end{split}
\]
whence the assertion follows.

\section{Explicit error estimates} \label{appendix:proof-Corollary}

\begin{proof}[Proof of Corollary~\ref{corollary:hp-convergence}]
We start by noting that, for the special case $s = 1$, the arguments below show that 
Assumption~\ref{ass:approx} is valid with $\varepsilon = O(h/p)$. By
Theorem~\ref{theorem:abstract-analysis}, this fixes $\eta_0$.

To simplify the exposition, we restrict our attention to the case $p \ge s$.
The case $p < s$ is a pure $h$-version that is shown along similar lines.
We shall nevertheless write $\min(p,s) = s$ at the appropriate places.

By \cite[Lemma~{2.3}]{bernardi_optimal_interpolation}, for any $v \in H^{s+1}(\Omega)$,
Assumption~\ref{assumption:bernardi89} implies that the following estimate for the pull-back $\widehat{v}:= v|_K \circ \Phi_K$ is valid for all $K \in \taun$:
\begin{align}
\label{eq:chain-rule}
\|\widehat{v} \|_{s+1,\widehat{K}} \leq c h_K^{s+1-3/2} \|v\|_{s+1,K}. 
\end{align}
We also note that, for $j \in \{0,1\}$ and for each face $F$ of element $K$ with corresponding pull-back $\widehat{F}:= \Phi_K^{-1}(F)$, bounds \eqref{eq:SR} imply
\begin{align}
\label{eq:Khat-to-K}
  |\widehat v|_{j,\widehat{K}} \sim
  h_K^{j-3/2} |v|_{j,K}, 
\qquad 
  |\widehat v|_{0,\widehat{F}} \sim
  h_K^{-1} |v|_{0,F}, 
\qquad 
  |\widehat \nabla \widehat v|_{0,\widehat{F}} \sim
  |\nabla v|_{0,F}. 
\end{align}
Properties~\eqref{eq:Khat-to-K} allow for transferring approximation results on the reference element $\widehat{K}$
to the physical elements $K$ (``scaling argument''). 
The last preliminary ingredient are $p$-explicit approximation results on the reference element 
for which we refer, e.g., to \cite[Lemma~{B.3}, Thm.~{B.4}]{melenk-sauter10}. 
As in, e.g., \cite{MelenkParsaniaSauter_generalDGHelmoltz},
combining the polynomial approximation results on $\widehat{K}$ with~\eqref{eq:Khat-to-K} and~\eqref{eq:chain-rule} allows for showing that 
\begin{align}
\inf_{v_h \in S^{p,0}(\Omega,\taun)} \|u - v_h\|_{\DGp} \leq c h^{\min(p,s)} p^{-s+\frac{1}{2}} \|u\|_{s+1,\Omega}. 
\end{align}
For the approximation of $\uext$ and $\m$, we obviate the discussion of changes of variables 
in fractional Sobolev norms by resorting to appropriate liftings. For the approximation of $\uext$,
let $U^{ext} \in H^{s+1}(\Omega)$ be a lifting of $\uext$ with $\|U^{ext}\|_{s+1,\Omega} \lesssim 
\|\uext\|_{s+\frac{1}{2},\Gamma}$. Since the mesh $\taun$ is a regular mesh (see the discussion 
at the outset of Section~\ref{subsection:dGOmega}),
\cite[Thm.~{B.4}]{melenk-sauter10} provides an $H^1(\Omega)$-conforming approximation with optimal convergence properties: 
$$
\inf_{v_h \in S^{p,1}(\Omega,\taun)} \|U^{ext}  - v_h\|_{1,\Omega}
\leq c h^{\min(p,s)} p^{-s} \|U^{ext}\|_{s+1,\Omega} 
\leq c h^{\min(p,s)}p^{-s}  \|\uext\|_{s+\frac{1}{2},\Gamma}. 
$$
By taking the trace of $v_h$ on $\Gamma$, we obtain the desired approximation of $\uext$. Finally, for $\m$, let $M \in H^{s}(\Omega)$
be a lifting of $\m \in H^{s-\frac{1}{2}}(\Gamma)$ with $\|M\|_{H^s(\Omega)} \lesssim \|\m\|_{H^{s-\frac{1}{2}}(\Gamma)}$. 
Let $\m_h \in S^{p-1,0}(\Gamma,\Gammah)$ be the
$L^2(\Gamma)$-projection of $\m$ into $S^{p-1,0}(\Gamma,\Gammah)$. For each 
face $F \in \FhB$, denote by $K_{F} \in \taun$ the element that has $F$ as a face. 
Using approximation results on the reference element $\widehat{K}$ and the ``scaling arguments'' \eqref{eq:Khat-to-K} we get 
\begin{align}
\label{eq:approx-m}
\|\m - \mh\|_{0,F} \leq c h_K^{\min(p,s)-1/2} p^{-s+1/2}\|M\|_{s,K_{F}}. 
\end{align}
By summation over all faces $F \in \FhB$, we arrive at 
\[
\|\msf^{1/2} p^{-1}  (\m-\mh)\|_{0,\Gamma} 
\lesssim h^{\min(p,s)} p^{-s-1/2}\| (\m-\mh) \|_{s-\frac{1}{2},\Gamma}.
\]
The $H^{-\frac{1}{2}}(\Gamma)$-estimate is obtained by a standard duality argument using the orthogonality provided by 
the $L^2(\Gamma)$-projection: 
\begin{align}
\label{eq:approx-m-dual}
\|\m - \m_h\|_{-\frac{1}{2}, \Gamma} = 
\sup_{v \in H^{\frac{1}{2}}(\Gamma)} \frac{|\langle \m - \m_h,v\rangle|}{\|v\|_{\frac{1}{2},\Gamma} }
 = 
\sup_{v \in H^{\frac{1}{2}}(\Gamma)} \inf_{v_h \in S^{p-1,0}(\Gamma,\Gammah)} \frac{|\langle \m - \m_h,v - v_h\rangle| }{\|v\|_{\frac{1}{2},\Gamma} }. 
\end{align}
The infimum is estimated by taking $v_h$ as the
$L^2(\Gamma)$-projection of $v$ into $S^{p-1,0}(\Gamma,\Gammah)$. To estimate $v - v_h$, 
let $V \in H^1(\Omega)$ be a lifting of $v \in H^{\frac{1}{2}}(\Gamma)$ with $\|V\|_{1,\Omega} \lesssim \|v\|_{\frac{1}{2},\Gamma}$. 
By the same arguments as in~\eqref{eq:approx-m} (taking $s = 1$), we have 
\begin{align*}
\|v - v_h\|_{0,F} \leq c h_K^{\min(p,1)-1/2} p^{-1+1/2} \|V\|_{1,K_{F}}.
\end{align*}
Inserting this in~\eqref{eq:approx-m-dual} yields 
\begin{align*}
\|\m - \m_h\|_{-\frac{1}{2},\Gamma} 
& \lesssim \sup_{v \in H^{\frac{1}{2}}(\Gamma)} \frac{ \sum_{F \in \FhB} \|\m - \m_h\|_{0,F} \|v - v_h\|_{0,F} }{\|v\|_{\frac{1}{2},\Gamma} } \\
& \lesssim \sup_{v \in H^{\frac{1}{2}}(\Gamma)} \frac{1}{\|v\|_{\frac{1}{2},\Gamma}} \sum_{F \in \FhB}  p^{-s} h_K^{\min(p,s)-1/2+1-1/2} \|M\|_{s,K_{F}}\|V\|_{1,K_{F}} \\
&\lesssim h^{\min(p,s)} p^{-s} \|m\|_{s-\frac{1}{2},\Gamma} , 
\end{align*}
which completes the proof.
\end{proof}

\end{document}

%% file: create_h_plot.tex
\pgfplotsset{
  log x ticks with fixed point/.style={
      xticklabel={
        \pgfkeys{/pgf/fpu=true}
        \pgfmathparse{exp(\tick)}%
        \pgfmathprintnumber{\pgfmathresult}
        \pgfkeys{/pgf/fpu=false}
      }
  },
  log y ticks with fixed point/.style={
      yticklabel={
        \pgfkeys{/pgf/fpu=true}
        \pgfmathparse{exp(\tick)}%
        \pgfmathprintnumber[fixed relative, precision=6]{\pgfmathresult}
        \pgfkeys{/pgf/fpu=false}
      }
  }
}

\pgfset{
  foreach/parallel foreach/.style args={#1in#2via#3}{evaluate=#3 as #1 using {{#2}[#3-1]}},
}
\newcommand{\generatehplot}[7]{
\begin{subfigure}{0.45\textwidth}
  \begin{tikzpicture}
    \pgfplotsset{   
      cycle list/Set1,
      cycle multiindex* list={
        mark list\nextlist
        Set1\nextlist
      },
    }

  \begin{loglogaxis}[
    width=6.2cm, height=6cm,     
    grid = major,
    grid style={dashed, gray!30},
    axis background/.style={fill=white},
    xlabel=$h$,
    ylabel=#6,
    legend style={legend pos=south west},
    log basis x=2,
    xtick = {1,  0.25,  0.0625, 0.015625,0.0078125,0.00390625},
    legend pos= south east
    ]
    \def\thisk{#1}
    \def\thatk{#2}
    \def\datacol{#3}
    \def\expectedorders{#4}
    \def\orders{1,2,3}
    \ifthenelse{\thisk<9}{\def\myxmin{0.03}}{\def\myxmin{0.012}}

    \foreach \order in \orders{
    \addplot+[solid,mark size=2pt,mark options={line width=1.0pt}] table    
    [
    x expr=2^(-\thisrow{num_refines}),
    y expr=\thisrow{\datacol}*x^(#5),
    col sep=comma,
    restrict expr to domain={\thisrow{order}}{\order:\order},
    restrict expr to domain={\thisrow{k}}{\thisk:\thatk}
    ]{output_dg2.csv};
    }
    \foreach \expectedorder/\cc[count=\count] in \expectedorders
    {
      \edef\temp{
        \noexpand\addplot [black,dashed ] %
        expression [domain=0.5:\myxmin, samples=15] { \cc*x^(\expectedorder)} %
        node [pos=0.85,inner sep=0pt, anchor=south east] { {\noexpand \tiny $\noexpand\mathcal{O}(h^{\expectedorder})$}};}
        \temp
    };
     \legend{$p=1$,$p=2$, $p=3$}
  \end{loglogaxis}
\end{tikzpicture}
\end{subfigure}\;}

%% file: conv_h_k1_0.tex
%
%
%
\begin{tikzpicture}
    \pgfplotsset{   
      cycle list/Set1,
      cycle multiindex* list={
        mark list\nextlist
        Set1\nextlist
      },
    }
  \begin{loglogaxis}[
    width=6.2cm, height=6cm,     
    grid = major,
    grid style={dashed, gray!30},
    axis background/.style={fill=white},
    xlabel=$h$,
    ylabel=$\frac{\Vert u-\uh \Vert_{0,\Omega}}{\Vert u \Vert_{0,\Omega}}$,
    legend style={legend pos=south west},
    log basis x=2,
    xtick = {1,  0.25,  0.0625, 0.015625,0.0078125,0.00390625},
    legend pos= south east
    ]
%
  %
%
    \def\thisk{10}
    \def\thatk{11}
    \def\datacol{errltwo}
    \def\orders{1,2,3}
    \ifthenelse{\thisk<9}{\def\myxmin{0.03}}{\def\myxmin{0.012}}
    \foreach \order in \orders{
    \addplot+[solid,mark size=2pt,mark options={line width=1.0pt}] table    
    [
    x expr=2^(-\thisrow{num_refines}),
    y expr=\thisrow{\datacol}*x^(0),
    col sep=comma,
    restrict expr to domain={\thisrow{order}}{\order:\order},
    restrict expr to domain={\thisrow{k}}{\thisk:\thatk}
    ]{output_dg2.csv};
    }
    \def\expectedorder{2}
    \def\cc{100}   
    \addplot [black,dashed ] %
        expression [domain=0.5:\myxmin, samples=15] { \cc*x^(\expectedorder)} %
        node [pos=0.85,inner sep=0pt, anchor=south east] { { \tiny $\noexpand\mathcal{O}(h^{2})$}};
    \def\expectedorder{3}
    \def\cc{100}   
    \addplot [black,dashed ] %
        expression [domain=0.5:\myxmin, samples=15] { \cc*x^(\expectedorder)} %
        node [pos=0.85,inner sep=0pt, anchor=south east] { { \tiny $\noexpand\mathcal{O}(h^{3})$}};

    \def\expectedorder{4}
    \def\cc{10}   
    \addplot [black,dashed ] %
        expression [domain=0.5:\myxmin, samples=15] { \cc*x^(\expectedorder)} %
        node [pos=0.85,inner sep=0pt, anchor=south east] { { \tiny $\noexpand\mathcal{O}(h^{4})$}};

     \legend{$p=1$,$p=2$, $p=3$}
  \end{loglogaxis}
\end{tikzpicture}
%
%

%% file: conv_h_k1_1.tex
%
%
%
  \begin{tikzpicture}
    \pgfplotsset{   
      cycle list/Set1,
      cycle multiindex* list={
        mark list\nextlist
        Set1\nextlist
      },
    }
  \begin{loglogaxis}[
    width=6.2cm, height=6cm,     
    grid = major,
    grid style={dashed, gray!30},
    axis background/.style={fill=white},
    xlabel=$h$,
    ylabel=$\frac{\Vert \nabla u- \nablah\uh \Vert_{0,\Omega}}{\Vert \nabla u \Vert_{0,\Omega}}$,
    legend style={legend pos=south west},
    log basis x=2,
    xtick = {1,  0.25,  0.0625, 0.015625,0.0078125,0.00390625},
    legend pos= south east
    ]
%
  %
    \def\thisk{10}
    \def\thatk{11}
    \def\datacol{errhone}
    \def\orders{1,2,3}
    \ifthenelse{\thisk<9}{\def\myxmin{0.03}}{\def\myxmin{0.012}}
    \foreach \order in \orders{
    \addplot+[solid,mark size=2pt,mark options={line width=1.0pt}] table    
    [
    x expr=2^(-\thisrow{num_refines}),
    y expr=\thisrow{\datacol}*x^(0),
    col sep=comma,
    restrict expr to domain={\thisrow{order}}{\order:\order},
    restrict expr to domain={\thisrow{k}}{\thisk:\thatk}
    ]{output_dg2.csv};
    }
    \def\expectedorder{1}
    \def\cc{10}   
    \addplot [black,dashed ] %
        expression [domain=0.5:\myxmin, samples=15] { \cc*x^(\expectedorder)} %
        node [pos=0.85,inner sep=0pt, anchor=south east] { { \tiny $\noexpand\mathcal{O}(h^{1})$}};
    \def\expectedorder{2}
    \def\cc{20}   
    \addplot [black,dashed ] %
        expression [domain=0.5:\myxmin, samples=15] { \cc*x^(\expectedorder)} %
        node [pos=0.85,inner sep=0pt, anchor=south east] { { \tiny $\noexpand\mathcal{O}(h^{2})$}};

    \def\expectedorder{3}
    \def\cc{10}   
    \addplot [black,dashed ] %
        expression [domain=0.5:\myxmin, samples=15] { \cc*x^(\expectedorder)} %
        node [pos=0.85,inner sep=0pt, anchor=south east] { { \tiny $\noexpand\mathcal{O}(h^{3})$}};

     \legend{$p=1$,$p=2$, $p=3$}
  \end{loglogaxis}
\end{tikzpicture}
%
%

%% file: conv_h_k1_2.tex
%
%
%
  \begin{tikzpicture}
    \pgfplotsset{   
      cycle list/Set1,
      cycle multiindex* list={
        mark list\nextlist
        Set1\nextlist
      },
    }
  \begin{loglogaxis}[
    width=6.2cm, height=6cm,     
    grid = major,
    grid style={dashed, gray!30},
    axis background/.style={fill=white},
    xlabel=$h$,
    ylabel=$\frac{h^{1/2} \Vert \m-\mh \Vert_{0,\Gamma}}{\Vert \m \Vert_{0,\Gamma}}$,
    legend style={legend pos=south west},
    log basis x=2,
    xtick = {1,  0.25,  0.0625, 0.015625,0.0078125,0.00390625},
    legend pos= south east
    ]
%
  %
    \def\thisk{10}
    \def\thatk{11}
    \def\datacol{errm}
    \def\orders{1,2,3}
    \ifthenelse{\thisk<9}{\def\myxmin{0.03}}{\def\myxmin{0.012}}
    \foreach \order in \orders{
    \addplot+[solid,mark size=2pt,mark options={line width=1.0pt}] table    
    [
    x expr=2^(-\thisrow{num_refines}),
    y expr=\thisrow{\datacol}*x^(0.5),
    col sep=comma,
    restrict expr to domain={\thisrow{order}}{\order:\order},
    restrict expr to domain={\thisrow{k}}{\thisk:\thatk}
    ]{output_dg2.csv};
    }
    \def\expectedorder{1.5}
    \def\cc{10}   
    \addplot [black,dashed ] %
        expression [domain=0.5:\myxmin, samples=15] { \cc*x^(\expectedorder)} %
        node [pos=0.85,inner sep=0pt, anchor=south east] { { \tiny $\noexpand\mathcal{O}(h^{1.5})$}};
    \def\expectedorder{2.5}
    \def\cc{12}   
    \addplot [black,dashed ] %
        expression [domain=0.5:\myxmin, samples=15] { \cc*x^(\expectedorder)} %
        node [pos=0.85,inner sep=0pt, anchor=south east] { { \tiny $\noexpand\mathcal{O}(h^{2.5})$}};

    \def\expectedorder{3.5}
    \def\cc{10}   
    \addplot [black,dashed ] %
        expression [domain=0.5:\myxmin, samples=15] { \cc*x^(\expectedorder)} %
        node [pos=0.85,inner sep=0pt, anchor=south east] { { \tiny $\noexpand\mathcal{O}(h^{3.5})$}};

     \legend{$p=1$,$p=2$, $p=3$}
  \end{loglogaxis}
\end{tikzpicture}
%
%

%% file: conv_h_k1_3.tex
%
%
%
  \begin{tikzpicture}
    \pgfplotsset{   
      cycle list/Set1,
      cycle multiindex* list={
        mark list\nextlist
        Set1\nextlist
      },
    }
  \begin{loglogaxis}[
    width=6.2cm, height=6cm,     
    grid = major,
    grid style={dashed, gray!30},
    axis background/.style={fill=white},
    xlabel=$h$,
    ylabel=$\frac{h^{-1/2} \Vert \uext-\uhext \Vert_{0,\Gamma}}{\Vert \uext \Vert_{0,\Gamma}}$,
    legend style={legend pos=south west},
    log basis x=2,
    xtick = {1,  0.25,  0.0625, 0.015625,0.0078125,0.00390625},
    legend pos= south east
    ]
%
  %
    %
    \def\thisk{10}
    \def\thatk{11}
    \def\datacol{errphi}
    \def\orders{1,2,3}
    \ifthenelse{\thisk<9}{\def\myxmin{0.03}}{\def\myxmin{0.012}}
    \foreach \order in \orders{
    \addplot+[solid,mark size=2pt,mark options={line width=1.0pt}] table    
    [
    x expr=2^(-\thisrow{num_refines}),
    y expr=\thisrow{\datacol}*x^(-0.5),
    col sep=comma,
    restrict expr to domain={\thisrow{order}}{\order:\order},
    restrict expr to domain={\thisrow{k}}{\thisk:\thatk}
    ]{output_dg2.csv};
    }
    \def\expectedorder{1.5}
    \def\cc{20}   
    \addplot [black,dashed ] %
        expression [domain=0.5:\myxmin, samples=15] { \cc*x^(\expectedorder)} %
        node [pos=0.85,inner sep=0pt, anchor=south east] { { \tiny $\noexpand\mathcal{O}(h^{1.5})$}};
    \def\expectedorder{2.5}
    \def\cc{12}   
    \addplot [black,dashed ] %
        expression [domain=0.5:\myxmin, samples=15] { \cc*x^(\expectedorder)} %
        node [pos=0.85,inner sep=0pt, anchor=south east] { { \tiny $\noexpand\mathcal{O}(h^{2.5})$}};

    \def\expectedorder{3.5}
    \def\cc{10}   
    \addplot [black,dashed ] %
        expression [domain=0.5:\myxmin, samples=15] { \cc*x^(\expectedorder)} %
        node [pos=0.85,inner sep=0pt, anchor=south east] { { \tiny $\noexpand\mathcal{O}(h^{3.5})$}};

     \legend{$p=1$,$p=2$, $p=3$}
  \end{loglogaxis}
\end{tikzpicture}
%
%

%% file: conv_h_k2_0.tex
%
%
%
  \begin{tikzpicture}
    \pgfplotsset{   
      cycle list/Set1,
      cycle multiindex* list={
        mark list\nextlist
        Set1\nextlist
      },
    }
  \begin{loglogaxis}[
    width=6.2cm, height=6cm,     
    grid = major,
    grid style={dashed, gray!30},
    axis background/.style={fill=white},
    xlabel=$h$,
    ylabel=$\frac{\Vert u-\uh \Vert_{0,\Omega}}{\Vert u \Vert_{0,\Omega}}$,
    legend style={legend pos=south west},
    log basis x=2,
    xtick = {1,  0.25,  0.0625, 0.015625,0.0078125,0.00390625},
    legend pos= south east
    ]
%
  %
%
    \def\thisk{5}
    \def\thatk{6}
    \def\datacol{errltwo}
    \def\orders{1,2,3}
    \ifthenelse{\thisk<9}{\def\myxmin{0.03}}{\def\myxmin{0.012}}
    \foreach \order in \orders{
    \addplot+[solid,mark size=2pt,mark options={line width=1.0pt}] table    
    [
    x expr=2^(-\thisrow{num_refines}),
    y expr=\thisrow{\datacol}*x^(0),
    col sep=comma,
    restrict expr to domain={\thisrow{order}}{\order:\order},
    restrict expr to domain={\thisrow{k}}{\thisk:\thatk}
    ]{output_dg2.csv};
    }
    \def\expectedorder{2}
    \def\cc{30}   
    \addplot [black,dashed ] %
        expression [domain=0.5:\myxmin, samples=15] { \cc*x^(\expectedorder)} %
        node [pos=0.85,inner sep=0pt, anchor=south east] { { \tiny $\noexpand\mathcal{O}(h^{2})$}};
    \def\expectedorder{3}
    \def\cc{10}   
    \addplot [black,dashed ] %
        expression [domain=0.5:\myxmin, samples=15] { \cc*x^(\expectedorder)} %
        node [pos=0.85,inner sep=0pt, anchor=south east] { { \tiny $\noexpand\mathcal{O}(h^{3})$}};

    \def\expectedorder{4}
    \def\cc{10}   
    \addplot [black,dashed ] %
        expression [domain=0.5:\myxmin, samples=15] { \cc*x^(\expectedorder)} %
        node [pos=0.85,inner sep=0pt, anchor=south east] { { \tiny $\noexpand\mathcal{O}(h^{4})$}};

     \legend{$p=1$,$p=2$, $p=3$}
  \end{loglogaxis}
\end{tikzpicture}%
%
%

%% file: conv_h_k2_1.tex
%
%
%
  \begin{tikzpicture}
    \pgfplotsset{   
      cycle list/Set1,
      cycle multiindex* list={
        mark list\nextlist
        Set1\nextlist
      },
    }
  \begin{loglogaxis}[
    width=6.2cm, height=6cm,     
    grid = major,
    grid style={dashed, gray!30},
    axis background/.style={fill=white},
    xlabel=$h$,
    ylabel=$\frac{\Vert \nabla u- \nablah\uh \Vert_{0,\Omega}}{\Vert \nabla u \Vert_{0,\Omega}}$,
    legend style={legend pos=south west},
    log basis x=2,
    xtick = {1,  0.25,  0.0625, 0.015625,0.0078125,0.00390625},
    legend pos= south east
    ]
%
  %
    \def\thisk{5}
    \def\thatk{6}
    \def\datacol{errhone}
    \def\orders{1,2,3}
    \ifthenelse{\thisk<9}{\def\myxmin{0.03}}{\def\myxmin{0.012}}
    \foreach \order in \orders{
    \addplot+[solid,mark size=2pt,mark options={line width=1.0pt}] table    
    [
    x expr=2^(-\thisrow{num_refines}),
    y expr=\thisrow{\datacol}*x^(0),
    col sep=comma,
    restrict expr to domain={\thisrow{order}}{\order:\order},
    restrict expr to domain={\thisrow{k}}{\thisk:\thatk}
    ]{output_dg2.csv};
    }
    \def\expectedorder{1}
    \def\cc{4}   
    \addplot [black,dashed ] %
        expression [domain=0.5:\myxmin, samples=15] { \cc*x^(\expectedorder)} %
        node [pos=0.85,inner sep=0pt, anchor=south east] { { \tiny $\noexpand\mathcal{O}(h^{1})$}};
    \def\expectedorder{2}
    \def\cc{5}   
    \addplot [black,dashed ] %
        expression [domain=0.5:\myxmin, samples=15] { \cc*x^(\expectedorder)} %
        node [pos=0.85,inner sep=0pt, anchor=south east] { { \tiny $\noexpand\mathcal{O}(h^{2})$}};

    \def\expectedorder{3}
    \def\cc{6}   
    \addplot [black,dashed ] %
        expression [domain=0.5:\myxmin, samples=15] { \cc*x^(\expectedorder)} %
        node [pos=0.85,inner sep=0pt, anchor=south east] { { \tiny $\noexpand\mathcal{O}(h^{3})$}};

     \legend{$p=1$,$p=2$, $p=3$}
  \end{loglogaxis}
\end{tikzpicture}%
%
%

%% file: conv_h_k2_2.tex
%
%
%
\begin{tikzpicture}
    \pgfplotsset{   
      cycle list/Set1,
      cycle multiindex* list={
        mark list\nextlist
        Set1\nextlist
      },
    }
  \begin{loglogaxis}[
    width=6.2cm, height=6cm,     
    grid = major,
    grid style={dashed, gray!30},
    axis background/.style={fill=white},
    xlabel=$h$,
    ylabel=$\frac{h^{1/2} \Vert \m-\mh \Vert_{0,\Gamma}}{\Vert \m \Vert_{0,\Gamma}}$,
    legend style={legend pos=south west},
    log basis x=2,
    xtick = {1,  0.25,  0.0625, 0.015625,0.0078125,0.00390625},
    legend pos= south east
    ]
%
  %
    \def\thisk{5}
    \def\thatk{6}
    \def\datacol{errm}
    \def\orders{1,2,3}
    \ifthenelse{\thisk<9}{\def\myxmin{0.03}}{\def\myxmin{0.012}}
    \foreach \order in \orders{
    \addplot+[solid,mark size=2pt,mark options={line width=1.0pt}] table    
    [
    x expr=2^(-\thisrow{num_refines}),
    y expr=\thisrow{\datacol}*x^(0.5),
    col sep=comma,
    restrict expr to domain={\thisrow{order}}{\order:\order},
    restrict expr to domain={\thisrow{k}}{\thisk:\thatk}
    ]{output_dg2.csv};
    }
    \def\expectedorder{1.5}
    \def\cc{4}   
    \addplot [black,dashed ] %
        expression [domain=0.5:\myxmin, samples=15] { \cc*x^(\expectedorder)} %
        node [pos=0.85,inner sep=0pt, anchor=south east] { { \tiny $\noexpand\mathcal{O}(h^{1.5})$}};
    \def\expectedorder{2.5}
    \def\cc{1}   
    \addplot [black,dashed ] %
        expression [domain=0.5:\myxmin, samples=15] { \cc*x^(\expectedorder)} %
        node [pos=0.85,inner sep=0pt, anchor=south east] { { \tiny $\noexpand\mathcal{O}(h^{2.5})$}};

    \def\expectedorder{3.5}
    \def\cc{1}   
    \addplot [black,dashed ] %
        expression [domain=0.5:\myxmin, samples=15] { \cc*x^(\expectedorder)} %
        node [pos=0.85,inner sep=0pt, anchor=south east] { { \tiny $\noexpand\mathcal{O}(h^{3.5})$}};

     \legend{$p=1$,$p=2$, $p=3$}
  \end{loglogaxis}%
\end{tikzpicture}%
%
%
%

%% file: conv_h_k2_3.tex
%
%
%
\begin{tikzpicture}
    \pgfplotsset{   
      cycle list/Set1,
      cycle multiindex* list={
        mark list\nextlist
        Set1\nextlist
      },
    }
  \begin{loglogaxis}[
    width=6.2cm, height=6cm,     
    grid = major,
    grid style={dashed, gray!30},
    axis background/.style={fill=white},
    xlabel=$h$,
    ylabel=$\frac{h^{-1/2} \Vert \uext-\uhext \Vert_{0,\Gamma}}{\Vert \uext \Vert_{0,\Gamma}}$,
    legend style={legend pos=south west},
    log basis x=2,
    xtick = {1,  0.25,  0.0625, 0.015625,0.0078125,0.00390625},
    legend pos= south east
    ]
%
  %
    %
    \def\thisk{5}
    \def\thatk{6}
    \def\datacol{errphi}
    \def\orders{1,2,3}
    \ifthenelse{\thisk<9}{\def\myxmin{0.03}}{\def\myxmin{0.012}}
    \foreach \order in \orders{
    \addplot+[solid,mark size=2pt,mark options={line width=1.0pt}] table    
    [
    x expr=2^(-\thisrow{num_refines}),
    y expr=\thisrow{\datacol}*x^(-0.5),
    col sep=comma,
    restrict expr to domain={\thisrow{order}}{\order:\order},
    restrict expr to domain={\thisrow{k}}{\thisk:\thatk}
    ]{output_dg2.csv};
    }
    \def\expectedorder{1.5}
    \def\cc{6}   
    \addplot [black,dashed ] %
        expression [domain=0.5:\myxmin, samples=15] { \cc*x^(\expectedorder)} %
        node [pos=0.75,inner sep=0pt, anchor=south east] { { \tiny $\noexpand\mathcal{O}(h^{1.5})$}};
    \def\expectedorder{2.5}
    \def\cc{2}   
    \addplot [black,dashed ] %
        expression [domain=0.5:\myxmin, samples=15] { \cc*x^(\expectedorder)} %
        node [pos=0.85,inner sep=0pt, anchor=south east] { { \tiny $\noexpand\mathcal{O}(h^{2.5})$}};

    \def\expectedorder{3.5}
    \def\cc{1}   
    \addplot [black,dashed ] %
        expression [domain=0.5:\myxmin, samples=15] { \cc*x^(\expectedorder)} %
        node [pos=0.85,inner sep=0pt, anchor=south east] { { \tiny $\noexpand\mathcal{O}(h^{3.5})$}};

     \legend{$p=1$,$p=2$, $p=3$}
  \end{loglogaxis}
\end{tikzpicture}%
%
%
%